\documentclass[10pt,reqno]{amsart}
\usepackage{geometry, graphicx, url}
\usepackage{amsfonts,color,amsmath,amssymb,amsthm}

\newcommand{\F}{{\mathbb F}}
\newcommand{\tr}{\textup{Tr}}
\newcommand{\ra}{\rangle}
\newcommand{\la}{\langle}

\newcommand{\cB}{{\mathcal B}}
\newcommand{\cC}{{\mathcal C}}
\newcommand{\cD}{{\mathcal D}}

\newcommand{\cM}{{\mathcal M}}
\newcommand{\cN}{{\mathcal N}}
\newcommand{\cP}{{\mathcal P}}
\newcommand{\cQ}{{\mathcal Q}}
\newcommand{\cS}{{\mathcal S}}
\newcommand{\cU}{{\mathcal U}}

\newcommand{\Gal}{\textup{Gal}}
\newcommand{\Aut}{\textup{Aut}}

\newcommand{\lcm}{\textup{lcm}}
\newcommand{\End}{\textup{End}}

\newcommand{\GL}{\textup{GL}}
\newcommand{\GU}{\textup{GU}}
\newcommand{\SL}{\textup{SL}}
\newcommand{\Sp}{\textup{Sp}}
\newcommand{\SU}{\textup{SU}}
\newcommand{\GO}{\textup{GO}}
\newcommand{\PG}{\textup{PG}}

\newcommand{\PSL}{\textup{PSL}}
\newcommand{\PSU}{\textup{PSU}}
\newcommand{\PSp}{\textup{PSp}}
\newcommand{\PGO}{\textup{PGO}}

\newcommand{\PGU}{\textup{PGU}}
\newcommand{\GamL}{\Gamma\textup{L}}

\newcommand{\GamU}{\Gamma\textup{U}}

\newcommand{\PGamL}{\textup{P}\Gamma\textup{L}}

\newcommand{\POmeg}{\textup{P}\Omega}

\newtheorem{thm}{Theorem}

\newtheorem{lemma}[thm]{Lemma}

\newtheorem{proposition}[thm]{Proposition}

\numberwithin{equation}{section}
\numberwithin{thm}{section}
\numberwithin{table}{section}

\newtheorem{remark}[thm]{Remark}

\begin{document}

\title[]{On subgroups of finite classical groups with exactly two orbits on singular or isotropic points}

\author[Feng, Xiang]{Tao Feng, Qing Xiang}
\address{Tao Feng, School of Mathematical Sciences, Zhejiang University, 38 Zheda Road, Hangzhou 310027, Zhejiang, China}
\email{tfeng@zju.edu.cn}
\address{Qing Xiang, Department of Mathematics and SUSTech International Center for Mathematics, Southern University of Science and Technology, Shenzhen 518055, China} \email{xiangq@sustech.edu.cn}

\begin{abstract}
In this paper, we classify the  groups of semisimilarities of finite classical polar spaces with exactly two orbits on the singular or isotropic points. As a byproduct, we obtain many highly symmetric regular sets in the point graphs of finite classical polar spaces.
\end{abstract}

\maketitle

\tableofcontents

\section{Introduction}

The classification of finite primitive permutation groups of low rank has attracted much attention of group theorists due to applications in design theory, graph theory and finite geometry. The rank $2$ case corresponds to multiply transitive permutation groups, cf. \cite[Chapter~7]{DixonMortimer}. The classification in the rank $3$ case is accomplished in a series of papers by several researchers, and we refer the reader to \cite{kantor82rank3}, \cite{LiebeckRank3} and \cite{liebeckSaxlRank3} for the complete classification. The classification in the rank $4$ case seems incomplete, and we refer the reader to the PhD theses \cite{CuypRank4,vauRank4}. For partial results in the other low rank cases, we refer the reader to \cite{SolLowRank}, \cite{Foulser} and \cite{MuzySpiga}. The paper \cite{MuzySpiga} contains a nice exposition about the historical developments as well as the state of the art.

In this paper, we consider an analogous problem for finite classical polar spaces; that is, we would like to determine groups of semisimilarities that have few orbits on the singular or isotropic points.  Suppose that $\cP$ is a finite classical polar space of rank at least $2$, and let $H$ be a group of semisimilarities that has $t$ orbits on the points of $\cP$. We refer the reader to \cite{giudici2020subgroups} for the classification results in the $t=1$ case;  and we will assume that $t>1$ in the sequel. The {\it point graph} (or {\it collinearity graph}) of $\cP$ has the points of $\cP$ as vertices, and two distinct points are adjacent if and only if they are perpendicular. It is known that the point graph of $\cP$ is strongly regular. The $H$-orbits on the points of $\cP$ form an equitable partition of the vertices of the point graph. If further $t=2$, then each orbit is a regular set of the graph, cf. \cite{SRG}. A regular set of the point graph is also called an intriguing set in the literature. There are two kinds of intriguing sets, namely tight sets and $m$-ovoids, cf. \cite{BambergTightsets2007}. The concept of tight sets and $m$-ovoids were first introduced by Payne \cite{Payne} and Thas \cite{thasMovo} respectively for finite generalized quadrangles, and were later generalized to finite polar spaces in \cite{Drudge} and \cite{ShultThas}. They are important geometric objects which have close connection with strongly regular graphs and projective two-weight codes, cf. \cite{SRG,TwoWeight}. The notion of intriguing sets was introduced in \cite{BambergTightsets2007} and \cite{BambergCombIntriComb}  to unify the two concepts. In the present paper, we aim to classify the groups of semisimilarities that have exactly two orbits on the points of $\cP$. This is achieved in Theorems \ref{thm_red}, \ref{thm_TS} and \ref{thm_mOvoid} which we state below.

\begin{thm}\label{thm_red}
Let $V$ be a vector space of dimension $d$ over $\F_q$ equipped with a nondegenerate quadratic, unitary or symplectic form $\kappa$, so that the associated polar space $\cP$ has rank $r\ge 2$ and is not $Q^+(3,q)$. Suppose that $H$ is a reducible subgroup of semisimilarities that has two orbits $O_1$, $O_2$ on the points of $\cP$. Let $W$ be an $H$-invariant subspace of the lowest dimension. Then $O_1$ is the set of singular or isotropic points in $W^\perp$ and $O_2$ is its complement in $\cP$, and we have one of the following cases:
 \begin{enumerate}
\item[(R1)] $W$ is a maximal totally singular or isotropic subspace,
\item[(R2)] $V=W\oplus W^\perp$ with $\dim(W)=1$, and $H|_{W^\perp}$ is transitive on both $O_1$ and the isometry class of nonsingular points in $W^\perp$ of size $\frac{1}{2}|O_2|$,
\item[(R3)] $\cP=Q^{\pm}(d-1,q)$ with $d$ even, $W$ is elliptic of dimension $2$, and $H|_{W^\perp}$ is transitive on both $O_1$ and the set of all nonsingular points in $W^\perp$,
\item[(R4)] $\cP=Q^{\pm}(2m+1,q)$ with $q$ even, and either $\Sp_{2a}(b)'\unlhd H$ ($m=ab$) or $G_2(q^b)'\unlhd H$ ($m=3b$).
\end{enumerate}
\end{thm}

\begin{remark}
Take the same notation as in Theorem \ref{thm_red}. The orbits of the isometry group of $(V,\kappa)$ on the nonsingular points are referred to as isometry classes of nonsingular points when $\kappa$ is orthogonal or unitary. For case (R1) $O_1$ is a $1$-tight set, and we are not able to explicitly determine all such subgroups in this case. For (R2) the parameter of $O_1$ is given in \cite[Lemma~7]{BambergTightsets2007}, and the possible group structures of $H|_{W^\perp}$ will be determined in Section \ref{subsec_grp_lems}. For (R3) $O_1$ is a $(q^{r-1}+1)$-tight set of $\cP=Q^-(2r+1,q)$ if $d=2r+2$, and is a $\frac{q^{r-2}-1}{q-1}$-ovoid of $\cP=Q^+(2r-1,q)$ if $d=2r$. The possible group structures of $H|_{W^\perp}$ will be determined in Section \ref{subsec_grp_lems}. For (R4) the parameter of $O_1$ is also given in \cite[Lemma~7]{BambergTightsets2007}.
\end{remark}

\begin{thm}\label{thm_TS}
Let $V$ be a vector space over $\F_q$ of dimension $d$ equipped with a nondegenerate quadratic, unitary or symplectic form $\kappa$, so that the associated polar space $\cP$ has rank $r\ge 2$ and is not $Q^+(3,q)$. Suppose that $H$ is an irreducible subgroup of semisimilarities that has two orbits $O_1$, $O_2$ on the points of $\cP$ which are $i_1$- and $i_2$-tight sets respectively. Then either it occurs in Tables \ref{tab_TS_C3}-\ref{tab_TS_cSother}, or it is one of the following:
\begin{enumerate}
\item[(T1)]$H$ preserves a decomposition $\cD:\,V=V_1\oplus\cdots\oplus V_t$ with $\dim(V_i)=m$ for each $i$, either $t=d$ and we have the cases in Table \ref{tab_TS_C2}, or $t=2$ and we have one of: (a) $\cP=W(2m-1,q)$ and $V_1$, $V_2$ are nondegenerate, (b) $\cP=Q^+(2m-1,q)$ and $V_1$, $V_2$ are totally singular. For (b), either $\SL_m(q)\unlhd H_{V_1}^{V_1}$ , or $(q,m)=(2,4)$ and $A_7\unlhd H_{V_1}^{V_1}$, or $(q,m)\in\{(2,3),(8,3)\}$ and $H_{V_1}^{V_1}\unlhd\GamL_1(q^3)$.
\item[(T2)]$\cP=Q^+(2m-1,q)$, $\{i_1,i_2\}=\{q+1,q^{m-1}-q\}$, and  $H$ preserves a decomposition $(V,\kappa)=(V_1,f_1)\otimes (V_2,f_2)$, where $\dim(V_1)=2$, $\dim(V_2)=m\ge 4$ and $f_1,f_2$ are symplectic.
\item[(T3)]$\cP=H(2m-1,q)$, $O_1=W(2m-1,q^{1/2})$ which is a $(q^{1/2}+1)$-tight set, and $\Sp_{2m}(q^{1/2})'\unlhd H$.
\item[(T4)]$\cP=W(d-1,q)$ with $q$ even, $H$ stabilizes $Q^+(d-1,q)$ which is a $(q^{d/2-1}+1)$-tight set.
\end{enumerate}
\end{thm}

\begin{remark} 
In the third block of Table \ref{tab_C3Remain}, we list the Aschbacher class of $H$ as a subgroup of the semisimilarity group of $\cP'$. For Table \ref{tab_TS_cSLiedef}, Cohen and Cooperstein \cite{Cohen1988} showed that $F_4(q)$ has two orbits on the singular points of $Q(24,q)$ in its action on the minimal module, and Cardinali and De Bruyn \cite{card} constructed the tight sets that arise from spin modules.  The remaining two infinite families of tight sets in Table \ref{tab_TS_cSLiedef} seem to be new. 
\end{remark}
\begin{remark}
Below are some comments on the group structure of $H$ in Theorem \ref{thm_TS}.
\begin{enumerate}
\item[(1)] The group $H_{V_1}^{V_1}$ in (T1) is the induced action of the stabilizer of $V_1$ in $H$ on the subspace $V_1$. We refer the reader to \cite[Theorem~3.3]{pls_linear} for more details on the possible structures of $H$ for (T1-a).
\item[(2)]  We refer the reader to Remark \ref{rem_C4C7} for  the possible structures of $H$ for (T2).
\item[(3)] The group $H$ is transitive on both the set of singular points and the set of nonsingular points of $Q^+(d-1,q)$ for (T4), and we refer the reader to Section \ref{subsec_grp_lems} for the possible structures of $H$ for $d\ge 6$.
\end{enumerate}
\end{remark}

\begin{table}[h]
\centering\caption{Class $\cC_2$ for Theorems \ref{thm_TS}, \ref{thm_mOvoid}: $H$ stabilizes a decomposition $\cD:\,V=\la e_1\ra\oplus\cdots\oplus \la e_d\ra$ with each $e_i$ nonsingular. Here, t.s.=tight set.}\label{tab_TS_C2}
\begin{tabular}{|c|ccccccc|}\hline
$\cP$& $H(3,4)$  & $Q(4,3)$ &$Q(6,3)$ &$Q^+(7,3)$ & $H(4,4)$ & $Q(4,3)$  & $Q^-(5,3)$ \\ \hline
$O_1,O_2$ &$2,3$-ovoids&$2,2$-ovoids &$5,8$-ovoids &$8,32$-ovoids &$6,27$-t.s. & $4,6$-t.s. &$8,20$-t.s. \\   \hline
\end{tabular}
\end{table}

\begin{table}[h]
\centering
\caption{Class $\cC_3$ for Theorem \ref{thm_TS}: one $H$-orbit is $\cM_1$, cf. \eqref{eqn_M1def} } \label{tab_TS_C3}
\begin{tabular}{|cccc|}
\hline
$\cP$     & $\cP'$   & Condition   &$H$\\ \hline
$Q^+(2mb-1,q)$  & $Q^+(2m-1,q^b)$ & $b=2$, or & If $m\ge 3$,  then $\Omega^+_{2m}(q^b)\unlhd H$,\\
  &   & $(q,b)\in\{(2,3),(8,3)\}$ & or $\SU_{m}(q^b)\unlhd H$ with $m$ even,\\
  &   &                           & or $\textup{Spin}_7(q^b)\unlhd H$ with $m=4$.\\
$W(2m-1,q)$ & $H(m-1,q^2)$ & $m\ge 2$ even &  $\SU_m(q)\unlhd H$ or $(q,m)=(2,2)$\\
$H(3m-1,q)$ & $H(m-1,q^3)$ & $q\in\{2^2,2^6\}$, &  $\SU_m(q^{3/2})\unlhd H$,\\
  &   &   and $m\ge3$ odd & or $(m,q)=(3,4)$\\
$Q^+(2mb-1,q)$ & $H(2m-1,q^b)$ & --  &  $\SU_{2m}(q^b)\unlhd H$\\
$Q^-(2m-1,q)$ & $Q(m-1,q^2)$ & $m\ge 3$, $qm$ odd  & $\Omega_{m}(q^2)\unlhd H$, or\\
  &   &                           &$G_2(q^2)\unlhd H$ with $m=7$ \\ \hline
\end{tabular}
\end{table}

\begin{table}[h]
\centering \caption{Class $\cC_3$  for Theorem \ref{thm_TS}: the remaining cases, cf. Propositions \ref{prop_C3_O1peqPp}, \ref{prop_C3_OipProper_r'eq1} and  \ref{prop_C3_remain}  }\label{tab_C3Remain}
\begin{tabular}{|ccccc|}\hline
$\cP$      & $\cP'$   & $i_1,i_2$ & $H$ &Remark\\ \hline
$W(3,3)$ & $W(1,9)$  &   $5,5$ & $\SL_2(5)\unlhd H$ & \\  
$Q^-(5,3)$&$H(3,3^2)$& $14,14$ & $\PSL_2(7)\unlhd H$&  \\ \hline
$W(3,q)$ & $W(1,q^2)$ & $q+1,q^2-q$& $\SL_2(q)\unlhd H$& \\
         & & $q_1^{2}+1,q^2-q_1^{2}$& $\SL_2(q_1^{2})\unlhd H$&  $q=q_1^{3}$ even\\
$W(3,7)$ & $W(1,7^2)$ & $20,30$ & $2.A_5\unlhd H$ &  \\
$Q^-(5,5)$ & $H(2,5^2)$ & $54,72$ & $H\le N_\Gamma(3_+^{1+2})$ &  \\
 & &  $42,84$ & $\PSL_2(7)\unlhd H$&  \\
 & &  $36,90$ & $3.A_6\unlhd H$&  \\
$Q^-(5,q)$  & $H(2,q^2)$ & $3q+3,q^3-3q-2$ & $H\le N_{\Gamma}(C_{q+1}^3)$ & $q\in\{3,4,5,8,32\}$\\ \hline
$W(11,2)$  & $W(5,4)$ & $25,40$& $J_2\unlhd H$ &  $\cS$\\
$Q^+(7,3)$ & $H(3,9)$ & $7,21$ & $2.A_7$ & $\cS$\\
           &          &$12,16$ &$H\le N_\Gamma(4\circ 2^{1+4})$& $\cC_6$\\
$Q^-(9,2)$ & $H(4,4)$  & $6,27$ & $H\le \GO_2^-(2)\wr S_3$& $\cC_2$ \\
           &   & $11,22$ & $\PSL_2(11)\unlhd H$& $\cS$ \\
$Q^+(11,2)$ & $H(5,4)$ & $6,27$ & $3.\PSU_4(3)\unlhd H$& $\cS$ \\
       & & $11,22$ & $3.M_{22}\unlhd H$& $\cS$ \\
$Q^+(4m-1,q)$&$H(2m-1,q^2)$ & $q+1,q^{2m-1}-q$ & $\Sp_{2m}(q)'\unlhd H$& $\cC_5$\\
\hline
\end{tabular}
\end{table}

\begin{table}[h]
\centering\caption{Class $\cC_6$  for Theorem \ref{thm_TS}: $H$ normalizes $R$}\label{tab_TS_C6}
\begin{tabular}{|ccc|ccc|}
\hline
$\cP$      & $R$ & $i_1,i_2$  &  $\cP$  & $R$  & $i_1,i_2$    \\ \hline
$H(3,9)$   & $4\circ 2^{1+4}$ & $12,16$  &  $W(3,7)$   & $2_-^{1+4}$      & $10,40$  \\
$W(3,3)$   & $2_-^{1+4}$      & $\{2,8\}$ or $\{4,6\}$ &$W(7,3)$   & $2_-^{1+6}$      & $18,64$ \\
$W(3,5)$   & $2_-^{1+4}$      & $10,16$  &  $Q^+(7,3)$ & $2_+^{1+4}$      & $\{7,21\}$, $\{12,16\}$ or $\{4,24\}$\\
\hline
\end{tabular}
\end{table}

\begin{table}[h]
\centering\caption{Class $\cS$  for Theorem \ref{thm_TS}: simple groups of Lie type in the defining characteristic}\label{tab_TS_cSLiedef}
\begin{tabular}{|cccc|}
\hline
$\cP$       & $H^\infty$    & $i_1,i_2$   & Remark         \\ \hline
$Q(6,q)$ ($q=3^f$) & $\PSL_3(q)$  & $q+1, q^3-q$   & adjoint module\\
$Q^+(7,q)$ ($q$ square) & $\textup{Spin}_8^-(q^{1/2})$  & $q^{3/2}-1,q^3-q^{3/2}$ & spin module\\
$Q(12,q)$ ($q=3^f$)& $\PSp_6(q)$ & $q^2+1,q^6-q^2$ & section of exterior square \\
$Q^+(15,q)$ & $\textup{Spin}_9(q)$   & $q^3+1,q^7-q^3$   & spin module  \\
$Q(24,q)$ ($q=3^f$) & $F_4(q)$  & $q^4+1,q^{12}-q^{4}$    & minimal module\\
 \hline
\end{tabular}
\end{table}

\begin{table}[h]
\centering\caption{Class $\cS$  for Theorem \ref{thm_TS}: the remaining cases}\label{tab_TS_cSother}
\begin{tabular}{|ccc|ccc|}
\hline
$\cP$    & $H^\infty$  & $i_1,i_2$  &  $\cP$  & $H^\infty$   & $i_1,i_2$  \\ \hline
$W(3,5)$ & $2.A_6$ & $6,20$    & $Q^-(5,5)$ & $\PSp_4(3)$ & $36,90$\\
$W(3,7)$ & $2.A_7$  & $15,35$  & $Q^-(9,2)$ & $M_{11}$    & $11,22$\\
$W(5,4)$ & $J_2$ &  $25,40$  &            & $A_{11}$    & $11,22$\\
$W(7,2)$ & $A_{10}$ & $3,14$ & $Q^-(13,2)$& $G_2(3)$ & $12,117$\\
$W(11,3)$& $2.$Suz & $90,640$    & $Q^-(17,2)$& $A_{20}$    & $19,494$\\
         & $2.G_2(4)$ &$90,640$& $Q^+(23,2)$ & $Co_1$ & $24,2025$\\
$H(3,9)$ & $2.A_7$  & $7,21$   & $H(5,4)$ & $3.\PSU_4(3)$ & $6,27$\\
$H(4,4)$ & $\PSL_2(11)$&$11,22$&           & $3.M_{22}$    & $11,22$\\
\hline
\end{tabular}
\end{table}

\begin{thm}\label{thm_mOvoid}
Let $V$ be a vector space over $\F_q$ of dimension $d$ equipped with a nondegenerate quadratic, unitary or symplectic form $\kappa$, so that the associated polar space $\cP$ has rank $r\ge 2$ and is not $Q^+(3,q)$. Suppose that $H$ is an irreducible subgroup of semisimilarities that has two orbits $O_1$, $O_2$ on the points of $\cP$ which are $m_1$- and $m_2$-ovoids respectively. Then either it occurs in Tables \ref{tab_TS_C2}, \ref{tab_TrOvoid_C3}, \ref{tab_TrOvoid_C3Rem} and \ref{tab_TrOvoid_cS}, or it is one of the following:
\begin{enumerate}
\item[(M1)] $\cP=W(3,3)$, $H=2^.\POmeg^-_4(2)\le N_\Gamma(2_-^{1+4})$ and $m_1=m_2=2$,
\item[(M2)] $\cP=W(d-1,q)$ with $q$ even, and one $H$-orbit is an elliptic quadric $Q^-(d-1,q)$.
\end{enumerate}
\end{thm}

\begin{table}
\centering
\caption{Class $\cC_3$ for Theorem \ref{thm_mOvoid}: one $H$-orbit is $\cM_1$, cf. \eqref{eqn_M1def} }\label{tab_TrOvoid_C3}
\begin{tabular}{|cccc|}
\hline
$\cP$     & $\cP'$   & Condition   &$H$\\ \hline	
$H(3m-1,q)$ & $H(m-1,q^3)$  & $m\ge 3$ odd, $q\in\{2^2,2^6\}$ & $\SU_{m}(q^{3/2})\unlhd H$\\
$Q^-(4m-1,q)$ & $Q^-(2m-1,q^2)$ & $m\ge 2$ & $\Omega_{2m}^-(q^2)\unlhd H$, or\\
    &    &    & $\SU_{m}(q^4)\unlhd H$ with $m$ odd\\
$Q^-(6m-1,q)$ & $Q^-(2m-1,q^3)$ & $m\ge 2$, $q\in\{2,2^3\}$ &$\Omega_{2m}^-(q^3)\unlhd H$, or\\
    &    &    & $\SU_{m}(q^3)\unlhd H$ with $m$ odd \\
$Q^+(2m-1,q)$ &$Q(m-1,q^2)$ & $qm$ odd, $m\ge 3$& $\Omega_{m}(q^2)\unlhd H$, or\\
    &    &    & $G_2(q^2)\unlhd H$ with $m=7$\\
$W(2m-1,q)$ & $H(m-1,q^2)$ & $m\ge 3$ odd & $\SU_{m}(q)\unlhd H$, or\\
    &    &    & $(q,m)=(2,3)$\\ \hline
\end{tabular}
\end{table}

\begin{table}
\centering
\caption{Class $\cC_3$  for Theorem \ref{thm_mOvoid}: the remaining cases}\label{tab_TrOvoid_C3Rem}
\begin{tabular}{|cccc|}
\hline
$\cP$     & $\cP'$   &  $m_1,m_2$  &$H$\\ \hline	
$Q^-(5,5)$  & $H(2,5^2)$ & $3,3$ & $3.A_7\unlhd H$\\
$Q^+(7,2)$  & $H(3,2^2)$ & $6,9$ & $H\le(\GO_2^-(2)\wr S_4).2$\\
$W(3,q)$, $q\in\{2,4\}$ & $H(0,q^4)$ &  $1,q$ & $H\le \GamU_1(q^2)$\\
$W(4b-1,q)$, $q=2^f$ & $W(3,q^b)$ &  $\frac{q^b-1}{q-1}$, $\frac{q^b-1}{q-1}q^b$  &$\textup{Sz}(q^b)\unlhd H$\\ \hline
\end{tabular}
\end{table}

\begin{table}
\centering\caption{Class $\cS$  for Theorem \ref{thm_mOvoid}}\label{tab_TrOvoid_cS}
\begin{tabular}{|ccc|ccc|}
\hline
$\cP$ & $H^{\infty}$ &  $m_1,m_2$ &$\cP$ & $H^{\infty}$ &  $m_1,m_2$  \\ \hline	
$Q^-(5,3)$& $2.\PSL_3(4)$    & $2,2$ & $Q^+(13,2)$&$A_{16}$ & $28,99$ \\
$W(5,5)$  & $2.J_2$ &$15,16$ &  $Q^+(7,5)$& $2.\PSp_6(2)$ & $60,96$ \\
$Q(6,3)$  & $A_9$  & $3,10$ &  &$2.\POmeg_{8}^+(2)$&$60,96$  \\
   & $\PSp_6(2)$   & $1,12$ &  &$2.A_{10}$     & $36,120$ \\
$Q(6,5)$  & $\PSp_6(2)$   & $15,16$ & $W(3,q)$, $q=2^{f}$  & $\textup{Sz}(q)$& $1,q$  \\
$W(7,2)$  & $\PSL_2(17)$  & $6,9$ &$Q(6,q)$, $q=3^f$  & $\PSU_3(q)$  & $1,q+q^2$  \\
$Q^+(7,2)$ & $A_9$ & $1,14$ & $Q^+(7,q)$, $q=2^f$ &$\PSL_2(q^3)$& $1,q+q^{2}+q^{3}$ \\ \hline
\end{tabular}
\end{table}

 In \cite{KellyCons}, Kelly described a method to obtain intriguing sets of one polar space from those of another with a smaller rank via field reduction. The intriguing sets in high rank polar spaces that do not arise from field reduction are rare. The abundance of examples that show up in our classification indicates that there are highly symmetric intriguing sets in high rank polar spaces yet to be discovered.

Below are some comments on the proofs of Theorems \ref{thm_TS} and \ref{thm_mOvoid}. The symplectic case of  Theorem \ref{thm_TS} can also be deduced from \cite{LiebeckRank3}, and for this purpose we need to determine which subgroups in the main theorem of \cite{LiebeckRank3}  preserve a nondegenerate symplectic form. We shall do so only for subgroups of Aschbacher classes $\cS$ and $\cC_6$ in Propositions \ref{prop_classS_caseS} and \ref{prop_C6} respectively, and for the remaining geometric classes we will handle the symplectic case together with the orthogonal and unitary cases in a unified manner. Theorem \ref{thm_mOvoid} follows from \cite{FengLiTao} in which the $m$-ovoids of finite classical polar spaces that admit a transitive automorphism group that acts irreducibly on the ambient space are classified. The proof of \cite{FengLiTao} relies on the classification of subgroups of classical groups with a large prime divisor in  \cite{bamberg2008overgroups}, and the latter is based on \cite{Guralnick1999Linear}.

The paper is organized as follows. In Section 2 we first present some basic definitions and group theoretical lemmas, and then handle some specific subgroups of Aschbacher class $\cS$. Those subgroups arise in the proof of Theorem \ref{thm_TS}, and we handle them altogether in this section to streamline the proof of Theorem \ref{thm_TS}. In Section 3, we prove Theorem \ref{thm_red} which handles the case where the subgroup $H$ is reducible. We present the proof of Theorem \ref{thm_TS} in the next two sections which deal with the non-geometric and geometric subgroups respectively. In Section 6, we give a short proof of Theorem \ref{thm_mOvoid} based on the classification results in \cite{FengLiTao}.

\section{Preliminaries}\label{sec:Preli}

\subsection{Notation and basic definitions}\label{subsec_basic}

In this subsection, we introduce some basic definitions and the notation that we will use throughout this paper. We use the standard group theory notation as in \cite[Section~1.2]{bray2013maximal}. For a property $P$, we use the Iverson bracket notation $[\![P]\!]$ which takes value $1$ or $0$ according as $P$ holds or not.

\medskip

\noindent\textbf{The classical groups}\medskip

Let $q=p^f$ be a prime power with $p$ prime, and use $\F_q$ to denote the finite field of size $q$. Let $V$ be a vector space of dimension $d$ over $\F_q$ equipped with a nondegenerate reflexive sesquilinear form or quadratic form $\kappa$. We write $\GamL(V)$ for the group of invertible semilinear transformations of $V$, and write $\PG(V)$ for the projective geometry associated with $V$. Let $\cP$ be the associated polar space of $(V,\kappa)$ and write $r$ for its rank. We write $\perp$ for the polarity of $\cP$, and assume that $r\ge 2$ throughout this paper. We have $2r\le d\le 2r+2$, and $d=2r+2$ if and only if $\cP=Q^-(2r+1,q)$. An {\it ovoid}  of $\cP$ is a set of points which intersects each maximal totally singular or isotropic subspace in exactly one point. We denote the size of a putative ovoid by $\theta_r$, which we call the {\it ovoid number} of $\cP$. We have $\theta_r=q^{d-r-\epsilon}+1$, where $\epsilon=0$, $\frac{1}{2}$ or $1$ according as $\kappa$ is symplectic, unitary or orthogonal, cf. Table \ref{tab_thetar}. The number of points in $\cP$ equals $\frac{q^r-1}{q-1}\cdot\theta_r$.   We take the same notation for classical groups as in \cite{kleidman1990subgroup} except for that we regard the ambient space $V$ as a vector space over $\F_q$ also in the unitary case. If $\kappa$ is symplectic or unitary, its semisimilarity group $\Gamma(V,\kappa)$ is defined as follows:
\begin{equation}\label{eqn_Gam} \Gamma(V,\kappa):=\{g\in\GamL(V)\mid \kappa(xg,yg)=\lambda(g)\kappa(x,y)^{\sigma(g)}\textup{ for all }x\in V\},
\end{equation}
where $\sigma(g)\in\Gal(\F_{q}/\F_p)$, $\lambda(g)\in\F_{q}^{*}$. If $\kappa$ is a quadratic form, the group $\Gamma(V,\kappa)$ is defined similarly. We refer the reader to \cite[p.~13]{kleidman1990subgroup} for the definitions of its subgroups $\Delta(V,\kappa)$, $I(V,\kappa)$, $S(V,\kappa)$ and $\Omega(V,\kappa)$. We shall write $\Gamma(V)$ or $\Gamma$ instead of $\Gamma(V,\kappa)$ if $(V,\kappa)$ is clear from the context. We put the letter P in front of those classical groups for the respective quotient groups modulo their intersections with the center of $\GL(V)$.
\begin{table}[h]
\centering\caption{The finite classical polar spaces of rank $r$ and their ovoid numbers $\theta_r$}	 \label{tab_thetar}
\begin{tabular}{|c|cccccc|}\hline
$\cP$  &$W(2r-1,q)$ & $Q^+(2r-1,q)$ & $Q(2r,q)$& $Q^-(2r+1,q)$ & $H(2r-1,q)$ & $H(2r,q)$\\
$\theta_{r}$ & $q^r+1$ & $q^{r-1}+1$ & $q^{r}+1$ & $q^{r+1}+1$ & $q^{r-1/2}+1$ & $q^{r+1/2}+1$\\ \hline
	\end{tabular}
\end{table}	
\medskip

\noindent\textbf{Intriguing sets}\medskip

Take $(V,\kappa)$ as above, and write $d=\dim(V)$. Let $r$ be the rank of the associated polar space $\cP$, and write $\theta_r$ for the ovoid number of $\cP$ as in Table \ref{tab_thetar}.  Let $\cM$ be  a nonempty subset of points of $\cP$. We say that $\cM$ is an \textit{intriguing set} if there exist some constants $h_1,h_2$ such that for each point $P$ of $\cP$ we have $|P^{\perp}\cap\cM|=h_{1}$ or $h_{2}$ according as $P\in\cM$ or not. By \cite[Section~4]{BambergTightsets2007}, there are exactly two types of intriguing sets:
\begin{enumerate}
\item[(1)]$i$-tight sets: $\cM$ has size $i\cdot\frac{q^r-1}{q-1}$ for some integer $i$, and $h_1=q^{r-1}+i\cdot\frac{q^{r-1}-1}{q-1}$, $h_2=i\cdot\frac{q^{r-1}-1}{q-1}$;
\item[(2)]$m$-ovoids: $\cM$ has size $m\theta_r$ for some integer $m$, and $h_1=(m-1)\theta_{r-1}+1$, $h_2=m\theta_{r-1}$.
\end{enumerate}
We call $i$ and $m$ the parameters of $\cM$ respectively. The set of all the singular points of $\cP$ is both a $\theta_r$-tight set and a $\frac{q^r-1}{q-1}$-ovoid, which is the trivial intriguing set. We refer the reader to \cite[Chapter~2]{SRG} and the references therein for the known constructions of intriguing sets in various classical polar spaces.

If a subgroup $H_0$ of $\textup{P}\Gamma(V,\kappa)$ has two orbits $O_1$, $O_2$ on the points of $\cP$, then each $H_0$-orbit forms an intriguing set. Suppose further that the two $H_0$-orbits are $i_1$- and $i_2$-tight sets respectively. Then $i_1+i_2=\theta_r$, and $\frac{q^r-1}{q-1}$ divides the sizes of $O_1$ and $O_2$. Since each $|O_i|$ divides the order of $H_0$, we deduce that $\textup{lcm}(|O_1|,|O_2|)$ divides $|H_0|$. In particular, we have the trivial bound $|H_0|\ge\frac{1}{2}|\cP|$. Recall that $|\cP|=(q^{d-\epsilon-r}+1)\frac{q^r-1}{q-1}$, where $\epsilon=0$, $\frac{1}{2}$ or $1$ according as $\kappa$ is symplectic, unitary or orthogonal. It follows that $|\cP|>q^{d-1-\epsilon}$, and so $d<1+\epsilon+\log_q(2\cdot|H_0|)$. We summarize these observations as a lemma for later reference.
\begin{lemma}\label{lem_ParaCond}
Let $\cP$ be a classical polar space of rank $r\ge 2$ associated with $(V,\kappa)$, and
set $\epsilon:=0$, $\frac{1}{2}$ or $1$ according as $\cP$ is symplectic, unitary or orthogonal. Suppose that $H_0$ is a group of semisimilarities that has two orbits $O_1$, $O_2$ on the points of $\cP$ which are respectively $i_1$- and $i_2$-tight sets. Then  
$$\dim(V)<1+\epsilon+\log_q(2\cdot|H_0|),$$ 
and $\frac{q^r-1}{q-1}$ divides $\gcd(|O_1|,|O_2|)$, $\lcm(|O_1|,|O_2|)$ divides $|H_0|$.
\end{lemma}
\medskip

\noindent\textbf{The field reduction}\medskip

Let $b$ be a  divisor of an integer $d$ with $b>1$. Suppose that $V'$ is a $d/b$-dimensional vector space over the finite field $\F_{q^b}$, equipped with a non-degenerate reflexive sesquilinear or quadratic form $\kappa'$. Let $\cP'$ be the polar space associated with $(V',\kappa')$. Let $\tr_{\F_{q^b}/\F_q}$ be the relative trace function from the field $\F_{q^b}$ to $\F_{q}$.  We regard $V'$ as a $d$-dimensional vector space $V$ over $\F_q$. By composing $\kappa'$ and the relative trace function, we obtain a new form $\kappa$ on $V$. Let $\cP$ be the polar space defined by $(V,\kappa)$. Such a process is called {\it field reduction}. In Table \ref{tab_extfieldP'gt0}, we list all the cases that can be obtained by iteratively composing the field reductions in \cite[Table~4.3A]{kleidman1990subgroup}. We refer the reader to \cite{gill2006polar} for restrictions on $\alpha$ in rows 11, 12 of Table \ref{tab_extfieldP'gt0}. The table also lists the sizes of $\cP$ and $\cP'$ for reference. The following group
\begin{equation}\label{eqn_GamJ}
 \Gamma^\#(V',\kappa')=\{g\in\Gamma(V,\kappa)\mid \alpha^{1-\sigma(g)}\lambda(g)\in\F_q^*\}
\end{equation}
is a subgroup of $\Gamma(V,\kappa)$ of Aschbacher class $\cC_5$, where we set $\alpha:=1$ by default if there is no parameter $\alpha$ in the corresponding row in Table \ref{tab_extfieldP'gt0}. For later use, we shall need the following subset of $\cP$:
\begin{equation}\label{eqn_M1def}
   \cM_1=\{\la \eta v\ra_{\F_q}\mid \eta\in\F_{q^b}^*,\,\la v\ra_{\F_{q^b}}\in \cP'\}.
\end{equation}

\begin{table}[h]
\centering\caption{The field reductions, where $\alpha+\alpha^{q^{b/2}}=0$ for rows 7, 8.}\label{tab_extfieldP'gt0}
\scalebox{0.9}{
		\begin{tabular}{|ccccccc|}
			\hline
			Case & $\cP'$ & $|\cP'|$ & $\cP$ & $|\cP|$ &$\kappa$& Condition\\ \hline
			1 &$W(d/b-1,q^b)$ & $\frac{q^d-1}{q^b-1}$ & $W(d-1,q)$ & $\frac{q^d-1}{q-1}$ & $\tr_{q^b/q}\circ\kappa'$&$d/b$ even\\
			2 & $Q^{+}(d/b-1,q^b)$ & $(q^{d/2-b}+1)\frac{q^{d/2}-1}{q^b-1}$ & $Q^{+}(d-1,q)$ & $(q^{d/2-1}+1)\frac{q^{d/2}-1}{q-1}$ & $\tr_{q^b/q}\circ\kappa'$& $d/b$ even \\
			3 & $Q^{-}(d/b-1,q^b)$ & $(q^{d/2}+1)\frac{q^{d/2-b}-1}{q^b-1}$ & $Q^{-}(d-1,q)$ & $(q^{d/2}+1)\frac{q^{d/2-1}-1}{q-1}$ & $\tr_{q^b/q}\circ\kappa'$& $d/b$ even \\
			4 & $Q(d/b-1,q^b)$ & $\frac{q^{d-b}-1}{q^b-1}$ & $Q(d-1,q)$ & $\frac{q^{d-1}-1}{q-1}$ & $\tr_{q^b/q}\circ\kappa'$& $dq$ odd \\
			5 & $H(d/b-1,q^b)$ & $(q^{d/2}+1)\frac{q^{(d-b)/2}-1}{q^b-1}$ & $H(d-1,q)$ & $(q^{d/2}+1)\frac{q^{(d-1)/2}-1}{q-1}$ & $\tr_{q^b/q}\circ\kappa'$& $d$ odd \\
			6 & $H(d/b-1,q^b)$ & $(q^{(d-b)/2}+1)\frac{q^{d/2}-1}{q^b-1}$ & $H(d-1,q)$ & $(q^{(d-1)/2}+1)\frac{q^{d/2}-1}{q-1}$ & $\tr_{q^b/q}\circ\kappa'$& $d/b$ even, $b$ odd \\
			7 & $H(d/b-1,q^b)$ & $(q^{d/2}+1)\frac{q^{(d-b)/2}-1}{q^b-1}$ & $W(d-1,q)$ & $\frac{q^{d}-1}{q-1}$ & $\tr_{q^b/q}\circ\alpha\kappa'$& $d/b$ odd, $b$ even\\
			8 & $H(d/b-1,q^b)$ & $(q^{(d-b)/2}+1)\frac{q^{d/2}-1}{q^b-1}$ & $W(d-1,q)$ & $\frac{q^{d}-1}{q-1}$ & $\tr_{q^b/q}\circ\alpha\kappa'$& $d/b$ even, $b$ even\\
			9 & $H(d/b-1,q^b)$ & $(q^{d/2}+1)\frac{q^{(d-b)/2}-1}{q^b-1}$ & $Q^{-}(d-1,q)$ & $(q^{d/2}+1)\frac{q^{d/2-1}-1}{q-1}$ & $\tr_{q^{b/2}/q}\circ\kappa'(-,-)$&$d/b$ odd, $b$ even \\
			10& $H(d/b-1,q^b)$ & $(q^{(d-b)/2}+1)\frac{q^{d/2}-1}{q^b-1}$ & $Q^{+}(d-1,q)$ & $(q^{d/2-1}+1)\frac{q^{d/2}-1}{q-1}$ & $\tr_{q^{b/2}/q}\circ\kappa'(-,-)$&$d/b$ even, $b$ even \\
			11 & $Q(d/b-1,q^b)$ & $\frac{q^{d-b}-1}{q^b-1}$ & $Q^{+}(d-1,q)$ & $(q^{d/2-1}+1)\frac{q^{d/2}-1}{q-1}$ & $\tr_{q^b/q}\circ\alpha\kappa'$&$qd/b$ odd \\
			12 & $Q(d/b-1,q^b)$ & $\frac{q^{d-b}-1}{q^b-1}$ & $Q^{-}(d-1,q)$ & $(q^{d/2}+1)\frac{q^{d/2-1}-1}{q-1}$ & $\tr_{q^b/q}\circ\alpha\kappa'$&$qd/b$ odd \\
			\hline
	\end{tabular}}
\end{table}\medskip

\noindent\textbf{The weak equivalence of representations}\medskip

Let $\rho_1:\,G\rightarrow\GL_n(q)$ be the representation of a finite group $G$ that is afforded by a $G$-module $V_1$ with respect to a chosen basis. Take a field automorphism $\theta$ of $\F_q$ and a group automorphism $\alpha\in\Aut(G)$. For a vector $v$ and a matrix $m$, we write $v^\theta$ and $m^\theta$ for the vector and the matrix obtained by applying $\theta$ to their entries respectively. We define $V_1^\theta$ as the $G$-module with the same ambient space as $V_1$ and $G$-action $v. g=v(g\rho_1)^\theta$, and let $\rho_1^\theta$ for the corresponding representation. Similarly, let ${}^\alpha V_1$ be the $G$-module with the same ambient space as $V_1$ and $G$-action $v. g=v(g^\alpha)\rho_1$, and write ${}^\alpha\rho_1$ for the corresponding representation. Let $V_2$ be another $G$-module.  If $V_2$ is equivalent to ${}^\alpha V_1$ for some $\alpha\in\Aut(G)$, then we say that $V_1$ and $V_2$ are {\it quasi-equivalent}. If $V_2$ is equivalent to $V_1^\theta$ for some field automorphism $\theta$, then we say that $V_1$ and $V_2$ are {\it algebraically conjugate}. If $V_2$ is equivalent to one of the  $G$-modules obtained from $V_1$ by recursive applications of group automorphisms, field automorphisms and duality, then we say that $V_1$ is {\it weakly equivalent} to $V_2$.

Suppose that two $G$-modules $V_1$, $V_2$ are weakly equivalent and $V_1$ has a nondegenerate $G$-invariant reflexive sesquilinear or quadratic form $\kappa_1$. It is routine to show that there is a corresponding nondegenerate $G$-invariant form $\kappa_2$ on $V_2$ such that $(V_1,\kappa_1)$ and $(V_2,\kappa_2)$ define the same classical polar space $\cP$, and $G$ has the same orbit structure on the points of $\cP$. Let $\rho_i:\,G\rightarrow\GL(V_i)$ be the respective representations for $i=1,~2$. Suppose further that $G\rho_1$ is a subgroup of $\Gamma(V_1,\kappa_1)$ of Aschbacher class $\cS$ that has two orbits  on the singular points of $\cP$. Then $G\rho_2$ is also a subgroup of $\Gamma(V_2,\kappa_2)$ of Aschbacher class $\cS$ with the same properties. \medskip

\noindent\textbf{Primitive prime divisors}\medskip

Let $n$, $k$ be positive integers. A prime divisor $p$ of $n^k-1$ is called \textit{primitive} if $p$ divides none of $n^i-1$, $1\le i\le k-1$. For a primitive prime divisor $p$ of $n^k-1$, we have $\textup{ord}_p(n)=k$, where we write $\textup{ord}_p(n)$ for the multiplicative order of $n$ modulo $p$. If $k=2$, then each odd prime divisor of $n+1$ is a primitive divisor of $n^2-1$ by the fact that $\gcd(n+1,n-1)$ is at most $2$. It follows that $n^2-1$ has no primitive prime divisor only if $n+1$ is a power of $2$. For $k\ge 3$, we have the following classical result.
\begin{lemma}\cite{zsigmondy1892theorie}\label{lem_zsigmondy}
	Let $n$ and $k$ be positive integers such that $n>1$, $k\geq 3$ and $(n,k)\neq(2,6)$. Then $n^k-1$ has at least one primitive prime divisor.
\end{lemma}

\subsection{Some group theoretic lemmas}\label{subsec_grp_lems}

Suppose that $V$ is a vector space of dimension $d$ over $\F_q$ equipped with a nondegenerate form $\kappa$, and write $\cP$ for the associated polar space. We call an $I(V,\kappa)$-orbit on the nonsingular points an isometry class of nonsingular points. There are at most two such classes of nonsingular points, and there is exactly one class only if either $\kappa$ is unitary, or $\kappa$ is orthogonal and $q$ is even. When $q$ is odd, let $\eta$ be the quadratic character of $\F_q^*$ such that $\eta(a)=1$ or $-1$ according as $a$ is a square or nonsquare.  If $\kappa$ is orthogonal and $q$ is odd, then each class of nonsingular points has the same value $\eta(\kappa(v))$ for its element $\la v\ra$. If $d$ is odd, then the two classes have different sizes. If $d$ is even, the two classes have the same sizes and are swapped by a similarity. In \cite{giudici2020subgroups} the authors classified all groups of semisimilarities that act transitively on the set of all subspaces of a given isometry type, where the subspaces are either nondegenerate or totally singular. We shall need the classification of subgroups of $\Gamma(V,\kappa)$ that are transitive on both the set of singular points and one isometry class of nonsingular points for the proofs of our main theorems. We determine such subgroups of the unitary and orthogonal groups in the next few lemmas.

\begin{lemma}\label{lem_H_TrSinNonsin}
Suppose that $V$ is a vector space over $\F_{q}$ with a nondegenerate Hermitian form $\kappa$, where $q$ is a square and $d=\dim(V)\ge 2$. Let $H$ be a group of semisimilarities that is transitive on both the singular points and the nonsingular points. Then either $\SU_d(q^{1/2})\le H$, or $(d,q)=(3,2^2)$ and $H=3_+^{1+2}\rtimes C_8$, or $(d,q)=(2,2^2)$.
\end{lemma}
\begin{proof}
The group $\SU_d(q^{1/2})$ is transitive on both the singular points and the nonsingular points by \cite[Lemma~2.10.5]{kleidman1990subgroup}. Assume that $H$ does not contain $\SU_d(q^{1/2})$. Let $n$ be the least common multiple of the numbers of singular and nonsingular points. Then $n$ divides $|H|$.  First suppose that $d\ge 3$. Since $H$ is transitive on the singular points, by \cite[Theorem~4.1 (b)]{giudici2020subgroups} $(d,q)$ is one of $(3,2^2)$, $(3,3^2)$, $(3,5^2)$, $(3,8^2)$, $(4,3^2)$ and $(9,2^2)$. Since $H$ is transitive on the nonsingular points, either $(d,q)=(3,2^2)$ or $d$ is even by \cite[Theorem~4.1 (a)]{giudici2020subgroups}. Therefore, $(d,q)$ is either $(3,2^2)$ or $(4,3^2)$. For $(d,q)=(3,2^2)$, we have $H=3_+^{1+2}\rtimes C_8$; for $(d,q)=(4,3^2)$, there is no such subgroup $H$ by examining all the subgroups whose orders are multiples of $n$ by Magma \cite{Magma}. Next suppose that $d=2$. By \cite[Theorem~3.7 (b)]{giudici2020subgroups}, we deduce that $q=2^2$ or $16^2$. For $q=16^2$, the groups of semisimilarities whose orders are multiples of $n$ all contain $\SU_{2}(16)$. This completes the proof.
\end{proof}
\begin{lemma}\label{lem_Q(dodd,q)sub}
Suppose that $(V,\kappa)$ is a quadratic space of dimension $d\ge 3$ over $\F_q$, where $qd$ is odd. Let $H$ be a group of semisimilarities that is transitive on both the set of singular points and one isometry class of nonsingular points. Then either $\Omega_{d}(q)\unlhd H$, or $d=7$ and $G_2(q)\unlhd H$. Moreover, $H^\infty$ is transitive on both isometry classes of nonsingular points in both cases.
\end{lemma}
\begin{proof}
The group $\Omega_d(q)$ is transitive on both the set of singular points and each isometry class of nonsingular points by \cite[Lemma~2.10.5]{kleidman1990subgroup}. Suppose that $H$ does not contain $\Omega_d(q)$. By \cite[Theorem~3.3]{giudici2020subgroups} $H$ is transitive on one class of nonsingular points only if it fixes a singular point for $d=3$, so $d\ne 3$. By \cite[Theorem~5.3]{giudici2020subgroups} $H$ is transitive on one class of nonsingular points only if it fixes a totally singular line for $d=5$, so $d\ne 5$. By \cite[Theorem~7.1]{giudici2020subgroups}, $H$ is transitive on the set of singular points only if $(d,H^\infty)=(7,G_2(q))$ for $d\ge 7$. There is a unique conjugacy class of $G_2(q)$ in the isometry group by \cite[Table~8.40]{bray2013maximal}, and it is transitive on both classes of nonsingular points by \cite[Theorem~7.1]{giudici2020subgroups}. This completes the proof.
\end{proof}

\begin{lemma}\label{lem_Q-(2m,q)sub}
Suppose that $(V,\kappa)$ is an elliptic quadratic space  and $\dim(V)=2m\ge 4$. Let $H$ be a subgroup of semisimilarities that is transitive on both the set of singular points and one isometry class of nonsingular points. Then either $\Omega^-_{2m}(q)\unlhd H$, or one of the following occurs:
\begin{enumerate}
\item[(1)] $m$ is odd, and $\SU_m(q)\unlhd H$,
\item[(2)] $m=2$, $q=2$ and $H=\GO_2^-(4).2$, or $q=4$ and $\GO_2^-(4^2).4\unlhd H$,
\item[(3)] $m=3$, $q=2$ and $3_+^{1+2}\unlhd H$, or $q=3$ and $2.\PSL_3(4)\unlhd H$.
\end{enumerate}
For $(m,q)=(3,3)$, the normalizer of $2.\PSL_3(4)$ in $\Delta(V)$ is transitive on all nonsingular points.
\end{lemma}
\begin{proof}
If either $\Omega_{2m}^-(q)\le H$, or $m$ is odd and $\SU_m(q)\unlhd H$, then $H$ transitive on both the set of singular points and one class of nonsingular points by \cite[Lemma~2.10.5]{kleidman1990subgroup}. Suppose that $H$ is in neither of the above two cases. By examining the $(m,q)$ pairs that admit transitive $H$-action on the singular points and one class of nonsingular points in \cite[Theorems~3.4,~4.2,~6.2]{giudici2020subgroups}, we deduce that $(m,q)$ is one of $(2,3)$, $(2,2)$, $(2,4)$, $(3,2)$ and $(3,3)$. For each $(m,q)$ pair, we examine the orbits of all groups of semisimilarities whose orders are multiples of the least common multiple of the numbers of singular points and nonsingular points in one class by Magma \cite{Magma}.
It turns out that $H=\GO_2^-(4).2$ for $(m,q)=(2,2)$, $\GO_2^-(4^2).4\unlhd H$ for $(m,q)=(2,4)$, $3_+^{1+2}\unlhd H$ for $(m,q)=(3,2)$, $2.\PSL_3(4)\unlhd H$ for $(m,q)=(3,2)$, and there is no such subgroup $H$ for $(m,q)=(2,3)$. The last claim is verified by Magma. This completes the proof.
\end{proof}

\begin{lemma}\label{lem_Q+(2m,q)sub}
Suppose that $(V,\kappa)$ is a hyperbolic quadratic space and $\dim(V)=2m\ge 6$. Let $H$ be a subgroup of semisimilarities that is transitive on both the set of singular points and one isometry class of nonsingular points. Then either $\Omega^+_{2m}(q)\unlhd H$, or one of the following occurs:
\begin{enumerate}
\item[(1)] $m$ is even, and $H^{(\infty)}=\textup{SU}_{m}(q)$,
\item[(2)] $m=4$, $H^{(\infty)}=\textup{Spin}_{7}(q)$ which is irreducible on $V$,
\item[(3)] $(m,q)=(4,2)$, and $H^{\infty}=A_9$.
\end{enumerate}
Moreover, the normalizer of $H^\infty$ in $\Gamma(V)$ is transitive on the set of all nonsingular points in each case.
\end{lemma}
\begin{proof}
If $H$ contains either $\Omega^+_{2m}(q)$ or $\SU_m(q)$ ($m$ even), then it is transitive on both the set of singular points and each class of nonsingular points by \cite[Lemma~2.10.5]{kleidman1990subgroup}. Assume that $H$ is in neither of the two cases. The group $H$ is transitive on the singular points, so by \cite[Theorem~8.4]{giudici2020subgroups} we have the following possibilities: (a) $m=4$ and $H^{\infty}=\textup{Spin}_{7}(q)$ which is irreducible on $V$, (b) $(m,q)$ is one of $(3,2)$, $(4,2)$, $(4,3)$. For case (a), there is a unique conjugacy class of irreducible $\textup{Spin}_7(q)$ in $\Gamma(V,\kappa)$ by \cite[Table~8.50]{bray2013maximal} and \cite{KleidmanThemaximal8}. The claims on transitivity for this case will be established in Proposition \ref{prop_spin7} below. Suppose that $H$ is in case (b) and is not covered by (a). If $(m,q)=(3,2)$, then we check by Magma \cite{Magma} that $H$ is transitive on the singular points only if $H^\infty=A_7$ but the normalizer of $A_7$ is not transitive on the nonsingular points.  If $(m,q)=(4,2)$, then $H$ is transitive on both singular and nonsingular points only if $H^\infty=A_9$. If $(m,q)=(4,3)$, then there is no such subgroup $H$. The last claim is verified by Magma. This completes the proof.
\end{proof}

We shall also need the following result in the symplectic case.
\begin{proposition}\label{prop_symp_PtLnTrans}
Suppose that $(V,\kappa)$ is a symplectic space, and $\dim(V)=2m\ge 4$. Let $H$ be a group of semisimilarities that is transitive on both the isotropic points and the totally isotropic lines. Then either $\Sp_{2m}(q)'\unlhd H$, or $(m,q)=(2,3)$ and $2_-^{1+4}.A_5\unlhd H$.
\end{proposition}
\begin{proof}
The group $\Sp_{2m}(q)$ has the desired property by Witt's Lemma \cite[Proposition~2.1.6]{kleidman1990subgroup}, and $\Sp_4(2)'$ also has the property (this was checked by using a computer). Assume that $H$ does not contain $\Sp_{2m}(q)'$. Since $H$ is transitive on the set of totally singular lines, we have $m=2$ by \cite[Theorem~5.2]{giudici2020subgroups}. Since it is also transitive on the isotropic points, we have $q=3$ by the same theorem. There are $40$ isotropic points and $40$ totally isotropic lines in this case. We check by Magma that $2_-^{1+4}.A_5\unlhd H$ by examining all the subgroups whose orders are divisible by $40$. This completes the proof.
\end{proof}

\subsection{The adjoint module}\label{subsec_adj}

Suppose that $G$ is one of $\SL_3(q)$, $\SU_3(q)$ and $\SU_4(q)$. We set $\F=\F_q$ or $\F_{q^2}$ respectively according as $G$ is linear or unitary, and let $\sigma$ be the field automorphism $x\mapsto x^q$ of $\F_{q^2}$. Let $M_n(\F)$ be the set of all $n\times n$ matrices over $\F$, and for $1\le i,j\le n$ write $E_{ij}$ for the $0$-$1$ matrix of order $n$ whose only nonzero entry is the $(i,j)$-th. We regard $G$ as a matrix group; in particular, $\SU_n(q)=\{g\in \SL_n(q^2)\mid gg^{\sigma\top}=I_n\}$. The group $G$ acts on $M_n(\F)$ via $X.g=g^{-1}Xg$ for $g\in G$, $X\in M_n(\F)$. For $\SL_n(q)$ set $M=M_n(\F)$, and for $\SU_n(q)$ set $M=\{X\in M_n(\F)\mid X^\top=X^\sigma\}$.  Let $U$ be the subspace of $M$ consisting of the matrices of trace $0$, and let $U'$ be the subspace consisting of scalar matrices in $M$. For $A\in U$, define $Q(A)=\sum_{1\le i<j\le n}(A_{ij}A_{ji}-A_{ii}A_{jj})$. Note that $Q(A)$ is the negative of the coefficient of $x^2$ in the characteristic polynomial $\det(xI_n-A)$ of $A$, so $Q$ is $G$-invariant. The adjoint module of $G$ is $V:=U/(U\cap U')$, which has dimension $n^2-1-[\![ p\mid n]\!]$. For $X\in U$, we write $\overline{X}$ for its image in $V$. The form $Q$ induces a nondegenerate $G$-invariant form $\kappa$ on $V$, and we have $\kappa=Q$ if $p$ does not divide $n$. We refer the reader to \cite[Section~5.4.1]{bray2013maximal} and \cite[Section~8.4]{maxsub_1314} for more details on the adjoint modules.

Since the nonzero scalar matrices acts on $V$ trivially, we have an embedding of $L:=G/Z(G)$ in $\Gamma(V,\kappa)$, where $Z(G)$ is the center of $G$. Let $\phi$ be the field automorphism $x\mapsto x^p$ of $\F$ applied to matrices entrywise, and define $X.(-\top)=X^\top$ for $X\in M$. The normalizer $K$ of $L$ in $\Gamma(V,\kappa)$ is the induced action of $\la\GL^\pm_n(q),\phi,-\top\ra$ on $V$, where $\GL^+=\GL$ and $\GL^-=\GU$.
\begin{proposition}\label{prop_adj_SL3}
If $G=\SL_3(q)$, then $K$ has at least three orbits on the singular points for $p\ne 3$, and $G$ has two orbits which are $(q+1)$- and $(q^3-q)$-tight sets of $Q(6,q)$ respectively for $p=3$.
\end{proposition}
\begin{proof}
The vectors $v_1=E_{12}$, $v_2=E_{12}+E_{23}$ and $v_3=E_{12}+E_{23}+E_{31}$ in $U$ are singular vectors for $Q$. They have different ranks, so $\la v_1\ra$, $\la v_2\ra$ and $\la v_3\ra$ are not in the same $K$-orbits. This establishes the claim for $p\ne 3$. Suppose that $p=3$, so that $(V,\kappa)$ is parabolic.  It is straightforward to show that the stabilizer of $\la\overline{v_1}\ra$ in $\SL_3(q)$ consists of $\begin{pmatrix}\lambda&a&b\\0&\mu&0\\0&c&\lambda^{-1}\mu^{-1}\end{pmatrix}$ with $\lambda,\mu\in\F_q^*$ and $a,b,c\in\F_q$, so its $\SL_3(q)$-orbits has size $(q+1)\frac{q^3-1}{q-1}$. Also, the stabilizer of  $\la\overline{v_2}\ra$ in $\SL_3(q)$ consists of $\begin{pmatrix}\lambda&b&c\\0&1&\lambda^{-1}b\\0&0&\lambda^{-1}\end{pmatrix}$ with $\lambda\in\F_q^*$ and $b,c\in\F_q$, so its $\SL_3(q)$-orbit has size $(q^3-q)\frac{q^3-1}{q-1}$. The two orbits comprises all the singular points by comparing sizes, and so they are intriguing sets by Section \ref{subsec_basic}.
Since the ovoid number $q^3+1$ does not divide their sizes, they are
both tight sets. This establishes the claim for $p=3$.
\end{proof}

\begin{proposition}\label{prop_adjMod_SU3}
If $G=\SU_3(q)$, then $K$ has at least three orbits on the singular points if $p\ne 3$, and $G$ has exactly two orbits one of which is Kantor's unitary ovoid if $p=3$.
\end{proposition}
\begin{proof}
For the case $p=3$, we refer the reader to \cite[Section~4]{KantorOvoids}. If $q\equiv 2\pmod{3}$, then there are three $K$-orbits which are $1$-, $(q^2+q)$- and $q^3$-ovoids of $Q^+(7,q)$ respectively, cf. \cite[Section~4]{KantorOvoids} and \cite[Example~3.9]{FengLiTao}. If $q\equiv 1\pmod{3}$, then $Q$ has minus sign, cf. \cite[Table~5.6]{bray2013maximal}. The ovoid number is $q^4+1$ and the rank is $r=3$. As in \cite[Section~4]{KantorOvoids}, we deduce that there is a $K$-orbit $\cN$ of size $q^3+1$ on the singular points that correspond to the rank-$1$ matrices in $U$. Since the size of $\cN$ is not divisible by $q^4+1$ or $\frac{q^3-1}{q-1}$, it is not an intriguing set. We conclude that there are more than two orbits for $q\equiv 1\pmod{3}$, cf. Section \ref{subsec_basic}. This completes the proof.
\end{proof}

\begin{proposition}
If $G=\SU_4(q)$, then $K$ has at least three orbits on the singular points.
\end{proposition}
\begin{proof}
Let $Y$ be the set of rank-$1$ matrices in $U$. Then $Y=\{\lambda vv^{\sigma\top}\mid \lambda\in\F_q^*,\,v^{\sigma\top}v=0\}$, which has size $(q-1)(q^2+1)(q^3+1)$ and forms a $K$-orbit of singular vectors of $Q$. Since the difference of two rank-$1$ matrices has rank at most $2$, we deduce that $\{y+\lambda I_4\mid \lambda\in\F\}$'s with $y\in Y$ are pairwise distinct. It follows that elements of $Y$ yield a $K$-orbit $\cN$ of singular points of size $(q^2+1)(q^3+1)$. Suppose to the contrary that $K$ has two orbits on the singular points. Then $\cN$ is an intriguing set by Section \ref{subsec_basic}. If $q$ is odd, then  the associated polar space has rank $r=7$ and its ovoid number is $\theta_r=q^7+1$. If $q=2^f$, then $Q$ has plus sign if and only if $f$ is even by \cite[Corollary~8.4.6]{maxsub_1314}. If $f$ is even, then the rank $r=7$ and the ovoid number is $\theta_r=q^6+1$; if $f$ is odd, then the rank $r=6$ and the ovoid number is $\theta_r=q^7+1$. The size of $\cN$ is divisible by neither $\frac{q^r-1}{q-1}$ nor $\theta_r$ in each case: a contradiction. This completes the proof.
\end{proof}

\subsection{The symmetric and antisymmetric squares}\label{subsec_S2W2W}

Suppose that $G=\Sp_{2m}(q)$ with $m\ge 2$. Let $W$ be its natural module of dimension $2m$ over $\F_q$, and let $\kappa_1$ be a nondegenerate $G$-invariant symplectic form on $W$. We choose a symplectic basis $e_1,\ldots,e_{2m}$ of $W$, i.e., $\kappa_1(e_i,e_j)=1$ if $j=i+m$ and $\kappa_1(e_i,e_j)=0$ if $j-i\ne\pm m$. We identify $W$ with $\F_q^{2m}$  and regard $G$ as a matrix group with respect to this basis. Then $G=\{g\in\GL_n(q)\mid gJg^\top=J\}$, where $J$ is the matrix of order $2m$ whose $(i,j)$-th entry is  $\kappa_1(e_i,e_j)$.  We define the similarity $\delta=\textup{diag}(w,\ldots,w,1,\ldots,1)$ for a primitive element $w$ of $\F_q$, and define $\phi$ as the field automorphism $a\mapsto a^p$ of $\F_q$ applied to the vectors of $W$ entrywise. We refer the reader to \cite[Section~5.2.1]{bray2013maximal} for the definitions of the symmetric square $S^2(W)$ and antisymmetric square $\wedge^2(W)$ of $W$.

\begin{remark}\label{rem_rank}
Let $M$ be the set of skew-symmetric matrices of order $2m$ with zero diagonal entries over $\F_q$. There is a vector space endomorphism
$\psi:\,\wedge^2(W)\rightarrow M$, $u\wedge v\mapsto u^\top v-v^\top u$
by the universality of exterior squares \cite{Multi}. It maps a nonzero element $u\wedge v$ to a rank-$2$ skew-symmetric matrix whose row space is $\la u,v\ra$. By linear algebra each element of $M$ is the sum of rank-$2$ skew-symmetric matrices, so $\psi$ is surjective. It is an isomorphism by comparing dimensions. For $x\in \wedge^2(W)$ we define its rank as that of $\psi(x)$.  We equip $M$ with the $G$-action $X.g=g^\top Xg$, so that $\psi$ is a $G$-module isomorphism. We can similarly build a $G$-module isomorphism between $S^2(W)$ and the set of symmetric matrices of order $2m$ via $uv\mapsto \frac{1}{2}(u^\top v+v^\top u)$ for $q$ odd, so that it makes sense to talk about the ranks of its elements.
\end{remark}

For $q$ odd, there is a nondegenerate $G$-invariant symmetric form $B$ on $S^2(W)$ such that
\[
  B(u_1u_2,v_1v_2)=\kappa_1(u_1,v_1)\kappa_1(u_2,v_2)+\kappa_1(u_1,v_2)\kappa_1(u_2,v_1)
\]
by the universality of symmetric squares. Let $\kappa_S$ be the quadratic form on $S^2(W)$ such that $\kappa_S(x)=\frac{1}{2}B(x,x)$. The action of $G$ on $S^2(W)$ induces an embedding of $L=\PSp_{2m}(q)$ in the isometry group of $\kappa_S$, and let $K$ be the normalizer of $L$ in the semisimilarity group. The group $K$ is generated by the induced action of $\la G,\delta,\phi\ra$ on $S^2(W)$.
\begin{proposition}
If $G=\Sp_4(q)$ with $q$ odd, then $K$ has at least three orbits on the singular points of  $S^2(W)$.
\end{proposition}
\begin{proof}
The vectors $v_1=e_1^2$, $v_2=e_1e_2$, $v_3=e_1e_2+e_3e_3$ are singular and have different ranks. We deduce that three corresponding projective points are in different $K$-orbits. This completes the proof.
\end{proof}

Suppose that $G=\Sp_6(q)$ from now on. There is a nondegenerate $G$-invariant symmetric bilinear form $\beta_1$ on $\wedge^2(W)$ such that
\[
  \beta_1(u_1\wedge v_1,u_2\wedge v_2)=\kappa_1(u_1,u_2)\kappa_1(v_1,v_2)-\kappa_1(u_1,v_2)\kappa_1(v_1,u_2).
\]
By the universality of exterior squares \cite[p.~5]{Multi}, there is an $\F_q$-linear map $\ell:\,\wedge^2(W)\rightarrow\F_q$ such that $\ell(u\wedge v)=\kappa_1(u,v)$. Since $\beta_1$ is nondegenerate, there is a unique element $h\in\wedge^2(W)$ such that $\beta_1(h,u\wedge v)=\kappa_1(u,v)$ for $u,v\in V$. The uniqueness implies that $h$ is $G$-invariant. It is routine to check that $h=e_1\wedge e_4+e_2\wedge e_5+e_3\wedge e_6$. Let $U'$ be the subspace spanned by $h$, and set $U:=\{x\in \wedge^2(W)\mid\beta_1(h,x)=0\}$, $V:=U/(U\cap U')$. Since $\beta_1(h,h)=3$, $U'$ is contained in $U$ if and only if $p=3$. The module $V$ is absolutely irreducible, cf. the proof of \cite[Theorem~5.1]{LubeckSmalldegree}. By \cite[Section~9.3]{maxsub_1314}, $V$ is not realized over a proper subfield and has a nondegenerate quadratic form $\kappa$. For $x\in U$, we write $\bar{x}$ for its image in $V$. Then $\beta(\bar{x},\bar{y}):=\beta_1(x,y)$ is the associated bilinear form of $\kappa$ by rescaling $\kappa$ properly. The action of $G$ on $V$ induces an embedding of $L=\PSp_{6}(q)$ in the isometry group of $(V,\kappa)$, and let $K$ be the normalizer of $L$ in the semisimilarity group. The group $K$ is generated by the induced action of $\la G,\delta,\phi\ra$ on $V$.

\begin{proposition}
If $G=\Sp_6(q)$, then $K$ has at least three orbits on the singular points of $V$ for $p\ne 3$, and $G$ has exactly two orbits which are $(q^2+1)$- and $(q^6-q^2)$-tight sets of $Q(12,q)$ for $p=3$.
\end{proposition}
\begin{proof}
Take a nonzero vector $u\wedge v$ in $U$, so that $\beta_1(h,u\wedge v)=\kappa_1(u,v)=0$.  We claim that $\kappa(\overline{u\wedge v})=0$. This is clear for $q$ odd since $\kappa(\overline{u\wedge v})=\frac{1}{2}\beta_1(u\wedge v,u\wedge v)=0$, so assume that $q$ is even. Take a  vector $w$ of $W$ that is perpendicular to $\la u,v\ra$ and not in $\la u,v\ra$. By Witt's Lemma, there are elements of $\Sp_6(q)$ that map the pair $(u,v)$ to $(w,v)$, $(u+w,v)$ respectively. By the $G$-invariance of $\kappa$,   the vectors $\overline{(u+w)\wedge v}$, $\overline{u\wedge v}$ and $\overline{w\wedge v}$ have the same $\kappa$-value $c$. We expand $\kappa(\overline{u\wedge v}+\overline{w\wedge v})=c$ to obtain $c=0$. This establishes the claim.

For $p\ne 3$, $e_1\wedge e_2$, $e_1\wedge e_2+e_3\wedge e_4$, $e_1\wedge e_2+e_3\wedge e_4+e_5\wedge e_6$ are singular vectors in $U$ and have different ranks. Therefore, there are at least three $K$-orbits on the singular points for $p\ne 3$. Suppose that $p=3$ in the sequel. We define $\cN_1:=\{\la \overline{u\wedge v}\ra\mid \kappa_1(u,v)=0\}$ which is $\Sp_6(q)$-transitive and consists of singular points. There is a bijection between elements of $\cN_1$ and the totally singular lines of $W$, so $|\cN_1|=(q^2+1)\frac{q^6-1}{q-1}$.  We set $z:=e_1\wedge e_2+e_3\wedge e_4$, whose row space is $\la e_1,e_2,e_3,e_4\ra$. Its rank is $4$ and $z+\lambda h$ has rank $6$ for $\lambda\in\F_q^*$, so the $\Sp_6(q)$-orbit of $\la \bar{z}\ra$ has the same size as that of $\la z\ra$. The singular vectors in $U$ with row space $\la e_1,e_2,e_3,e_4\ra$ are of the form $x=e_2\wedge(ae_1+ce_4)+e_3\wedge(be_1+de_4+ke_2)$, where $a,b,c,d,k$ are elements of $\F_q$ such that $D:=ad-bc\ne0$. In particular, $x=\lambda z$ if we specify $-a=d=\lambda\in\F_q^*$ and $b=c=k=0$. By Witt's Lemma there is an element $g\in\Sp_6(q)$ such that
\[
   g:\,(e_1,e_2,e_3,e_4)\mapsto(ae_1+ce_4,e_2,De_3,D^{-1}(be_1+de_4+ke_2)),
\]
and it maps $z$ to $x$.  Therefore, those $(q^4-q^2)(q-1)$ singular vectors with row space $\la e_1,e_2,e_3,e_4\ra$ are in the $\Sp_6(q)$-orbit of $z$.  The stabilizer of the subspace $\la e_1,e_2,e_3,e_4\ra$ in $\Sp_6(q)$ has size $q^9(q^2-1)^2(q-1)$, so that its $\Sp_6(q)$-orbit has size $(q^2+1)\frac{q^6-1}{q-1}$. We thus deduce that the $\Sp_6(q)$-orbit $\cN_2$ of $\la \bar{z}\ra$  has size $(q^6-q^2)\frac{q^6-1}{q-1}$. Since $|\cN_1|+|\cN_2|=\frac{q^{12}-1}{q-1}$, we deduce that $\cN_1$, $\cN_2$ are the only two $\Sp_6(q)$-orbits on the singular points. By Section \ref{subsec_basic} they are both intriguing sets of $Q(12,q)$. The ovoid number $q^6+1$ does not divide their sizes, so $\cN_1$ and $\cN_2$ are tight sets rather than $m$-ovoids.  This completes the proof.
\end{proof}

\subsection{The antisymmetric cube}\label{subsec_wedge3}

Suppose that $G=\SU_6(q)$, where $q=p^f$ with $p$ prime. Let $W$ be the natural $G$-module over $\F=\F_{q^2}$, and write $\kappa_1$ for the associated unitary form on $W$. We choose an orthonormal basis $e_1,\ldots,e_6$ of $W$ such that $\kappa_1(e_i,e_j)=[\![ i=j]\!]$, and identify $W$ with $\F_{q^2}^6$ with respect to this basis. Let $\phi$ be the field automorphism $a\mapsto a^p$ of $\F_{q^2}$, and set $\sigma=\phi^f$. Then $G=\{g\in\GL_6(q^2)\mid gg^{\sigma\top}=I_6\}$, where $I_6$ is the identity matrix of order $6$. Write $X:=\{1,\ldots,6\}$, and let $\binom{X}{3}$ be the set of all $3$-subsets of $X$. Let $U$ be the antisymmetric cube $\wedge^3(W)$.  For a  $3$-subset $I=\{i,j,k\}$ of $X$ with $i<j<k$, write both $e_I$ and $e_{ijk}$ for $e_i\wedge e_j\wedge e_k$. Then $\cB_0:=\{e_I\mid I\in\binom{X}{3}\}$ is a basis of $U$.

\begin{remark}\label{rem_Ext3_detU}
We shall work inside the exterior algebra $\wedge(W)$ of $W$, cf. \cite{Multi}.  For each $i$-dimensional subspace $W_1=\la u_1,\ldots,u_i\ra$ of $W$, it corresponds to a $1$-dimensional subspace $\la u_1\wedge\cdots\wedge u_i\ra$ of $\wedge^i(W)$ by \cite[p.~91]{Multi}. The latter is independent of the choice of the basis of $W_1$, so we denote it by $\la\det(W_1)\ra$.
\end{remark}

Set $z:=e_1\wedge\cdots\wedge e_6$. For $x,y\in U$, we define $\beta(x,y)$ to be the coefficient of $z$ in $x\wedge y$. It is a nondegenerate alternating form, and $\beta(e_I,e_J)$ is nonzero if and only if $I\cap J=\emptyset$ for $e_I,e_J\in\cB_0$. For $g\in G$, we have $z.g=e_1g\wedge\cdots\wedge e_6g=\det(g)z$, cf. \cite[p.~91]{Multi}. It follows that $\beta$ is $G$-invariant.

\begin{lemma}\label{lem_Ext3_Hf}
There is a $G$-invariant Hermitian form $\beta'$ on $U$ such that
\[
 \beta'(u_1\wedge u_2\wedge u_3,v_1\wedge v_2\wedge v_3)=\det\left(\begin{pmatrix}\kappa_1(u_i,v_j)\end{pmatrix}_{1\le i,j\le 3}\right).
\]
\end{lemma}
\begin{proof}
We write $U^*$ for the dual of $U$. We define $f:\,W^6\rightarrow\F_{q^2}$ such that
\[
  f(u_1,u_2,u_3,v_1,v_2,v_3)=\det\left(\begin{pmatrix}\kappa_1(v_i,u_j)\end{pmatrix}_{1\le i,j\le 3}\right).
\]
Since $f$ is alternating multilinear in the last three coordinates, for each tuple $(u_1,u_2,u_3)$ there is a well-defined element $\ell(u_1,u_2,u_3)\in U^*$ that maps $v_1\wedge v_2\wedge v_3$ to $f(u_1,u_2,u_3,v_1,v_2,v_3)$ by the universality of exterior powers \cite{Multi}. The mapping $(u_1,u_2,u_3)\mapsto \sigma\circ\ell(u_1,u_2,u_3)$ is alternating multilinear, so there is a well-defined mapping $\psi:\,U\rightarrow U^*$ such that $\psi(u_1\wedge u_2\wedge u_3)=\sigma\circ\ell(u_1,u_2,u_3)$. We thus have a mapping $\beta':\,U\times U\rightarrow\F_{q^2}$ such that
\begin{align*}
  \beta'&(u_1\wedge u_2\wedge u_3,v_1\wedge v_2\wedge v_3)
  =\psi(u_1\wedge u_2\wedge u_3)(v_1\wedge v_2\wedge v_3)\\
  &=f(u_1,u_2,u_3,v_1,v_2,v_3)^\sigma
  =\det\left(\begin{pmatrix}\kappa_1(u_i,v_j)\end{pmatrix}_{1\le i,j\le 3}\right).
\end{align*}
The form $\beta'$ is a nondegenerate $G$-invariant Hermitian form on $U$. This completes the proof.
\end{proof}

In the sequel, we assume that $q$ is even. Let $k$ be the algebraic closure of $\F_p$. Then $U\otimes k$ is the restriction of the irreducible $\SL_6(k)$-module $L(p^i\omega_3)$ for some $i$ by \cite{LubeckSmalldegree}. We deduce that $U$ can be realized over $\F_q$ but no smaller fields by \cite[Theorem~5.1.13]{bray2013maximal}. Take $\eta\in\F_{q^2}\setminus\F_q$ such that $\eta^{q+1}=1$. Let $V$ be the $\F_q$-span of the following basis of $U$:
\begin{equation}\label{eqn_wedg3basis}
  \cB=\left\{e_I+e_{X\setminus I}, \frac{1}{1+\eta} e_I+\frac{\eta}{1+\eta} e_{X\setminus I}\mid I \in\binom{X}{3}, 1\in I\right\}.
\end{equation}
It follows that $V\otimes\F_{q^2}=U$, and for $\sum_{I}a_Ie_I\in V$ we have $a_{X\setminus I}=a_I^q$ for each $I\in\binom{X}{3}$.
\begin{lemma}
The $\F_q$-space $V$ is $G$-invariant.
\end{lemma}
\begin{proof}
Let $\rho:\,G\rightarrow\GL_{20}(q^2)$ be the matrix representation of $g\in G$ afforded by $U$ with respect to the basis $\cB$. We claim that $g\rho$ has entries in $\F_q$ for $g\in G$, so that the $G$-invariance of $V$ follows. Let $\beta'$ be the unitary form on $U$ in Lemma \ref{lem_Ext3_Hf}. The Gram matrices of the forms $\beta$ and  $\beta'$ with respect to the basis $\cB$ in \eqref{eqn_wedg3basis} coincide and are block diagonal with diagonal entries $\begin{pmatrix}0&1\\1&0\end{pmatrix}$, which we denote by $B$.  Then $\left((g\rho)^{\sigma}\right)^\top B(g\rho)=B$, $(g\rho)^\top B(g\rho)=B$, from which we deduce that $(g\rho)^{\sigma}=g\rho$. This completes the proof.
\end{proof}

\begin{lemma}\label{lem_Ext3_FS}
For $p=2$,
there is a nondegenerate $G$-invariant quadratic form $\kappa$ on $V$ whose associated bilinear form is the restriction of $\beta$ to $V$.
\end{lemma}
\begin{proof}
We define the following quadratic form $Q$ on $U$: for $x=\sum_Ix_Ie_I\in W$ we set $Q(x)=\sum_I'x_Ix_{X\setminus I}$, where the summation $\sum'$ is taken over the $3$-subsets of $X$ that contains $1$. Its associated bilinear form is $\beta$. It remains to show that $Q$ is preserved by $G$, so that its restriction to $V$ is the desired form $\kappa$.

Let $\Lambda$ be the set of $\la\det(W_1)\ra$'s, where $W_1$ ranges over all $3$-dimensional subspaces of $W$. We claim that $Q$ vanishes on $\Lambda$. Let $W_1=\la x,y,z\ra$ be a $3$-dimensional subspace of $W$, where $x=\sum_{i=1}^6x_ie_i$, $y=\sum_{i=1}^6y_ie_i$, $z=\sum_{i=1}^6z_ie_i$. Set $l_{ij}:=y_iz_j+y_jz_i$ for $1\le i<j\le 6$. We regard  $Q(x\wedge y\wedge z)$ as a homogenous polynomial of degree $2$ in the indeterminants $x_1,\ldots,x_6$. For $i<j$ the coefficient of $x_ix_j$ is $\sum l_{ab}l_{cd}=0$, where the summation is take over the tuples $(a,b,c,d)$ such that $\{i,j,a,b,c,d\}=X$, $a<b$ and $c<d$. It follows that $Q(x\wedge y\wedge z)=0$ as desired, so $Q$ vanishes on $\Lambda$.

Since $U$ is an irreducible $G$-module, the $G$-invariant set $\Lambda$ spans $U$ over $\F_{q^2}$. Then $Q$ is the unique quadratic form on $U$ that vanishes on the $G$-invariant set $\Lambda$ and polarizes to the $G$-invariant bilinear form $\beta$, so $Q$ is $G$-invariant. This completes the proof.
\end{proof}

\begin{proposition}
Let $\kappa$ be as in Lemma \ref{lem_Ext3_FS}, and let  $K$ be the normalizer of $G$ in $\Gamma(V,\kappa)$. Then $K$ has at least three orbits on the singular points.
\end{proposition}
\begin{proof}
For $q=2$ we use the online Atlas data \cite{Atlas} to check that $K$ has three orbits on the singular points of $(V,\kappa)$, so assume that $q>2$. Take $\alpha,\beta\in\F_{q^2}^*$ such that $1+\alpha^{q+1}+\beta^{q+1}=0$; such elements exist when $q>2$. Then $x=(e_1+e_2)\wedge(e_3+e_4)\wedge(e_5+4_6)$, $y=e_{123}+e_{456}+\eta e_{124}+\eta^q e_{356}$, $z=e_{123}+e_{456}+\alpha e_{124}+\alpha^q e_{356}+\beta e_{135}+\beta^q e_{246}$ are singular vectors in $V$, where we use the expression of $\kappa$ as in the proof of Lemma \ref{lem_Ext3_FS}. The dimension of $\{u\in W\mid u\wedge w=0\}$ is $3,1,0$ for $w=x,y,z$ respectively, so $\la x\ra$, $\la y\ra$, $\la z\ra$ are in distinct $K$-orbits. This completes the proof.
\end{proof}

\subsection{The twisted tensor modules}\label{subsec_twTens}

Suppose that $G=\Sp_{2m}(q^2)$ with $m\ge 2$. Let $\sigma$ be the involutionary field automorphism of $\F_{q^2}$. Let $W$ be the natural module of $G$ over $\F_{q^2}$, and write $\kappa_1$ for the associated symplectic form. We choose a symplectic basis $e_1,\ldots,e_{2m}$ of $W$ such that $\kappa_1(e_i,e_j)=1$ if and only if $j=i+m$ and $\kappa_1(e_i,e_j)=0$ if $j-i\ne\pm m$. Let $J$ be the matrix of order $2m$ whose $(i,j)$-th entry is $\kappa_1(e_i,e_j)$. We identify $W$ with $\F_{q^2}^{2m}$ and regard  $G$ as the matrix group $\{g\in\GL_{2m}(q^2)\mid gJg^\top=J\}$ with respect to the chosen basis. Let $W^\sigma$ be the $G$-module with the same ambient space as $W$ and $G$-action $v.g=vg^\sigma$. We set $U=W\otimes W^\sigma$, which has a nondegenerate bilinear form $\beta'$ such that $\beta'(u\otimes v,u'\otimes v')=uJu'^\top vJ^\sigma v'^\top$.
If $q$ is odd, $\kappa'(x):=\frac{1}{2}\beta'(x,x)$ is a quadratic form on $U$. If $q$ is even, then there is a unique quadratic form $\kappa'$ on $U$ that vanishes on the $u\otimes v$'s and has $\beta'$ as its associated bilinear form, cf. \cite[Proposition~1.9.4]{bray2013maximal}. We define $V:=\{x\in U\mid f(x)=x\}$, where
$f(\lambda u\otimes v)=\lambda^q v^\sigma\otimes u^\sigma$ is a semilinear mapping that commutes with the action of $G$. Then $V\otimes\F_{q^2}=U$, and $V$ is an absolutely irreducible $G$-module. The restriction of $\kappa'$ to $V$, denoted by $\kappa$, is nondegenerate. This gives an embedding of $\PSp_{2m}(q^2)$ in $\Gamma(V,\kappa)$, and we write $K$ for its normalizer in the latter group.  Please refer to \cite{schaffer,cossi,Kingsubgroup2005} for more details.

\begin{proposition}
If $G=\Sp_4(q^2)$, then $K$ has at least three orbits on the singular points of $(V,\kappa)$.
\end{proposition}
\begin{proof}
We write $M=M_{4}(\F_{q^2})$ with $G$-action $X.g=g^\top Xg^\sigma$. There is a $G$-module isomorphism $\psi:\,U\rightarrow M$ such that $u\otimes v\mapsto u^\top v$ by the universality of tensors, cf. \cite{Multi}. We write $f'$, $V'$ for the respective counterparts of $f$, $V$ for $M$. Then $f'(X)=X^{\sigma\top}$ for $X\in M$, and $V'$ is the set of  Hermitian matrices in $M$. The vectors $e_1\otimes e_1$, $e_1\otimes e_2+e_2\otimes e_1$, $e_1\otimes e_2+e_2\otimes e_1+e_3\otimes e_3$ in $V$ are singular, and their images under $\psi$ have different ranks. Therefore, they are in distinct $K$-orbits. This completes the proof.
\end{proof}

\subsection{The spin modules}\label{subsec_spin}

Let $(V,Q)$ be a quadratic space over $\F_q$, where $\dim(V)\ge 5$ and $q=p^f$ with $p$ prime. Here, we allow $q$ to be even when $\dim(V)$ is odd. We define $B(x,y)=Q(x+y)-Q(x)-Q(y)$. For a nonzero vector $v$, we define
$v^{\perp,B}:=\{x\in V\mid B(v,x)=0\}$. We write $C(V,Q)$ for the Clifford algebra associated with the quadratic space $(V,Q)$ and let $\bar{\cdot}$ be the anti-isomorphism of $C(V)$ that extends the identity mapping of $V$, cf. \cite{Grove2002class}. If clear from the context, we will write $C(V)$ instead of $C(V,Q)$ for brevity. The Clifford group is the normalizer of $V$ in the group of invertible elements in $C(V)$. The even Clifford group consists of elements of the form $u_1\cdots u_{2s}$, where $u_i$'s are nonsingular vectors of $V$. For such an element $g$, define its spin norm as $N(g)=g\bar{g}$. The spin group $\textup{Spin}(V)$ is the kernel of $N$. For a nonsingular vector $v$, we define the reflection $r_v$ of $V$ such that $xr_{v}=x-\frac{B(v,x)}{Q(v)}v$ for $x\in V$. In $C(V)$, we have $xr_v=-vxv^{-1}$ for $x\in V$.

We suppose that $\dim(V)=2m\ge 8$ and $Q$ has plus sign for the rest of this subsection. We specify a basis $\{e_1,\ldots,e_m,f_1,\ldots,f_m\}$ of $V$ such that $U=\la f_1,\ldots,f_m\ra$ and $W=\la e_1,\ldots,e_m\ra$ are totally singular and $B(e_i,f_j)=[\![ i=j]\!]$. We set $X:=\{1,\ldots,m\}$. For subsets $I,J$ of $X$, we write $e_I=\prod_{i\in I}e_i$, $f_J:=\prod_{i\in J}f_j$, where the subscripts are arranged in the increasing order. In particular, $e_\emptyset=f_\emptyset=1$. Let $C_+(V)$ be the subspaces spanned by the $e_If_J$'s with $|I|+|J|$ even. Set $S_U:=C(V)/C(V)\cdot U$, where $C(V)\cdot U$ is the left ideal generated by all the elements of $U$. We write $e_I$'s for their corresponding images in $S_U$ by abuse of notation. Let $S_U^{+}$ be the quotient image of $C_{+}(V)$ in $S_U$, which has dimension $2^{m-1}$ and has the $e_I$'s with $I$ ranging over even subsets of $X$ as a basis. The even Clifford group acts on $S_U^+$ via left multiplication. In particular, $S_U^+$ is the half spin module of $\textup{Spin}(V)$.

For a maximal totally singular subspace $U'$ of $V$, the subspace $\{x\in S_U\mid ux=0\textup{ for all } u\in U'\}$ has dimension $1$ which we denote by $\la p_{U'}\ra$. We call any nonzero vector in it a {\it pure spinor} associated with $U'$. For instance, we have $\la p_U\ra=\la 1\ra$. There are two equivalence classes of maximal totally singular subspaces of $V$ by \cite[III.1.10]{alg_spinor}, and $\Omega(V,Q)$ is transitive on each equivalence class.
\begin{lemma}
Suppose that the maximal totally singular subspace $U'$ is the image of $U$ under $g=r_{v_s}\cdots r_{v_{1}}$, where each $v_i$ is nonsingular. Then $\la p_{U'}\ra=\la v_1\cdots v_{s}\ra$.
\end{lemma}
\begin{proof}
For an element $x\in U$, we have $xg=(-1)^sv_1\cdots v_{s}xv_{s}^{-1}\cdots v_{1}^{-1}$ in $C(V)$. It follows that $(xg)v_1\cdots v_{s}=(-1)^sv_1\cdots v_{s}x=0$ in $S_U^+$ for each $x\in U$. The claim then follows from the fact $Ug=U'$.
\end{proof}

First suppose that $m=4$. By \cite[p.~147]{alg_spinor}, there is a hyperbolic quadratic form $\kappa$ on $S_U^+$ such that $\kappa(zx)=N(z)\kappa(x)$ for $z$ in the even Clifford group and $x\in S_U^+$. The nonzero singular vectors of $(S_U^+,\kappa)$ are exactly the pure spinors by \cite[IV.1.1]{alg_spinor}. Set $w=e_4-f_4$, $w'=e_4+f_4$ and $V'=w^{\perp,B}$. We have natural embeddings of $C(V')$, $\textup{Spin}(V')$ in $C(V)$ and $\textup{Spin}(V)$ respectively. The group $\textup{Spin}(V')$ is $\textup{Spin}_7(q)$ or $\Sp_6(q)$ according as $q$ is odd or even, and we call $S_{U}^+$ its spin module.

\begin{proposition}\label{prop_spin7}
The group $\textup{Spin}(V')$ is transitive on both the set of singular points and each isometry class of nonsingular points of $S_U^+$. For odd $q$, a subgroup of the even Clifford group of $V'$ that contains $\textup{Spin}(V')$ is transitive on the set of nonsingular points if and only if it contains an element whose spin norm is nonsquare.
\end{proposition}
\begin{proof}
The first claim follows from \cite[Theorem~8.4]{giudici2020subgroups}.  The second claim follows from the fact that $\kappa(zx)=N(z)\kappa(x)$ for $z$ in the even Clifford group and $x\in S_U^+$.
\end{proof}

We further assume that $q=q_0^2$, and take an element $\delta\in\F_q\setminus\F_{q_0}$. We define
\begin{equation}\label{eqn_Spin_V0}
  V_0:=\la e_1, f_1, e_2, f_2, e_3, f_3,e_4+f_4,\delta e_4+\delta^q f_4\ra_{\F_{q_0}},
\end{equation}
and let $Q_0:=Q|_{V_0}$ be the restriction of $Q$ to $V_0$. Then $(V_0,Q_0)$ is an elliptic quadratic space, and $V=V_0\otimes\F_q$. The Clifford algebra $C(V_0)$ embeds in $C(V)$ naturally, and we have $\textup{Spin}(V_0)=\textup{Spin}(V)\cap C(V_0)$. The group $\textup{Spin}(V_0)$ is a covering group of $\Omega_8^-(q_0)$, and we call  $S_U^+$ its spin module.

\begin{thm}\label{thm_spinSqrt}
Suppose that $m=4$ and $q=q_0^2$,  and let $V_0$ be defined as in \eqref{eqn_Spin_V0}. Then $\textup{Spin}(V_0)$ has exactly two orbits on the singular points of $S_U^+$ which are $(q_0^3+1)$- and $(q^3-q_0^3)$-tight sets of $Q^+(7,q)$ respectively.
\end{thm}

\begin{proof}
Let $X$ be the set of maximal totally singular subspace of $(V,Q)$ in the same equivalence class as $U$ that intersects $V_0$ in a $3$-dimensional $\F_{q_0}$-subspace, and set $\cN_1:=\{\la p_{U'}\ra_{\F_q}\mid U'\in X\}$. The quadratic space $(V_0,Q_0)$ has $(q_0^3+1)\frac{q^4-1}{q-1}$ totaly singular $3$-subspaces, and each is contained in exactly one element of $X$. It follows that $|\cN_1|=(q_0^3+1)\frac{q^4-1}{q-1}$. Since $\Omega_8^-(q_0)$ acts transitively on the totaly singular $3$-subspaces of $V_0$, we deduce that it is transitive on the set $X$. It follows that $\textup{Spin}(V_0)$ is transitive on $\cN_1$.

We claim that $\Omega_8^-(q_0)$ is transitive on the set $Y$ of maximal totally singular subspaces of $V$ in the same equivalence class as $U$ that intersect $V_0$ in a $1$-dimensional $\F_{q_0}$-subspace. We define $Y_1:=\{U'\in Y\mid U'\cap V_0=\la f_1\ra_{\F_{q_0}}\}$. Since $\Omega_8^-(q_0)$  is transitive on the singular points of $V_0$, it suffices to show that $\Omega_8^-(q_0)_{\la f_1\ra}$ is transitive on $Y_1$. Set $V_1'=f_1^{\perp,B}/\la f_1\ra_{\F_q}$, and let $V_0'$ be the quotient image of $V_0$ in $V_1'$. Let $\kappa_1'$ be the induced quadratic form on $V_1'$, and identify the quadric defined by $(V_1',\kappa_1')$ as the Klein quadric $Q^+(5,q)$. Then $\PG(V_0')$ is a Baer subgeometry that meets the quadric in a $Q^-(5,q_0)$. Since the elements in $X_1$ intersect $V_0$ in $\la f_1\ra_{\F_{q_0}}$, their quotients in $V_1'$ are totally singular $3$-subspaces that intersects $V_0'$ trivially. There is an Hermitian surface $H(3,q)$ in $\PG(3,q)$ such that the Klein correspondence maps totally singular lines of $H(3,q)$ to the singular points of $Q^-(5,q_0)$, cf. \cite[Section~15.4]{Hirsch3dim}. Under such a correspondence the $q_0^3(q_0-1)(q+1)$ nonsingular points of $\PG(3,q)$ are mapped to the totally singular $3$-subspaces of one equivalence class in $V_1'$ that intersect $V_0'$ trivially, and the same is true for the $q_0^3(q_0-1)(q+1)$ non-tangent planes of $\PG(3,q)$. By considering the full preimages of those totally singular $3$-subspaces in $V$, we deduce that $|Y_1|=q_0^3(q_0-1)(q+1)$. Since $\SU_4(q_0)$ is transitive on both the set of nonsingular points and the set of non-tangent planes of  $H(3,q)$, we deduce that $\Omega_6^-(q_0)\cong \SU_4(q_0)$ is transitive on $Y_1$. Since $\Omega_6^-(q_0)\le\Omega_8^-(q_0)_{\la f_1\ra}$, this establishes the claim.

We have $|Y|=|Q^-(7,q_0)|\cdot|Y_1|=(q^3-q_0^3)\frac{q^4-1}{q-1}$ by the arguments in the last paragraph. Since $\Omega_8^-(q_0)$ is transitive on $Y$, we deduce that $\textup{Spin}(V_0)$ is transitive on $\cN_2:=\{\la p_{U'}\ra_{\F_q}\mid U'\in Y\}$. Since $|\cN_1|+|\cN_2|=|Q^+(7,q)|$, it follows that they are the only two $\textup{Spin}(V_0)$-orbits on the singular points of $S_U^+$. By Section \ref{subsec_basic}, $\cN_1$ and $\cN_2$ are both intriguing sets. Since their sizes are not divisible by the ovoid number $q^3+1$, they are tight sets rather than $m$-ovoids. This completes the proof.
\end{proof}

Suppose that $m=5$ in the sequel, and define
\begin{equation}\label{eqn_spin_V1}
 V_1:=z^{\perp,B},\quad\textup{where }z:=e_5-f_5.
\end{equation}
The restriction of $Q$ to $V_1$ is parabolic, and $\textup{Spin}(V_1)$ embeds in $\textup{Spin}(V)$ as  $\textup{Spin}(V_1)=\{x\in \textup{Spin}(V)\mid xzx^{-1}=z\}$. The group $\textup{Spin}(V_1)$ is absolutely irreducible on $S_U^+$, and we call $S_U^+$ its spin module. The group $\Omega(V_1)$ is transitive on each equivalence class of maximal totally singular subspaces of $V$ by  \cite[Theorem~8.4]{giudici2020subgroups}, so $\textup{Spin}(V_1)$ is transitive on the set $\cN_1$ of projective points defined by pure spinors in $S_U^+$. We have $|\cN_1|=(q^3+1)\frac{q^8-1}{q-1}$, which is the number of maximal totally singular subspaces of $V$ in one equivalence class. We define
\begin{equation}\label{eqn_spin_beta}
  \beta':\,S_U^+\times S_U^+\rightarrow \F_q,\;(x,y)\mapsto f_1\cdots f_5\bar{x}zy.
\end{equation}
Here, $\bar{x}$ is the image of any preimage of $x$ in $C_+(V)$ under the anti-isomorphism $\bar{\cdot}$ of $C(V)$. It is routine to check that $\beta'$ is well-defined, indeed takes value in $\F_q$ and is $\textup{Spin}(V_1)$-invariant. Moreover, for even subsets $I,J$ of $\{1,2,\ldots,5\}$ we have $\beta'(e_I,e_J)\ne0$ if and only if their symmetric difference is $\{1,\ldots,4\}$. It follows that $\beta'$ is a nondegenerate bilinear form on $S_U^+$.

\begin{lemma}\label{lem_spin9_form}
Suppose that $m=5$, and let $V_1$ be as in \eqref{eqn_spin_V1}. There is a nondegenerate $\textup{Spin}(V_1)$-invariant quadratic form $\kappa'$ of plus sign on $S_U^+$ that vanishes on the pure spinors in $S_U^+$ and has $\beta'$ in \eqref{eqn_spin_beta} as its associated bilinear form. Moreover, the even Clifford group of $V_1$ is in its similarity group.
\end{lemma}
\begin{proof}
By \cite[Proposition~5.4.9]{kleidman1990subgroup}, $S_U^+$ has a $\textup{Spin}(V_1)$-invariant quadratic form $\kappa'$ of plus sign.
Since $\textup{Spin}(V_1)$ is absolutely irreducible on $S_U^+$ and $\beta'$ is a nondegenerate $\textup{Spin}(V_1)$-invariant bilinear form, we deduce that $\beta'$ is the associated bilinear form of $\kappa'$ after properly rescaling $\kappa'$ by \cite[Lemma~1.8.8]{bray2013maximal}. We set $V_2=\la e_1,f_1,\ldots,e_4,f_4\ra$ and $U_2:=U\cap V_2$. There is a natural embedding of the half spin module $S_{U_2}^+$ of $\textup{Spin}(V_2)$ in $S_U^+$, and $S_{U_2}^+$ has a basis consisting of the $e_I$'s with $I$ ranging over even subsets of $\{1,\ldots,4\}$. The restriction of $\beta'$ to $S_{U_2}^+$ is nondegenerate upon direct check. Since $S_{U_2}^+$ is an absolutely irreducible $\textup{Spin}(V_2)$-module, we deduce that the restriction of $\kappa'$ is the unique $\textup{Spin}(V_2)$-invariant form on $S_{U_2}^+$ up to a scalar. It follows that $\kappa'$ vanishes on the pure spinor $1$ in $S_{U_2}^+$ by \cite[IV.1.1]{alg_spinor}. Since $\textup{Spin}(V_1)$ is transitive on the set $\cN_1$ of projective points defined by pure spinors in $S_U^+$, $\kappa'$ vanishes on all pure spinors in $S_U^+$. This establishes the first claim.

By \cite[p.~30]{kleidman1990subgroup} the product of even number of reflections preserves each equivalence class of maximal totally singular subspaces of $V$, so we deduce that the even Clifford group of $V$ preserves the set $\cN_1$. Take an element $g$ in the even Clifford group of $V_1$. Then $g=v_1\cdots v_{2s}$ for some nonsingular vector $v_i$'s of $V_1$. We have $\beta'(xg,yg)=\lambda \beta'(x,y)$, where $\lambda=\prod_{i=1}^{2s}Q(v_i)=N(g)$. Since $\cN_1$ consists of singular points and is $g$-invariant, we deduce that $g$ is a similarity of $\kappa'$ as desired. This completes the proof of the second claim.
\end{proof}

\begin{thm}\label{thm_spin9_Horb}
Suppose that $m=5$, let $V_1$ be as in \eqref{eqn_spin_V1} and $\kappa'$ be as in Lemma \ref{lem_spin9_form}, and let $\cP=Q^+(15,q)$ be the associated polar space of $(S_U^+,\kappa')$. Let $\cN_1$ be the set of projective points defined by pure spinors in $S_U^+$ and $\cN_2$ be its complement in $\cP$. Then $\cN_1$, $\cN_2$ are the only $\textup{Spin}(V_1)$-orbits on the points of $\cP$, and they are $(q^3+1)$- and $(q^7-q^3)$-tight sets of $\cP$ respectively.
\end{thm}
The remaining part of this subsection is devoted to the proof of Theorem \ref{thm_spin9_Horb}. Take the same notation as in Lemma \ref{lem_spin9_form}. We have $\cP=Q^+(15,q)$, and the set $\cN_1$ consists of singular points of $\kappa'$ by Lemma \ref{lem_spin9_form}. For each nonzero singular vector $v\in V$, we define the subspace $T_{\la v\ra}:=\{x\in S_U^+\mid  vx=0\}$.

\begin{lemma}\label{lem_spin_Tv}
Let $Y$ be the set of singular points of $Q$ in $V_1$, and suppose that $q>2$. Then $\textup{Spin}(V_1)$ is transitive on the set $\mathcal{T}:=\{T_{\la v\ra}\mid \la v\ra\in Y\}$, and the union of elements in $\mathcal{T}$ is $Q^+(15,q)$.
\end{lemma}

\begin{proof}
The subspace $T_{\la f_1\ra}$ has dimension $8$, and it has a basis consisting of the $e_I$'s with $I$ ranging over the even subsets of $\{2,\ldots,5\}$ upon direct check. It is totally isotropic with respect to the form $\beta'$ by the paragraph following \eqref{eqn_spin_beta}. For any nonzero elements $a_2,\ldots,a_4\in\F_q^*$, $e_2\cdots e_5=(e_2+a_2f_2)\cdots(e_4+a_4f_4)(e_5+f_5)$ is a pure spinor of $S_U^+$ that lies in $T_{\la f_1\ra}$. Similarly, we deduce that the specified basis of $T_{\la f_1\ra}$ are all pure spinors. Therefore,  $T_{\la f_1\ra}$ is a maximal totally singular subspace of $S_U^+$ by Lemma \ref{lem_spin9_form}.

For nonsingular vectors $w_1,w_2$ of $V$, we have $T_{\la vr_{w_2}r_{w_1}\ra}=w_1w_2\cdot T_{\la v\ra}$. Since $\Omega(V_1)$ is transitive on $Y$, we deduce that $\textup{Spin}(V_1)$ is transitive on the set $\mathcal{T}$. The subspace $T_{\la f_1\ra}$ is totally singular, so is each subspace in $\mathcal{T}$. Fix a nonzero singular vector $x$ of $S_U^+$. It remains to show that there is a nonzero singular vector $v$ in $V_1$ such that $x\in T_{\la v\ra}$, so that the last claim in the lemma follows.

We claim that $S_x:=\{z'ux\mid u \in V_1\}$ is a totally singular subspace of $S_U^+$, where $z':=e_5+f_5$. Take a vector $u$ in $V_1$. If $u$ is nonsingular, then $z'u$ is in the even Clifford group of $V_1$. By Lemma \ref{lem_spin9_form}, $z'ux$ is also singular. If $u$ is singular, then take a nonsingular vector $v$ in $\la z,u\ra^{\perp,B}$. The vector $v+\lambda u$ is nonsingular, so that $z'(v+\lambda u)x$ is singular for any $\lambda\in\F_q$. Since $q>2$, we expand $\kappa'(z'(v+\lambda u)x)=0$ and compare the coefficients of $1,\lambda,\lambda^2$ to deduce that $z'ux$ is singular. This establishes the claim.

Since $S_x$ is totally singular, it has dimension at most $8$. Since $V_1$ has dimension $9$, there is a nonzero vector $v\in V_1$ such that $z'vx=0$ by the definition of $S_x$. Since $z'$ is invertible, we have $vx=0$, i.e., $x\in T_{\la v\ra}$. It follows that $Q(v)x=vvx=0$, so $Q(v)=0$. This completes the proof.
\end{proof}

\begin{proof}[Proof of Theorem \ref{thm_spin9_Horb}]  We verify the claim for $q=2$ by using the Atlas data \cite{Atlas}, so assume that $q>2$ in the sequel. We set  $V_3=\la e_1,f_1,z\ra^{\perp,B}$. Then $V_3$ is spanned by $\{e_2,f_2,\ldots,e_4,f_4,z'\}$, where $z'=e_5+f_5$.  Let $Y$ be the set of singular points of $Q$ in $V_1$, and set $\mathcal{T}:=\{T_{\la v\ra}\mid \la v\ra\in Y\}$ which consists of maximal totally singular subspaces of $S_U^+$. By Lemma \ref{lem_spin_Tv}, $\textup{Spin}(V_1)$ is transitive on $\mathcal{T}$ and the union of  subspaces in $\mathcal{T}$ is $Q^+(15,q)$. There is a natural embedding of the half spin module $S_{U_3}^+$ of $\textup{Spin}(V_3)$ inside $S_U^+$, which identifies $S_{U_3}^+$ with $T_{\la f_1\ra}$. The even Clifford group of $V_3$ has two orbits  on the points of $T_{\la f_1\ra}$ by Proposition \ref{prop_spin7}. For nonsingular vectors $v_1,\ldots, v_{2s}$ in $V_3$, set $\lambda=\prod_iQ(v_i)^{-1}$ and $g=v_1\cdots v_{2s}$. The spin norm of $g'=g(e_1+f_1)(e_1+\lambda f_1)$ is $1$ and it commutes with $z$, so $g'$ is in $\textup{Spin}(V_1)$. For $x\in T_{\la f_1\ra}$, we have $gx=g'x$. In other words, there is a subset of $\textup{Spin}(V_1)$ that induces the same action as the even Clifford group of $V_3$ on $T_{\la f_1\ra}$. We thus conclude that there are at most two $\textup{Spin}(V_1)$-orbits on the singular points of $S_U^+$. Since $\cN_1$ is a $\textup{Spin}(V_1)$-orbit, it follows that its complement $\cN_2$ is the other $\textup{Spin}(V_1)$-orbit. By Section \ref{subsec_basic} they are both intriguing sets, and they are tight sets since their sizes are not divisible by the ovoid number $q^7+1$. This completes the proof.\end{proof}

\section{Proof of Theorem \ref{thm_red}}\label{sec_red}

In this section we present the proof of Theorem \ref{thm_red}. Let $V$ be a vector space over $\F_q$ equipped with a nondegenerate quadratic, unitary or symplectic form $\kappa$, and write $d=\dim(V)$. Let $\cP$ be the associated polar space, and assume that it has rank $r\ge 2$ and is not $Q^+(3,q)$. Let $\theta_r$ be its ovoid number, cf. Table \ref{tab_thetar}. Suppose that $H$ is a reducible subgroup of $\Gamma(V,\kappa)$ that has two orbits $O_1$, $O_2$ on the points of $\cP$. Let $H_0$ be the quotient group of $H$ in $\textup{P}\Gamma(V,\kappa)$, and write $Z$ for the center of $\GL(V)$. By Section \ref{subsec_basic} the two $H$-orbits are either $i_1$- and $i_2$-tight sets of $\cP$ or $m_1$- and $m_2$-ovoids of $\cP$ respectively, where $i_1+i_2=\theta_r$ and $m_1+m_2=\frac{q^r-1}{q-1}$.

Let $W$ be an $H$-submodule of $V$ of the lowest dimension. Let $N_\Gamma(W)$ be the stabilizer of $W$ in $\Gamma(V,\kappa)$, so that $H\le N_\Gamma(W)$. Let $\cN_1$ be the set of  singular points in $W^\perp$, and let $\cN_2$ be its complement in $\cP$. Then $\cN_1$, $\cN_2$ are both $N_\Gamma(W)$-invariant, and they are the two $H$-orbits when they are both nonempty. Let $H|_{W^\perp}$ be the restriction of $H$ to $W^\perp$, and let $\kappa|_W$ be the restriction of $\kappa$ to $W$. Since the subspace $W\cap W^\perp$ is also $H$-invariant, we deduce that either $W\cap W^\perp=0$, or $W\cap W^\perp=W$, i.e., $W\le W^\perp$. If $W\cap W^\perp=0$, then $W$ is nondegenerate and $V=W\oplus W^\perp$. Suppose that $W\le W^\perp$ and $W$ is not totally singular or isotropic. This occurs only if $\kappa$ is orthogonal and $q$ is even. Since $\kappa$ is nondegenerate, $d$ is also even. The mapping $\kappa|_W:\,W\rightarrow\F_q$ is $\F_q$-semilinear, i.e., $\kappa(\lambda x+y)=\lambda^2\kappa(x)+\kappa(y)$ for $x,y\in W$ and $\lambda\in\F_q$. Since $W$ is not totally singular, $\kappa|_W$ is surjective. We deduce that $\ker(\kappa|_W)$ is an $H$-invariant $\F_q$-subspace of codimension $1$ in $W$, and it is trivial by the choice of $W$. It follows that $\dim(W)=1$.  To summarize, we have one of the following cases:
\begin{enumerate}
\item[(1A)] $W$ is nondegenerate, and $V=W\oplus W^\perp$,
\item[(1B)] $W$ is totally singular or isotropic,
\item[(1C)] $\kappa$ is orthogonal with $q$, $d$ even,  and $W=\la v\ra$ for some nonsingular vector $v$.
\end{enumerate}

\begin{proposition}\label{prop_C1_singular}
If either $\kappa$ is symplectic or we are in case (1B), then $W$ is a maximal totally singular or isotropic subspace of $V$.
\end{proposition}
\begin{proof}
We only gives details for the case where $\kappa$ is symplectic, since the other case is similar. Suppose that $\kappa$ is symplectic, so that each nonzero vector is isotropic. We can not be in case (1A), since otherwise $N_\Gamma(W)$ stabilizes both $W$ and $W^\perp$ and it has at least three orbits on the points of $\cP$. Suppose that we are in case (1B), i.e., $W$ is totally isotropic. Let $\cM'$ be the projective points of $\PG(W)$, which is nonempty and $N_\Gamma(W)$-invariant. Then  $\cM'$ is one of the two $H$-orbits, and so it is either a tight set or $m$-ovoid. Let $U$ be a maximal totally isotropic subspace that contains $W$, so that the points of $\PG(U)$ form a $1$-tight set. Since an $m$-ovoid and a $1$-tight set intersects in $m$ points by \cite[Theorem~4]{BambergTightsets2007}, we deduce that $\cM'$ is a $1$-tight set, i.e., $W=U$. This completes the proof.
\end{proof}

We assume that $\kappa$ is orthogonal or unitary in the sequel in view of Proposition \ref{prop_C1_singular}. The case (1B) has also been covered in the same proposition, and it remains to consider (1A), (1C).
\begin{proposition}\label{prop_C1_WNondeg}
Suppose that we are in case (1A) with $\kappa$ orthogonal or unitary. Then $W$ contains no nonzero singular vectors, and $\cN_1$, $\cN_2$ are the two $H$-orbits. Moreover,
\begin{enumerate}
\item[(1)] If $W=\la u\ra$, then $H|_{W^\perp}$ is transitive on both $\cM_1$ and the set $\mathcal{T}_2$ of nonsingular points $\la v\ra$ in $W^\perp$ such that $\kappa(v,v)=-\kappa(u,u)$  or $\kappa(v)=-\kappa(u)$. If $q$ is odd and $\kappa$ is orthogonal, then $|\mathcal{T}_2|=\frac{1}{2}|\cN_2|$.
\item[(2)] If $\dim(W)=2$, then $d$ is even, $\kappa$ is orthogonal, and $H|_{W^\perp}$ is transitive on both $\cN_1$ and the set of all nonsingular points in $W^\perp$.
\end{enumerate}
\end{proposition}
\begin{proof}
If $W$ has a nonzero singular vector $u_1$, then $W^\perp$ also has a nonzero singular vector $u_2$ since $\dim(W^\perp)\ge \dim(W)$. The singular points $\la u_1\ra$, $\la u_2\ra$ and $\la u_1+u_2\ra$ are in distinct $N_\Gamma(W)$-orbits: a contradiction. Therefore, $W$ has no singular points. It follows that $\dim(W)\le 1$ or $2$ according as $\kappa$ is unitary or orthogonal, and $\dim(W)=2$ only if $W$ is elliptic.  Since $\cP$ has rank $r\ge 2$ and is not $Q^+(3,q)$, we deduce that $W^\perp$ has a nonzero singular points by considering dimensions, i.e., $\cN_1$ is nonempty. The two sets $\cN_1$, $\cN_2$ are $N_\Gamma(W)$-invariant and so are the two $H$-orbits as desired.

We first suppose that $\dim(W)=1$ and establish (1). Suppose that $W=\la u\ra$. We only give details for the case $\kappa$ is orthogonal, since the unitary case is similar.  We have $\kappa(u)\ne 0$ and $\cN_2=\{\la v+u\ra\mid v\in W^\perp,\,\kappa(v)=-\kappa(u)\}$.  Take $\la v_1\ra$, $\la v_2\ra$ from the set $\mathcal{T}_2$. For $i=1,2$, there is $\lambda_i\in\F_{q}^*$ such that $\kappa(\lambda_i v_i)=-\kappa(u)$. Then $\la \lambda_i v_i+u\ra$, $i=1,2$, are singular points in $\cN_2$.  Since $H$ is transitive on $\cN_2$, there is $h\in H$ and $\lambda\in\F_q^*$ such that $\lambda_1 v_1h+uh=\lambda(\lambda_2 v_2+u)$. It follows that $uh=\lambda u$ and $v_1h=\lambda\lambda_1^{-1}\lambda_2 v_2$. This shows the transitivity of $H|_{W^\perp}$ on $\mathcal{T}_2$ as desired. Further suppose that $q$ is odd. We define $S_1:=\{v\in W^\perp\mid \kappa(v)=-\kappa(u)\}$. Then $|\cN_2|=|S_1|=2\cdot|\mathcal{T}_2|$, which gives the second claim. This concludes the proof of (1).

We next suppose that $\dim(W)=2$ and establish (2). Since $W$ has no singular points, we see that $\kappa$ is orthogonal and its restriction to $W$ is elliptic. We have  $\{\kappa(v)\mid v\in W\setminus\{0\}\}=\F_q^*$. For each nonsingular vector $v$ in $W^\perp$, there is an element $u\in W$ such that $\kappa(u)=-\kappa(v)$ and so $u+v$ is a singular vector. We have $|(v+u)^\perp\cap\cN_1|=|u^\perp\cap\cN_1|$. By the same arguments as in the $\dim(W)=1$ case, we deduce that $H|_{W^\perp}$ is transitive on the set of nonsingular points in $\PG(W^\perp)$. If $d=2r+1$ is odd and $q$ is odd, then $|P^\perp\cap\cN_1|$ takes the two values $|Q^+(2r-3,q)|$, $|Q^-(2r-3,q)|$ as $P$ varies in $\cN_2$. It follows that $\cN_1$ is not an intriguing set for $d=2r+1$. We conclude that $d$ is even.  This completes the proof.
\end{proof}

\begin{proposition}\label{prop_C1_ca3Trans}
Suppose that we are in case (1C) and $d=2m+2\ge 6$. Then either $\Sp_{2a}(q^b)'\unlhd H$ with $m=ab$, or $G_{2}(q^b)'\unlhd H$ with $m=3b$.
\end{proposition}

The remaining part of this subsection is devoted to the proof of Proposition \ref{prop_C1_ca3Trans}. We first introduce some notation. We assume without loss of generality that $W=\la v\ra$ with $\kappa(v)=1$. Take $\delta\in\F_q$ such that $\tr_{\F_q/\F_2}(\delta)=1$. Let $\la u,v\ra$ be an elliptic $2$-subspace such that $\kappa(x_1u+x_2v)=\delta x_1^2+x_1x_2+ x_2^2$, and set $U:=\la u,v\ra^\perp$. Let $f$ be the associated bilinear form of $(V,\kappa)$, and let $r_v$ be the reflection such that $xr_v=x+\frac{f(x,v)}{\kappa(v)}v$ for $x\in V$.
Let $f_1$ be the induced symplectic form on $W^\perp/W$.

By considering the induced action on the symplectic space $(W^\perp/W,f_1)$, we obtain a natural group homomorphism $\psi:\,N_{\Gamma}(W)\rightarrow\Gamma(W^\perp/W,f_1)$. It is routine to show that $\ker(\psi)=\la r_v\ra$ and $\psi$ is surjective by using \cite[Theorem~14.2]{Grove2002class} and Witt's lemma. Since $V=U\perp\la u,v\ra$, we have a symplectic space isomorphism between $(U,f|_U)$ and $(W^\perp/W,f_1)$ via $x\mapsto\bar{x}$ for $x\in U$. By composition we obtain a group homomorphism $\psi:\,N_\Gamma(W)\rightarrow\Gamma(U,f|_U)$, which we still denote by $\psi$ by abuse of notation.
\begin{lemma}\label{lem_C1_ca3PsiInv}
Take notation as above, and suppose that $g\in N_{I(V)}(W)$. Then there is an isometry $h$ of the symplectic space $(U,f|_U)$, an element $w\in U$ and an element $\eta\in\F_q$ such that
\begin{equation}\label{eqn_1C_hCond}
\textup{$\eta+\eta^2=\kappa(w)$, and $f(w,zh)^2=\kappa(z)+\kappa(zh)$ for $z\in U$}
\end{equation}
and
\begin{equation}\label{eqn_1C_gact}
  g:\,(u,v,z)\mapsto (u+\eta v+w,v,zh+f(w,zh)v),\quad\textup{ for } z\in U.
\end{equation}
Conversely, each triple $(h,w,\eta)$ that satisfies \eqref{eqn_1C_hCond} defines an isometry $g\in N_{I(V)}(W)$ as in \eqref{eqn_1C_gact}.
\end{lemma}
\begin{proof}
The homomorphism $\psi$ maps $N_{I(V)}(W)$ to $I(U,f|_U)$ surjectively with kernel $\la r_v\ra$, so $h:=\psi(g)$ is a isometry of $(U,f|_U)$. Since $g$ is an isometry and stabilizes $W$, we have $vg=v$. For $z\in U$, we have $zg-zh=\beta(z)v$ for some $\F_q$-linear function $\beta:\,U\rightarrow\F_q$. Since $f|_U$ is nondegenerate, there is a unique element $w\in U$ such that $\beta(z)=f(w,zh)$ for $z\in U$. We have $f(zg,zh)=f(zh+\beta(z)v,zh)=0$ and so
\begin{align*}
\beta(z)^2=\kappa(zg-zh)=\kappa(zg)+\kappa(zh)+f(zg,zh)
    =\kappa(z)+\kappa(zh).
\end{align*}
We deduce that $w\in U$ is uniquely determined by the property $f(w,zh)^2=\kappa(z)+\kappa(zh)$ for $z\in U$. Write $ug=au+\eta v+w'$, where $a,\eta\in\F_q$ and $w'\in U$. We deduce from $f(ug,vg)=1$ that $a=1$. For $z\in U$, we deduce from $f(ug,zg)=0$ that $0=f(u+bv+w,zh+\beta(z)v)=f(w+w',zh)$, so $w'=w$. From $\kappa(ug)=\delta$ we deduce that $\kappa(w)=\eta+\eta^2$. This establishes \eqref{eqn_1C_hCond} and \eqref{eqn_1C_gact}.

Conversely, if a triple $(h,w,\eta)$ satisfies  \eqref{eqn_1C_hCond}, then we defines a linear transformation $g$ as in \eqref{eqn_1C_gact}. It is clear that $g$ stabilizes $W$. The fact that $g$ is an isometry follows by direct check. This completes the proof.
\end{proof}

\begin{lemma}\label{lem_C1_ca3Typ}
Choose $z_0\in U$ such that $\kappa(z_0)=\delta$. Let $\kappa_1$ be the restriction of $\kappa$ to $U$, and define $\kappa_0(z)=\kappa_1(z)+f(z_0,z)^2$ for $z\in U$. Then the quadratic space $(U,\kappa_0)$ has the same sign as $(V,\kappa)$.
\end{lemma}
\begin{proof}
Since $\la u,v\ra$ is elliptic, $(U,\kappa_1)$ has an opposite sign to that of $(V,\kappa)$. Since $U$ is nondegenerate and $\dim(U)\ge 4$, there is an elliptic line $\la z_0,z_1\ra$ in $U$ such that $\kappa_1(x_0z_0+x_1z_1)=x_1^2+x_1x_0+\delta x_0^2$. For $z=x_0z_0+x_1z_1$, we have $f(z_0,z)=x_1$ and $\kappa_0(z)=x_1(x_0+\delta x_1)$, so $\la z_0,z_1\ra$ is a hyperbolic line for $\kappa_0$. Since $\kappa_0$ and $\kappa_1$ agrees on the orthogonal complement of $\la z_0,z_1\ra$ in $U$, we deduce that $(U,\kappa_0)$ has an opposite sign to  that of $(U,\kappa_1)$. This establishes the claim.
\end{proof}

\begin{proof}[Proof of Proposition \ref{prop_C1_ca3Trans}]Take the same notation as introduced following this proposition. Each element of $\cN_1$ is of the form $\la \kappa(z)^{q/2}v+z\ra$ with $z\in U\setminus\{0\}$, which is in bijection with the points of $\PG(U)$. Since $H$ is transitive on $\cN_1$, we deduce that $\psi(H)$ is transitive on the points of $\PG(U)$. By \cite[Theorem~5.2(a)]{giudici2020subgroups} we have $K_1\unlhd \psi(H)$ ($m=ab$) or $K_2\unlhd \psi(H)$ ($m=3b$), where $K_1=\Sp_{2a}(q^b)'$ and $K_2=G_{2}(q^b)'$. Let $H_i$ be the full preimage of $K_i$ under $\psi$ for $i=1,2$. Since $\ker(\psi)$ has order $2$ and $q$ is even, we deduce that $H_1^{\infty}\cong\Sp_{2a}(q^b)'$ ($m=ab$) and $H_2^{\infty}\cong G_{2}(q^b)'$ ($m=3b$). Both $H_1$ and $H_2$ lie in $N_{I(V)}(W)$, and they are transitive on $\cN_1$ by the bijection between $\cN_1$ and points of $\PG(U)$. It remains to show that $H_1$ and $H_2$ are both transitive on $\cN_2$.

Fix an vector $z_0\in U$ such that $\kappa(z_0)=\delta$, so that $\la u+z_0\ra$ is a singular point in $\cN_2$.  Take $g$ in $N_{I(V)}(W)$ that stabilizes $\la u+z_0\ra$, and let $h,w,\eta$ be as in Lemma \ref{lem_C1_ca3PsiInv}. Since $g$ is an isometry, we deduce that $u+z_0=ug+z_0g$, i.e., $(f(w,z_0h)+\eta)v=z_0h+z_0+w$ by (1) of Lemma \ref{lem_C1_ca3PsiInv}. The right hand side is in $U$ while $v$ is not, so $w=z_0h+z_0$ and $\eta=f(w,z_0h)=f(z_0,z_0h)$. It follows that  $w$ and $\eta$ are uniquely determined by $h$. We plug them into the second equation in (2) of Lemma \eqref{lem_C1_ca3PsiInv} to deduce that $f(z_0,z)^2+\kappa(z)=f(z_0,zh)^2+\kappa(zh)$ for $z\in U$. By specifying $z=z_0$ we obtain $f(z_0,z_0h)^2+\kappa(z_0h)+\kappa(z_0)=0$, i.e., $\eta+\eta^2=\kappa(w)$. We conclude that $g$ stabilizes $\la u+z_0\ra$ if and only  if $h$ stabilizes the quadratic form $Q_0(z):=f(z_0,z)^2+\kappa(z)$ on $U$.

We now determine the stabilizer of $Q_0$ in the corresponding $K_i$ for $i=1,2$. By Lemma \ref{lem_C1_ca3Typ}, $Q_0$ has the same sign as $(V,\kappa)$. There is an embedding of $\F_{q^b}$ in $\End(U)$ such that $U$ becomes a vector space $U'$ over $\F_{q^b}$ equipped with a $\psi(K_i)$-invariant symplectic form $f'$ such that $f|_U=\tr_{\F_{q^b}/\F_q}\circ f'$. Let $\kappa':\,U'\rightarrow\F_{q^b}$ be any quadratic form such that $\kappa|_U=\tr_{\F_{q^b}/\F_q}\circ\kappa'$, and set $Q'(z):=f'(z_0,z)^2+\kappa'(z)$. Then $Q_0=\tr_{\F_{q^b}/\F_q}\circ Q'$, and  $(U',Q')$ has the same sign as $(V,\kappa)$ by Table \ref{tab_extfieldP'gt0} and Lemma \ref{lem_C1_ca3Typ}. An element $h\in K_i$ is an isometry of $(U,Q_0)$ if and only if it is an isometry of $(U',Q')$. The stabilizer of $(U',Q')$ in $\Sp_{2a}(q^b)$ is $\Omega(U',Q')$, and the stabilizer in $G_2(q^b)$ is $\SU_3(q^b).2$ or $\SL_3(q^b).2$ according as $(V,\kappa)$ is elliptic or hyperbolic by \cite[Propositions~2.1,~6.1]{CK_G2}. We deduce that the orbit of $\la u+z_0\ra$ under $H_i$ has the same size as $\cN_2$ for $i=1,2$. It follows that $H_1$ and $H_2$ are both transitive on $\cN_2$. This completes the proof.
\end{proof}

Let us summarize what we have done so far in this section. Let $W$ be an $H$-submodule of $V$ of the lowest dimension, so that we have either case (1A), (1B) or (1C). If $\kappa$ is symplectic or we are in case (1B), then we have (R1) of Theorem \ref{thm_red} by Proposition \ref{prop_C1_singular}. Suppose that $\kappa$ is orthogonal or unitary. For case (1A) we obtain (R2) and (R3) of Theorem \ref{thm_red} by Proposition \ref{prop_C1_WNondeg}, and for case (1C) we obtain (R4) of Theorem \ref{thm_red} by Proposition \ref{prop_C1_ca3Trans}. This completes the proof of Theorem \ref{thm_red}.

\section{Proof of Theorem \ref{thm_TS}: the non-geometric subgroups}\label{sec_classS}

Write $q=p^f$ with $p$ prime. Let $V$ be a $d$-dimensional vector space over $\F_q$ that is equipped with a nondegenerate quadratic, unitary or symplectic form $\kappa$. Suppose that the associated polar space $\cP$ has rank $r\ge 2$ and $\cP\ne Q^+(3,q)$. We assume that $H$ is a subgroup of  $\Gamma(V,\kappa)$ of Aschbacher class $\cS$ that has two orbits $O_1$, $O_2$ on the points of $\cP$ which are $i_1$- and $i_2$-tight sets respectively. We refer the reader  to \cite[Definition~2.1.3]{bray2013maximal} for the exact definition of the Aschbacher class $\cS$. We write $H_0$ for the quotient group of $H$ in $\textup{P}\Gamma(V,\kappa)$, and write $Z$ for the center of $\GL(V)$. There is a simple group $S$ such that $S\le H_0\le M$, where $M$ is the normalizer of $S$ in $\textup{P}\Gamma(V,\kappa)$. We write $L:=H^\infty$ and $K:=N_{\Gamma(V)}(L)$, so that $L\unlhd H\unlhd K$. If $S$ is a simple group of Lie type in the characteristic $p$, then we write $\tilde{S}$ for the covering group of $S$ by the $p'$-part of its Schur multiplier; otherwise, we write $\tilde{S}$ for the full covering group of $S$.

We specify a basis of $V$ such that $\kappa$ takes the standard form as in \cite[Table~1.1]{bray2013maximal}. We regard elements of $V$ as row vectors and identify $\Delta(V)$ as a matrix group with respect to the specified basis. The $L$-module $V$ is naturally a $\tilde{S}$-module, and let $\rho:\,\tilde{S}\rightarrow\GL_d(q)$ be the corresponding representation afforded by $V$. We then have $L=\tilde{S}\rho$. Since $H$ is not in the Aschbacher class $\cC_5$, $\F_q$ is the smallest field over which $V$ can be realized. By \cite[Theorem~74.9]{CurRei_Rep}, $\F_q$ is the smallest extension field of $\F_p$ that contains the traces of $g\rho$, $g\in \tilde{S}$.

\begin{proposition}\label{prop_classS_caseS}
If $\cP=W(d-1,q)$, then we have one of the symplectic examples in Table \ref{tab_TS_cSother}.
\end{proposition}

\begin{proof}
As $H$ has two orbits on the singular points, the group $HZ$ has two orbits on the nonzero vectors of $V$. First consider the case where $S$ is a simple group of Lie type in the defining characteristic $p$. By \cite[Section~2]{LiebeckRank3}, we have one of the cases (A8)-(A11) in the main theorem therein. In each case, it is routine to determine the smallest field of realization of the respective module by using \cite[Theorem~5.1.13]{bray2013maximal}. For (A8), we have $H^\infty=\SL_5(q)$ and $V=\wedge^2(W)$ which has no form by \cite[Proposition~5.4.3]{bray2013maximal}, where $W$ is the natural module of $\SL_5(q)$. For (A9), $V$ is the spin module for $\tilde{S}=\textup{Spin}_7(q)$ which has a nondegenerate quadratic form by \cite[Proposition~5.4.9]{bray2013maximal}. For (A10), $V$ is the spin module for $\tilde{S}=\textup{Spin}^+_{10}(q)$ which is not self-dual by \cite[Proposition~5.4.9]{kleidman1990subgroup}. It follows that $V$ does not have a symplectic form by \cite[Lemma~2.10.15]{kleidman1990subgroup}. For (A11),  we have $\tilde{S}=\textup{Sz}(q)$ and one of the two orbits is an ovoid of $W(3,q)$.

Next consider the case where $S$ is an alternating group, a simple group of Lie type in characteristic different from $p$ or a sporadic simple group. We have one of the examples in \cite[Table~2]{LiebeckRank3} by \cite[Sections~2-4]{LiebeckRank3}, and we list those with a symplectic form in Table \ref{tab_TS_cSother} by \cite[Table~11.6]{SRG}. In the row about $W(3,5)$ (in Table~1.6), the normalizer $M$ of $2.A_6$ has two orbits $O_1$, $O_2$ on the singular points while $2.A_6$ has three orbits by \cite[Lemma~4.3]{LiebeckRank3}.  In all the other examples, $S$ has exactly two orbits on the points of $\cP$. We remark that the $A_9$ row of \cite[Table~2]{LiebeckRank3} is not in Table \ref{tab_TS_cSother} since it arises from the embedding $A_9<\Omega_8^+(2)$ which is in Aschbacher class $\cC_8$. This completes the proof.
\end{proof}

In view of Proposition \ref{prop_classS_caseS}, we suppose that $\kappa$ is orthogonal or unitary for the rest of this section.  We shall make use of the numerical conditions in Lemma \ref{lem_ParaCond}.
Since $2r\le d\le 2r+2$, we deduce that $\lceil\frac{d}{2}\rceil-1\le r\le\lfloor\frac{d}{2}\rfloor$. Since $|H_0|$ divides $|\Aut(S)|$, we deduce that $d<2+\log_q(2\cdot|\Aut(S)|)$ by Lemma \ref{lem_ParaCond}. Together with the bound $2r\le d$, we deduce that $r<1+\frac{1}{2}\log_q(2\cdot|\Aut(S)|)$.
\begin{remark}\label{rem_classS_Smethod}
We describe a general strategy to handle a specific simple group $S$. Let $R(S)$ be a lower bound on the dimensions of absolutely irreducible representations of $V$ in characteristic $p$. Then $R(S)\le d$ and so $\frac{R(S)}{2}-1\le r$. There are finitely many $(q,r)$ pairs such that $\frac{q^r-1}{q-1}$ divides $|\Aut(S)|$ and
\[
  \frac{1}{2}R(S)-1\le r<1+\frac{1}{2}\log_q(2\cdot|\Aut(S)|).
\]
If $q$ is not a square, then we expect $\cP$ to be orthogonal and so the ovoid number is $\theta_r=q^{d-1-r}+1$. In such a case, we check whether there is an integer $i$ such that $\lcm(i,\theta_r-i)\cdot\frac{q^r-1}{q-1}$ divides $|\Aut(S)|$. We then look up \cite{LubeckSmalldegree} and \cite{HissMalle} for candidate representations with Frobenius-Schur indicators  $+$ or $\circ$ whose dimensions are in the interval $[\max\{2r,R(S)\},2r+2]$. By Section \ref{subsec_basic}, it suffices to consider one representative from each weak equivalence class. Their matrix representations are usually available via the online Atlas database \cite{Atlas} or the Magma Command \textrm{AbsolutelyIrreducibleModules}, and the Magma command \textrm{ClassicalForms} returns the form $\kappa$. Then we can determine the number of their orbits on the singular points by computer.
\end{remark}

\subsection{The alternating group $A_n$, $n\ge 5$}\label{subsec_Alt}

Suppose that $S$ is an alternating group $A_n$, $n\ge 5$. We first recall the definition of fully deleted modules for $A_n$. The group $A_n$ acts on the vector space $\F_{p}^{n}$ by permuting the coordinates. Set $U=\left\{(a_{1},\ldots,a_{n})\mid \sum_{i=1}^{n}a_{i}=0\right\}$, $W=\{(a,\ldots,a)\mid a\in\F_{p}\}$, and define the fully deleted module for $A_n$ as $V_0:=U/(U\cap W)$. We have $\dim(V_0)=n-1-[\![ p\mid n]\!]$. For a nonzero vector $v\in U$, we write $c_i$ for the number of occurrence of $i$ among its coordinates for $i\in\F_p$.  If $\{i\mid c_i> 0\}=\{i_1,\ldots,i_s\}$, then we say that $v$ is of the shape $i_1^{c_{i_1}}\cdots i_s^{c_{i_s}}$. Also, we define the Hamming weight of $v$ as $\omega_H(v)=n-c_0$. If $p$ is odd, then $A_n$ preserves the quadratic form $Q(a)=\sum_{i=1}^na_i^2$ on $U$. If $p=2$ and $n\not\equiv2\pmod{4}$, then $A_n$ preserves the quadratic form  $Q(a)=(-1)^{\omega_H(a)/2}$ on $U$. In both cases $U\cap W$ is  the radical of $Q$, and so there is an induced nondegenerate quadratic form on $V_0$. In the remaining cases, $V_0$ does not have an $A_n$-invariant quadratic or unitary form, cf. \cite[p.~187]{kleidman1990subgroup}. This leads to the following embeddings of $A_n$ in orthogonal groups:
\begin{enumerate}
\item[(1)]If $p$ is odd, $n=2m+1$ and $p\nmid n$, then $A_{n}\leq\Omega_{2m}^{\epsilon}(p)$, where $\epsilon=+$ if and only if $(-1)^m\cdot n$  is a square of $\F_p^*$;
\item[(2)]If $pn$ is odd and $p\mid n$, then $A_{n}\leq\Omega_{n-2}(p)$;
\item[(3)]If $p$ is odd, $n=2m+2$ and $p\nmid n$, then $A_{n}\leq\Omega_{2m+1}(p)$;
\item[(4)]If $p$ is odd, $n=2m+2$ and $p\mid n$, then $A_{n}\leq\Omega_{2m}^{\epsilon}(p)$, where $\epsilon=-$ if and only if  $m$ is even and $p\equiv 3\pmod{4}$;
\item[(5)]If $p=2$ and $n=2m+1$, then $A_n\le\Omega_{2m}^\epsilon(2)$, where $\epsilon=+$ if and only if $m \pmod{4}\in\{0,3\}$;
\item[(6)] If $p=2$ and $n=2m+2$ ($m$ odd), then $A_n\le\Omega_{2m}^\epsilon(2)$, where $\epsilon=+$ if and only if $m\equiv 3\pmod{4}$.
\end{enumerate}
\begin{lemma}\label{lem_tp11}
If $p$ is odd and $p\equiv 2\pmod{3}$, then there are nonzero elements of $\F_q^*$ such that $1^j+a_1^j+a_2^j+a_3^j=0$ for $j=1,2$. Moreover, $(a_1,a_2,a_3)\ne (1,1,1)$ for each such triple.
\end{lemma}
\begin{proof}
Consider the quadric $\cQ$ defined by $x_1^2+\cdots+x_4^2=0$ in $\PG(3,q)$. The vector $w=(1,1,1,1)$ is not singular, so the tangent hyperplane $x_1+\cdots+x_4=0$ at $\la w\ra$ intersects $\cQ$ in a conic. Let $\la (a_0,a_1,a_2,a_3)\ra$ be a point on the conic, and assume that $a_0=1$ without loss of generality. We claim that the $a_i$'s are nonzero. Suppose to the contrary that one of them, say $a_3$, is zero. Then $a_2=-1-a_1$, and $1+a_1^2+(-1-a_1)^2=0$ which simplifies to $a_1^2+a_1+1=0$. It follows that $a_1$ has multiplicative order $3$, which contradicts the condition that $p\equiv 2\pmod{3}$. The triple $(a_1,a_2,a_3)$ has all the desired properties in the lemma. This completes the proof.
\end{proof}
\begin{proposition}\label{prop_S_Anfdm}
If $V$ is the fully deleted module for $A_n$, $n\ge 5$, then $p=2$ and $n\in\{11,20\}$. The two $A_n$-orbits are respectively $11$, $22$-tight sets of $Q^-(9,2)$ for $n=11$ and $19$, $494$-tight sets of $Q^-(17,2)$ for $n=20$.
\end{proposition}
\begin{proof}
Take the same notation as in the preceding Lemma \ref{lem_tp11}, and suppose that $V$ is the fully deleted module $V_0$. We have $q=p$, and the rank $r$ is uniquely determined by the $(p,n)$ pair in each case. If $p$ divides $n$, we define $\omega_H(\bar{u}):=\min\{\omega_H(x)\mid x\in u+W\}$ for $\bar{u}=u+W\in V$. It is preserved by the induced action of $S_n$ on $V$.

We claim that the rank $r\le 10$ in each case (1)-(6). Suppose to the contrary that this is not the case. Then we have $d\ge 22$ and correspondingly $n\ge 23$. If $p\equiv0, 1\pmod{3}$, then take an element $a\in\F_p^*$ such that $a^2+a+1=0$ and set $b:=-1-a$ which is nonzero. If $p$ is odd and $p\equiv 2\pmod{3}$, then take such a triple $(a_1,a_2,a_3)$ as in Lemma \ref{lem_tp11}. If $p=2$, take singular vectors $u_1,u_2,u_3$ of the shapes $0^{n-4}1^4$,  $0^{n-8}1^8$, $0^{n-12}1^{12}$ respectively. If $p$ is odd and $p\equiv 2\pmod{3}$, then take singular vectors $u_1,u_2,u_3$ of the shapes $0^{n-4}1a_1a_2a_3$, $0^{n-8}1^2a_1^2a_2^2a_3^2$, $0^{n-12}1^3a_1^3a_2^3a_3^3$ respectively; if $p\equiv 0,1\pmod{3}$, take singular vectors $u_1,u_2,u_3$ of the shapes $0^{n-3}1^1a^1b^1$, $0^{n-6}1^2a^2b^2$, $0^{n-9}1^3a^3b^3$ respectively. Here, we slightly abuse the notation since the $a_i$'s as well as $1,a,b$ may not be distinct. If $p\nmid n$, then the $u_i$'s are in distinct $S_n$-orbits. If $p\mid n$, then it is routine to check that $\omega_H(\overline{u_i})=\omega_H(u_i)$ for $i=1,2,3$ and the $\overline{u_i}$'s are in distinct $S_n$-orbits.  This establishes the claim.

We first consider the case $p=2$. Suppose that we are in case (5). If $n\ge 13$, there are singular vectors of the shapes $0^{n-4}1^4$, $0^{n-4}1^8$, $0^{n-4}1^{12}$ which are in distinct $S_n$-orbits. If $n\in\{9,11\}$, the nonzero singular vectors all have Hamming weights $4$ or $8$, and $A_n$ acts transitively on those of each weight. For $n=9$, the two orbits are respectively $1$- and $14$-ovoids of $Q^+(7,2)$, and for $n=11$ the two orbits are respectively $11$- and $22$-tight sets of $\cP=Q^-(9,2)$. If $n=7$, then all the nonzero singular vectors have weights $4$ and $A_7$ is transitive on them. Suppose that we are in case (6), and write $n=2m+2$. Since $r\le 10$, we have $n\le 16$ if $m\equiv 3\pmod{4}$ and $n\le 20$ if $m\equiv 1\pmod{4}$. Therefore, $n$ is one of $8$, $12$, $16$, $20$. If $n\in\{8,12\}$, then all the nonzero singular vectors have Hamming weights $4$ and $A_{12}$ is transitive on them. If $n\in\{16,20\}$, then the nonzero singular vectors have weights $4$ or $8$ and corresponding $A_{20}$ has two orbits. For $n=16$, $\cP=Q^+(13,2)$ and the orbit lengths are not divisible by $2^7-1$, so they are not tight sets. For $n=20$, $\cP=Q^-(17,2)$ and the orbit lengths are $\binom{20}{4}=4845$ and $\binom{20}{8}=125970$ which are not divisible by the ovoid number $2^9+1$. It follows that the two orbits are both tight sets.

We next consider the case where $p$ is odd and $n\ge 8$. Since $r\le 10$, we deduce that $d\le 22$ and $n\le d+2\le 24$. If $8\le n\le 24$, then the $(p,n)$ pairs such that $\frac{p^r-1}{p-1}$ divides $|\Aut(S)|$ are $(3,10)$, $(3,13)$, $(3,14)$, $(3,15)$, $(7,10)$, $(7,11)$. For the pairs with $p=3$ and $p\nmid n$, the  singular vectors of the shapes $0^{n-3}1^3$, $0^{n-6}1^6$, $0^{n-6}1^32^3$  are not in the same $S_{n}$-orbits. For $(p,n)=(3,15)$,  the images of the singular vectors $0^{12}1^3$, $0^91^{6}$, $0^{9}1^32^3$ of $U$ in $V$ are in distinct $S_{15}$-orbits. For the pairs with $p=7$ and $p\nmid n$, the singular vectors $0^{n-3}1^12^14^1$, $0^{n-6} 1^2 2^2 4^2$, $0^{n-9} 1^3 2^3 4^3$ are singular vectors in distinct $S_{n}$-orbits.

It remains to consider the case where $p$ is odd and $n\le 7$. Since the rank $r\ge 2$ and we exclude $Q^+(3,q)$ from consideration, we deduce that $n$ is at least $7$, $7$, $6$, $8$, $7$, $8$ in the cases (1)-(6).   If $n=6$, then we are in case (3), $p\ne 3$ and $\cP=Q(4,p)$. The ovoid number is $\theta_2=p^2+1$ and the rank $r=2$. For the primes $p$ such that $p+1$ divides $|\Aut(A_6)|$, there is an integer $i$ such that $\lcm(i,p^2+1-i)\cdot(p+1)$ divides $|\Aut(A_6)|$ only if $p\in\{5,7,11\}$. For those three $p$'s, there are nonzero singular vectors of three distinct weights by computer check, so $N_{\PGO(5,p)}(A_6)=S_6$ has more than two orbits on the singular points. If $n=7$, then we are in case (1) or (2), and $d\le n-1=6$, $r\le 3$. There is no prime $p$ such that $\frac{p^3-1}{p-1}$ divides $|\Aut(A_7)|$, so we have $r=2$. It follows that the sign is $-$ and so $(-1)^3\cdot7\pmod{p}$ is a nonsquare of $\F_p^*$ for $p\ne 7$. If $p=7$, then the images of the singular vectors of the shapes $0^41^12^13^1$, $0^31^24^16^1$ in $V$ are in distinct $S_7$-orbits.  If $-7$ is nonsquare in $\F_p^*$, i.e., $p\mod{7}\in\{3,5,6\}$, then $\cP=Q^-(5,p)$ and the ovoid number is $p^3+1$. For such primes that $p+1$ divides $|\Aut(A_7)|$, we have $|\Aut(A_7)|\ge\frac{1}{2}|\cP|$ only if $p\in\{3,5\}$. There are respectively three and four $S_7$-orbits for $p=3,5$ by a computer check. This completes the proof.
\end{proof}

\begin{lemma}\label{lem_An_nlt15}
If $V$ is not the fully deleted module of $A_n$ ($n\ge 5$) up to weak equivalence, then we have $n\le 14$.
\end{lemma}

\begin{proof}
Suppose to the contrary that $n\ge 15$. Let $\tilde{A}_n$ be the full covering group of $A_n$ whose center has order $2$, and recall that $k$ is the algebraic closure of $\F_p$. Since $V$ is not the fully deleted module, we have $d\ge n+1$ by \cite{Wagner76,Wagner77}. It follows that $d\ge 16$ and so $r\ge 7$. Let $p'$ be a primitive prime divisor of $p^{fr}-1$.  We have $\textup{ord}_{p'}(p)=fr$, so $p'\ge 1+fr$. Since $\frac{p^{fr}-1}{p-1}$ divides $|\Aut(A_n)|=n!$, we deduce that $p'\le n$. It follows that $d\le 2(r+1)\le 2p'\le 2n$. If $V$ is not a faithful $\tilde{A}_n$-module, i.e., the center of $\tilde{A}_n$ acts trivially on $V$, then we have $d\ge\frac{1}{4}n(n-5)$  by \cite[Theorems~6,~7]{James} which contradicts $d\le 2n$ for $n\ge 15$. If $V$ is a faithful $\tilde{A}_n$-module, then $p\ne 2$ and $V\otimes k$ is a faithful irreducible representation of $\tilde{A}_n$. By \cite{Wales79,KT}, $d$ is divisible by $2^{\lfloor(n-s-1)/2\rfloor}$ and is at least $2^{\lfloor(n-2-[\![ p|n]\!])/2\rfloor}$. Here, $s$ is the $2$-adic weight of $n=2^{i_1}+\cdots+2^{i_s}$ with $i_1<\cdots<i_s$. This is incompatible with the bound $d\le 2n$ for $n\ge 15$. This completes the proof.
\end{proof}

\begin{proposition}\label{prop_An_notFDM}
If $V$ is not the fully deleted module of $A_n$ ($n\ge 5$) up to weak equivalence, then $L=2.A_7$,  $\cP=H(3,3^2)$ and the two $A_7$-orbits on the singular points are $7$- and $21$-tight sets respectively.
\end{proposition}

\begin{proof}
Suppose that $V$ is not the fully deleted module of $A_n$ up to weak equivalence. We have $n\le 14$ by Lemma \ref{lem_An_nlt15}. There is a lower bound $R(A_n)$ on $d$ by \cite[Proposition~5.3.7]{kleidman1990subgroup}, so that we have $r\ge \lceil\frac{R(A_n)}{2}\rceil-1$. For instance, we have $R(A_n)=n-2$ for $n\ge 9$. We follow the approach described in Remark \ref{rem_classS_Smethod}. We call a pair $(q,r)$ feasible if $r\ge \lceil\frac{R(A_n)}{2}\rceil-1$ and $\frac{q^r-1}{q-1}$ divides $|\Aut(A_n)|$. For $n=10$, the feasible $(q,r)$ pairs  are $(2,3)$, $(4,3)$, $(2,4)$, $(3,4)$, $(7,4)$ and $(2,6)$. For $n=11,12$, the feasible $(q,r)$ pairs are $(2,4)$, $(3,4)$, $(7,4)$ and $(2,6)$. For $n=13,14$, the feasible $(q,r)$ pairs are $(2,6)$, $(3,6)$, $(4,6)$ and $(2,12)$. There is no candidate representation with Frobenius-Schur indicator $+$ or $\circ$ whose dimension is in the interval $[2r,2r+2]$ in each case by \cite{HissMalle}, so we assume that $n\le 9$ in the sequel.

For $n=9$, the feasible $(q,r)$ pairs are $(2,3)$, $(4,3)$, $(2,4)$, $(3,4)$ and $(2,6)$. For the three absolutely irreducible representations with $(q,d)=(2,8)$ in the online Atlas \cite{Atlas}, we all have $\cP=Q^+(7,2)$. Both (8b) and (8c) have a transitive $A_9$-action on the singular points, and (8a) corresponds to the fully deleted module. For the absolutely irreducible representation of $2.A_9$ with $(q,d)=(3,8)$, we have $\cP=Q^+(7,3)$ and $2.A_9$ is transitive on the singular points. There is no further candidate representations.

For $n=8$, we have $A_8\cong\PSL_4(2)$. We have $R(A_8)=4$, but the representation that attains this bound has no form. Since $V$ is not the fully deleted module, we have $d\ge 14$ if $p=2$ and $d\ge 8$ if $p\ne 2$. It follows that $r\ge 3$. A feasible $(q,r)$ pair  is one of $(3,4)$ and $(2,6)$. For $(q,r)=(3,4)$, $V$ is a faithful representation of $2.A_8$ of dimension $8$. We have  $\cP=Q^+(7,3)$ with a transitive $A_8$-action on the singular points. For $(q,r)=(2,6)$, $V$ is a $14$-dimensional representation of $\SL_4(2)$ such that $V\otimes k$ has highest weight $\omega_1+\omega_3$ up to weak equivalence. We have $\cP=Q^-(13,2)$ and $S_8=N_{\PGO^-_{14}(2)}(A_8)$ has six orbits on the singular points.

For $n=7$, a feasible $(q,r)$ pair either has $r=2$ or is one of $(2,3)$, $(4,3)$, $(2,4)$, $(3,4)$ and $(2,6)$. There is no candidate representation for $(q,r)\in\{(2,4),(3,4)\}$. For $(q,r)=(2,6)$, there is exactly one $14$-dimensional candidate representation in characteristic $2$ that can be realized over $\F_2$. We have $\cP=Q^-(13,2)$ and $S_7=N_{\PGO^-_{14}(2)}$ has $12$ orbits on the singular points.  For $(q,r)=(4,3)$, there is a candidate representation of dimension $6$  with Frobenius-Schur indicator $\circ$ which is a faithful representation of $3.A_7$. It has a unitary form and the normalizer of $3.A_7$ has more than two orbits on the points of  $\cP$. For $(q,r)=(2,3)$, there is exactly one candidate representation of dimension $6$ with Frobenius-Schur indicator $+$. It turns out that $\cP=Q^+(5,2)$ and $A_7$ is transitive on it. We next consider the case where $r=2$. There is no candidate representation with dimension in the interval $[4,6]$ with a Frobenius-Schur indicator $+$, so we expect $\cP=H(3,q)$ with $q$ a square. If $p=2$, the $4$-dimensional candidate representations with Frobenius-Schur indicators $+$ can all be realized over $\F_2$ and so can not be unitary. The only odd prime power $q$ such that $q$ is a square and $q+1$ divides $|\Aut(A_7)|$ is $q=3^2$. We use the online data in \cite{Atlas} to check that there are two $A_7$-orbits of sizes $70$, $210$ respectively for both representations (8a) and (8b). The two orbits are $7$- and $21$-tight sets of $H(3,9)$ respectively, and we have $L=2.A_7$ in this case.

For $n=6$, we have $A_6=\PSp_4(2)'=\PSL_2(9)$ and $|\Aut(A_6)|=2\cdot 6!$. Its Schur multiplier is cyclic of order $6$, so its projective representation in characteristic $3$ is also a representation of $\SL_2(9)$. A feasible $(q,r)$ pair either has $r=2$ or is one of $(2,4)$ and $(3,4)$. For $(q,r)=(3,4)$, we deduce that $d$ is a square by \cite[Corollary~5.3.3]{bray2013maximal}. It follows that $d=9$. It turns out that $\cP=Q(8,2)$ and the normalizer of $\SL_2(9)$ has more than two orbits. For $(q,r)=(2,4)$, there is no candidate representations with Frobenius-Schur indicator $+$ that can be realized over $\F_2$. We then consider the case where $r=2$. We have either $\cP=Q(4,q)$, $Q^-(5,q)$ or $H(3,q)$. If $q$ is a square and $q+1$ divides $|\Aut(A_6)|$, then $q=4$ or $9$. The $4$-dimensional candidate representations in characteristic $p\in\{2,3\}$ can all be realized over $\F_p$, so $\cP$ is not unitary. For $Q(4,q)$, there is an integer $i$ such that $\lcm(q^2+1-i,i)\cdot(q+1)$ divides $|\Aut(A_6)|$ only  when $q\in\{3,5,7,11\}$; for $Q^-(5,q)$, we  deduce that $q\in\{2,3,5\}$ similarly. There are candidate representations only for $Q(4,q)$ with $q\in\{5,7,11\}$, which are $p$-modular reductions of two irreducible representations over the complex numbers. We check with computer that the normalizer of $A_6$ has more than two orbits on the singular points in each case.

For $n=5$, we have $A_5=\PSL_2(2^2)=\PSL_2(5)$. Similar to the $A_6$ case, its representations in characteristic $2$, $5$ are representations over $\SL_2(4)$, $\SL_2(5)$ respectively. A feasible $(q,r)$ pair either has $r=2$ or is one of $(2,4)$ and $(3,4)$. For $(q,r)=(3,4)$, there is no candidate representation. For $(q,r)=(2,4)$, we deduce that $d$ is a power of $2$ as well as a square by \cite[Section~5.3]{bray2013maximal}. There is no such integer $d$ in the interval $[2r,2r+2]=[8,10]$. We then consider the case where $r=2$. We proceed as in the case of $A_6$ to show that $\cP$ is not unitary and it can only be $Q(4,3)$ or $Q^-(5,2)$. There is no such representation for $2.A_5$ in either case by \cite{HissMalle}. This completes the proof.
\end{proof}

To summarize, we have handled the fully deleted modules in Proposition \ref{prop_S_Anfdm} and the remaining modules in Proposition \ref{prop_An_notFDM} in this subsection. We obtain three rows in Table \ref{tab_TS_cSother} that correspond to $H(3,9)$, $Q^-(9,2)$ and $Q^-(17,2)$ respectively. This completes the analysis of the alternating groups.

\subsection{The groups of Lie type in defining characteristic}\label{subsec_defChar}

Suppose that $S$ is a simple group of Lie type ${}^tX_{\ell}(p^e)$ defined over $\F_{p^e}$. Let $\tilde{S}={}^t\hat{X}_{\ell}(p^e)$ be the covering group of $S$ by the $p'$-part of its Schur multiplier. Write $G={}^tX_\ell(k)$, where $k$ is the algebraic closure of $\F_p$. There is a certain endomorphism $F$ of the algebraic group $G$ such that $\tilde{S}$ consists of elements of $G$ fixed by $F$. We regard $G$ as a matrix group, and let $\phi$ be the field automorphism $x\mapsto x^p$ applied to the entries of $G$. If $t=1$, then $F=\phi^e$. If $t\ge 2$, then $F^t$ is a power of $\phi^e$. In particular, if $X_\ell$ is one of $A_\ell$, $D_\ell$, $E_6$ and $D_4$, the group $G$ has a graph automorphism $\tau_o$ of order $2$, $2$, $2$ and $3$ respectively such that $F=\tau_o^{-1}\phi^e$. In those cases, we also write $\tau_o$ for its restriction to $\tilde{S}$ which agrees with $\phi^e$. The finite dimensional irreducible $k\tilde{S}$-modules are restrictions of irreducible rational $kG$-modules to the action of $k\tilde{S}$, and the irreducible rational $kG$-modules are characterized by their highest weights, cf. \cite{stein63}. We take the same notation as in \cite{LubeckSmalldegree}, which will be our main reference in this subsection. When $X_\ell$ is one of $A_\ell$, $D_\ell$ and $E_6$, we write $\tau$ for the symmetry induced on the Dynkin diagram by the graph automorphism $\tau_o$. For an irreducible $kG$-module $M$, we write $M^{(i)}$ for ${}^{\phi^i}M$ for brevity. If $M=L(\omega)$ is an irreducible module with highest weight $\omega$, then $M^{(i)}=L(p^i\omega)$. If we write $M^*$ for its dual module, then by \cite[3.1.6]{humph} we have
\begin{equation}\label{eqn_dualMod}
L(\omega)^*\cong\begin{cases}L(\tau(\omega)),\quad&\textup{for types $A_\ell$, $D_\ell$ ($\ell$ odd), $E_6$,}\\L(\omega),\quad&\textup{otherwise.}
\end{cases}
\end{equation}

Suppose that $V\otimes k$ is the restriction of the irreducible $kG$-module $L(\omega)$. If either $t=1$ or ${}^tX_\ell$ is one of ${}^2A_\ell$, ${}^2D_\ell$ or ${}^2E_6$, then we can apply \cite[Theorem~5.1.13]{bray2013maximal} to determine $\F_{q}$. We shall make frequent use of the following result which combines \cite[Proposition~5.4.6]{kleidman1990subgroup} and the remark following it.
\begin{thm}\label{thm_classS_Vstruc}
Take notation as above, and write $q=p^f$ with $p$ prime. Let $k$ be the algebraic closure of $\F_p$, and let $\theta$ be the field automorphism $x\mapsto x^p$ of $\F_q$.
\begin{enumerate}
\item[(i)] If $t=1$, then $f\mid e$ and there is an irreducible $k\tilde{S}$-module $M$ such that
    \begin{equation}\label{eqn_VkMTens}
    V\otimes k\cong M\otimes M^{(f)}\otimes\cdots\otimes M^{(e-f)}.
    \end{equation}
    In particular, $\dim(V)=(\dim_kM)^{e/f}$.
\item[(ii)] If $\tilde{S}$ is one of $^2A_\ell(p^e)$, $^2D_\ell(p^e)$, $^2E_6(p^e)$, $^3D_4(p^e)$, then one of the following occurs: (a) $f\mid e$, $V\cong {}^{\tau_o}V$, there is an irreducible $k\tilde{S}$-module $M$ such that $M\cong {}^{\tau_o}M$ and $\dim(V)=(\dim_kM)^{e/f}$; (b) $f\mid te$ but $f$ does not divide $e$, and there is irreducible $k\tilde{S}$-module $M$ such that $M\not\cong {}^{\tau_o}M$ and $\dim(V)=\dim(M)^{te/f}$.
\item[(iii)] If $\tilde{S}$ is one of ${}^2B_2(p^e)$, ${}^2G_2(p^e)$ and ${}^2F_4$, where $e$ is odd and $p=2,3,2$ respectively, then $f\mid e$ and there is irreducible $k\tilde{S}$-module $M$ such that $\dim(V)=(\dim_kM)^{e/f}$. Moreover, we have $\dim_kM\ge R_p(L)$ with $R_p(L)=4,7,26$ respectively.
 \end{enumerate}
\end{thm}

\begin{remark}\label{rem_Sdef_Form}
We now describe how to determine whether $V$ has a nondegenerate form. For some specific modules of dimensions up to $17$, we refer the reader  to \cite{bray2013maximal}, \cite{maxsub_1314} and \cite{maxsub_1617} for this information. Let $V^\ast$ be the dual module of $V$, and let $\theta$ be the field automorphism $x\mapsto x^p$ of $\F_q$. By \cite[Lemma~2.10.15]{kleidman1990subgroup}, $V$ has a nondegenerate symmetric bilinear form only if $V^*\cong V$, and $V$ has a nondegenerate unitary form if and only if $f$ is even and $V^*\cong V^{\theta^{f/2}}$. Suppose that $V\otimes k$ is the restriction of the irreducible $kG$-module $L(\omega)$. If $X_\ell$ is one of $B_\ell$, $C_\ell$, $D_\ell$ ($\ell$ even), $G_2$, $F_4$, $E_7$ and $E_8$, then $V$ has a nondegenerate symmetric or alternating bilinear form by \eqref{eqn_dualMod}. Suppose that $X_\ell$ is one of $A_\ell$, $D_\ell$ ($\ell$ odd) and $E_6$. Then $V^*\otimes k$ is the restriction of $L(\tau(\omega))$ and $V^{\theta^{f/2}}\otimes k$ is the restriction of $L(\omega)^{\theta^{f/2}}=L(p^{f/2}\omega)$. Therefore, to decide whether $V$ has a nondegenerate bilinear or unitary form it suffices to examine the restrictions of the $kG$-modules $L(\omega)$, $L(\tau(\omega))$ and $L(p^{f/2}\omega)$ to $\tilde{S}$. By \cite[\S~13]{stein_end}, the $p^e$-restricted $kG$-modules remain irreducible and inequivalent upon restriction to $\tilde{S}$ for those three types.
\end{remark}

We assume that $\kappa$ is orthogonal or unitary in view of Proposition \ref{prop_classS_caseS}. Recall that $q=p^f$ with $p$ prime, $d=\dim_{\F_q}(V)$, $\cP$ is the associated polar space, $L=H^{\infty}$ and $K$ is its normalizer in $\Gamma(V,\kappa)$. We have $d\ge 4$ by the hypothesis $r\ge 2$. We suppose that $S$ is not one of $\PSL_2(4)$, $\PSL_2(5)$, $\PSL_2(9)$, $\PSL_4(2)$, $\PSp_4(2)'$  which are isomorphic to alternating groups. We suppose that $S$ is not one of $\PSU_2(p^e)$, $\PSp_2(p^e)$, $\POmeg_3(p^e)$ ($p$ odd), $\POmeg_4^-(p^e)$, $\POmeg_5(p^e)$ ($p$ odd), $\POmeg_6^{\pm}(p^e)$, ${}^2G_2(3)'$ which are isomorphic to classical groups of types $A_\ell$ or ${}^2A_\ell$.

\begin{proposition}\label{prop_classS_SL}
If $S=\PSL_{d'}(p^e)$ with $(d',p^e)\not\in\{(2,4),(2,5),(2,9),(4,2)\}$, then $(d',q)=(3,p^e)$, $V$ is the adjoint module of $\SL_3(q)$, and the two $S$-orbits are $(q+1)$- and $(q^3-q)$-tight sets of $\cP=Q(6,q)$ respectively.
\end{proposition}
\begin{proof}
We have $\tilde{S}=\SL_{d'}(p^e)$ and $G=\SL_{d'}(k)$ in this case, and we write $W$ for the natural $\tilde{S}$-module over $\F_{p^e}$. We first show that $V$ is not weakly equivalent to $W$, $S^2(W)$ or $\wedge^2(W)$ if $q=p^e$. Suppose that $q=p^e$. Since $W$ has no form, $V$ is not weakly equivalent to it. We argue as in Remark \ref{rem_Sdef_Form} to deduce that $S^2(W)$ ($p$ odd) has a nondegenerate form only if $d'=2$ in which case it has dimension $3$: a contradiction to $d\ge 4$. Similarly, $\wedge^2(W)$ has a nondegenerate form only if $d'=4$ in which case it has dimension $6$ and is the natural module of $\Omega_6^+(q)\cong\SL_4(q)$. The group $\Omega_6^+(q)$ is transitive on the points of $Q^+(5,q)$. This establishes the claim.

By Theorem \ref{thm_classS_Vstruc} there is an irreducible $k\tilde{S}$-module $M$ such that \eqref{eqn_VkMTens} holds and $d=(\dim_kM)^s$, where $s:=e/f$ is an integer. We have $\dim_kM\ge d'$ by \cite[Theorem~1.11.5]{bray2013maximal}. Also, we have $|\Aut(S)|=\epsilon sfq^{sd'(d'-1)/2}\prod_{i=2}^{d'}(q^{si}-1)$, where $\epsilon=1$ or $2$ according as $d'=2$ or not. It follows that $|\Aut(S)|<\epsilon sfq^{s(d'^2-1)}$. By Lemma \ref{lem_ParaCond} we deduce that
\begin{equation}\label{eqn_tt22}
  d'^s\le d<2+(d'^2-1)s+\log_{q}(2\epsilon sf).
\end{equation}

We first suppose that $d'=2$. All the absolutely irreducible $\SL_2(q)$-modules are self-dual, so we expect $\cP$ to be orthogonal. Since we do not consider $Q^+(3,q)$, we have $d\ge 5$. It follows from \eqref{eqn_tt22} that $s\le 3$ or $(p,f,s)=(2,1,4)$. We divide into four subcases according to the value of $s$. (i) If $s=1$, then we have $d\le 6$ by \eqref{eqn_tt22}. It follows that $\cP$ is one of $Q(4,q)$ ($q$ odd) and $Q^{\pm}(5,q)$. If $\cP=Q^{\pm}(5,q)$, then the condition $|\Aut(S)|\ge\frac{1}{2}|\cP|$ is violated upon routine check. In the case $\cP=Q(4,q)$, we have $d=5$. By \cite[Section~5.3]{bray2013maximal}, we have $p\ge 5$ and $V=S^4(W)$ up to weak equivalence. We take the model of $S^4(W)$ as in \cite[p.~281]{bray2013maximal} and use Lemma 5.3.4 therein to deduce that the $\SL_2(q)$-orbit that contains $\la e_0\ra$ has length $q+1$ and is $\GamL_2(q)$-invariant. However, $\lcm(1,q^2+1-1)\cdot(q+1)=q^2(q+1)$ does not divide $|\Aut(S)|=fq(q^2-1)$, so there are more than two $\GamL_2(q)$-orbits on the singular points. (ii) If $s=2$, then we have $d=(\dim_kM)^2\le 10$ by \eqref{eqn_tt22}. Since $d\ge 5$, it follows that $\dim_kM=3$ and $d=9$. Therefore, $q$ is odd and we expect $\cP=Q(8,q)$ whose rank is $r=4$. We deduce from the second inequality in \eqref{eqn_tt22} that $p^f\le 4f$. It holds only for $(p,f)=(3,1)$, but then we have $|\Aut(S)|<\frac{1}{2}|\cP|$ by direct check: a contradiction. (iii) If $s=3$, then we have $d\le 13$ by \eqref{eqn_tt22}. It follows that $\dim_kM=2$ and $d=8$. We have $\cP=W(7,q)$ or $Q^+(7,q)$ according as $q$ is odd or even by \cite[Proposition~2.6]{schaffer}, so $r=4$ and $q$ is even. The condition that $\frac{q^4-1}{q-1}$ divides $|\Aut(S)|$ reduces to $q^2+1\mid3f(q-1)$ which does not hold: a contradiction. (iv) If $(p,f,s)=(2,1,4)$, then we have $16\le d\le 17$ by \eqref{eqn_tt22}. It follows that $\dim_kM=2$ and $d=16$, and so we expect $\cP=Q^{\pm}(15,2)$. If $\cP=Q^+(15,2)$, then $|\Aut(S)|=16320< \frac{1}{2}|\cP|=16447.5$: a contradiction. If $\cP=Q^-(15,2)$, then $r=7$ and we derive the contradiction that $|\Aut(S)|$ is not divisible by $2^7-1$. To summarize, we obtain no examples for $d'=2$.

We then suppose that $d'=3$. It follows from \eqref{eqn_tt22} that $s\le 3$. We also split into three cases according to the value of $s$. (i) If $s=1$, then $d\le 12$ by \eqref{eqn_tt22}. We deduce from \cite[A.~6]{LubeckSmalldegree} that either $V$ is the adjoint module up to weak equivalence or $\dim(V)=10$. The former case has been handled in Section \ref{subsec_adj}. There are more than two $K$-orbits on the points of $\cP$ for $p\ne 3$, and there are two $S$-orbits that are $(q+1)$- and $(q^3-q)$-tight sets of $\cP=Q(6,q)$ respectively for $p=3$. In the case $\dim(V)=10$, we have $r=4$ or $5$. We derive the contradiction that $\frac{q^r-1}{q-1}$ does not divide $|\Aut(S)|$ in either case upon direct check. (ii) If $s=2$, then $9\le d\le 21$ by \eqref{eqn_tt22}. It follows that $\dim_kM=3$ or $4$. By \cite[A.~6]{LubeckSmalldegree} $M$ is the restriction of $L(\omega_1)$ or $L(\omega_2)$ up to weak equivalence, and correspondingly $V\otimes k$ is the restriction of $L((1+q)\omega_1)$ or $L((1+q)\omega_2)$ by \eqref{eqn_VkMTens}. We deduce that $V$ does not have a nondegenerate form by Remark \ref{rem_Sdef_Form}. (iii) If $s=3$, then $3^3\le d\le 29$ by \eqref{eqn_tt22}. It follows that $\dim_kM=3$, and we deduce that $V$ has no form as in the case $s=2$. To summarize, we obtain the example in the statement of this proposition for $d'=3$.

Finally, we suppose that $d'\ge 4$. We first show that $p^{fr}-1$ has a primitive prime divisor. If $s\ge 2$ then $d\ge 16$ and so the rank $r\ge 7$. If $s=1$, then $V$ is not $W$, $S^2(W)$ or $\wedge^2(W)$ by the first paragraph of the proof. It follows from \cite[Theorem~2.2]{LiebeckRank3} that $d\ge 15$ unless $(d',p)=(4,2)$ and $V$ is the adjoint module of dimension $14$. In the latter case, $V$ has a quadratic form whose type is plus or minus according as $f$ is even or odd by \cite[Corollary~8.4.6]{maxsub_1314}. The normalizer of $\PSL_4(2)$ in $\GL_{14}(2)$ has $6$ orbits on $\cP=Q^-(13,2)$, so we have $q=2^f>2$. We conclude that $fr\ge 7$ in all cases, and so $p^{fr}-1$ has a primitive prime divisor $p'$ by \cite{zsigmondy1892theorie}. Since $\textup{ord}_{p'}(p)=fr$, we have $\gcd(p',fr)=1$. Also, $p'\ge 1+fr\ge \frac{1}{2}d\ge\frac{1}{2}d'^s>s$, so $\gcd(p',s)=1$. Since $p'$ divides $|\Aut(S)|$, we deduce that $fr$ divides one of $\{2sf,\ldots,d'sf\}$. In particular, $r\le d's$. It follows that $d'^s\le d\le 2d's+2$. For $s\ge 2$, it holds only if $(d',s)=(4,2)$ in which case $16\le d\le 18$. It follows that $\dim_kM=4$ and $d=16$. Then $M$ is the restriction of $L(p^i\omega_1)$ or $L(p^i\omega_3)$ for some $i$ and correspondingly $V\otimes k$ is the restriction of $L(p^i(1+q)\omega_1)$ or $L(p^i(1+q)\omega_3)$ by \eqref{eqn_VkMTens}. We deduce that $V$ has no form by Remark \ref{rem_Sdef_Form}. For $s=1$, we have $d'\le d\le 2d'+2\le\binom{d'+1}{2}$. By \cite[Theorem~5.11.5]{kleidman1990subgroup}, we have excluded the candidate representations in the beginning of the proof. This completes the proof.
\end{proof}

\begin{proposition}\label{prop_classS_Sp}
If $S=\textup{PSp}_{d'}(p^e)$ with $d'\ge 4$ and $(d,q)\ne(4,2)$, then $p^e=q$ and we have one of the following:
\begin{enumerate}
\item[(1)]$(d',p)=(6,3)$, $V$ is a section of $\wedge^2(W)$ and the two $S$-orbits are respectively $(q^2+1)$- and $(q^6-q^2)$-tight sets of $\cP=Q(12,q)$, where $W$ is the natural $\Sp_6(q)$-module,
\item[(2)]$(d',p)=(8,2)$, $V$ is the spin module and the two $S$-orbits are respectively $(q^3+1)$- and $(q^7-q^3)$-tight sets of $\cP=Q^+(15,q)$.
\end{enumerate}
\end{proposition}
\begin{proof}
We have $\tilde{S}=\Sp_{d'}(p^e)$ and $G=\Sp_{d'}(k)$ in this case, and we write $W$ for the natural $\tilde{S}$-module over $\F_{p^e}$.  The absolutely irreducible module $V$ is self-dual by \eqref{eqn_dualMod}, so we expect $\cP$ to be orthogonal. It follows that $V$ is not weakly equivalent to $W$ when $q=p^e$. By Theorem \ref{thm_classS_Vstruc}  there is an irreducible $k\tilde{S}$-module $M$ such that \eqref{eqn_VkMTens} holds and $d=(\dim_kM)^s$, where $s:=e/f$ is an integer. We have $\dim_kM\ge d'$ by \cite[Theorem~1.11.5]{bray2013maximal}. Also, we have $|\Aut(S)|=\epsilon sf q^{sd'^2/4}\prod_{i=1}^{d'/2}(q^{2si}-1)$, where $\epsilon=\frac{2}{\gcd(2,q^s-1)}$ or $1$ according as $d'=4$ or not.  It follows that $|\Aut(S)|<\epsilon sf q^{d'(d'+1)s/2}$. By Lemma \ref{lem_ParaCond} we deduce that
\begin{equation}\label{eqn_ttSp}
  d'^s\le d<2+\frac{1}{2}d'(d'+1)s+\log_{q}(4sf).
\end{equation}

We first consider the case $d'=4$. We deduce from \eqref{eqn_ttSp} that $s\le 2$ and we have $16\le d=(\dim_kM)^2<25$ for $s=2$. If $s=2$, then we deduce that $(\dim_kM,d)=(4,16)$ which has been excluded in Section \ref{subsec_twTens}. If $s=1$, then $d<14$ if $p\in\{2,3\}$ and $d\le 12$ otherwise by \eqref{eqn_ttSp}. If $d=4$ or $5$, then $V$ is either the natural module of $\Sp_4(q)$ or $\Omega_5(q)$ ($p$ odd). In both cases $S$ is transitive on the singular points, so we have $d\ge 6$. By \cite[A.~22]{LubeckSmalldegree}, we deduce that either $V=S^2(W)$ ($p$ odd) up to weak equivalence, or $(p,d)=(5,12)$ and $V\otimes k$ is the restriction of $L(5^i(\omega_1+\omega_2))$ for some $i$. The former case has been excluded in Section \ref{subsec_S2W2W}, and the latter case is symplectic by \cite[Table~5.6]{bray2013maximal}. To summarize, there is no example for $d'=4$.

We suppose that $d'\ge 6$ in the sequel. We claim that $p^{fr}-1$ has a primitive prime divisor. By \cite{zsigmondy1892theorie} the claim is true provided one of the following holds: (1) $d\ge 15$, so that $r\ge 7$, (2) $p$ is odd and $d\ge 7$, so that $r\ge 3$, (3) $V$ is the spin module, so that $r=2^{d'/2-1}\ge 4$ and is a power of $2$. By \cite[Proposition~5.4.11]{kleidman1990subgroup}, it remains to consider the case where $(d',p)=(6,2)$ and $V$ is a section of $\wedge^2(W)$ of dimension $14$.
We have $r\ge 6$, and so $p^{fr}=2^6$ if and only if $(p^f,r)=(2,6)$. The normalizer of $\Sp_6(2)$ in $\GL_{14}(2)$ has three orbits on $\cP=Q^-(13,2)$, so $p^{fr}\ne 2^6$. We thus have $fr\ge 12$ and so $p^{fr}-1$ has a primitive prime divisor by \cite{zsigmondy1892theorie}.

Let $p'$ be a primitive prime divisor of $p^{fr}-1$. It is routine to show that $\gcd(p,2fs)=1$ as in the proof of Proposition \ref{prop_classS_SL}, and we deduce from the fact $p'$ divides $|\Aut(S)|$ that $r\le d's$. It follows that $d'^s\le d\le 2r+2\le 2d's+2$. It follows that $s=1$, and by \cite[Proposition~5.4.11]{kleidman1990subgroup} we deduce that up to weak equivalence $V$ is one of the following modules: (a) a section of $\wedge^2(W)$ with $(d',d)=(6,14-[\![ p=3]\!])$, (b) a spin module of dimension $2^{d'/2}$ with $d'\in\{6,8\}$ and $p$ odd, (c) a section of $\wedge^3(W)$ with $(d',d)=(6,14)$ and $p=2$. The case (a) has been considered in Section \ref{subsec_S2W2W}: there are two $S$-orbits which are $(q^2+1)$- and $(q^6-q^2)$-tight sets of $\cP=Q(12,q)$ for $p=3$ and more than three $K$-orbits for $p\ne 3$. For (b), the group $S$ is transitive on $\cP=Q^+(7,q)$ by \cite[Theorem~8.4]{giudici2020subgroups} if $d'=6$, and $S$ has exactly two orbits on the points of  $\cP$ which are respectively $(q^3+1)$- and $(q^7-q^3)$-tight sets if $d'=8$ by Section \ref{subsec_spin}. For (c), $\cP$ is symplectic by \cite[Proposition~9.3.5]{maxsub_1314}. This completes the proof.
\end{proof}

\begin{proposition}\label{prop_classS_Omeg}
If $S=\textup{P}\Omega_{d'}(p^e)$ with $pd'$ odd and $d'\ge 7$, then $(d',p^e)=(9,q)$, $\cP=Q^+(15,q)$, and the two $S$-orbits are $(q^3+1)$-  and $(q^7-q^3)$-tight sets respectively.
\end{proposition}
\begin{proof}
Write $\ell=\frac{d'-1}{2}$, and let $W$ be the natural $\Omega_{d'}(p^e)$-module over $\F_{p^e}$. We have $\tilde{S}=\textup{Spin}_{d'}(p^e)$ and $G=\textup{Spin}_{d'}(k)$.   If $q=p^e$, then $V$ is not weakly equivalent to $W$ since $\Omega_{d'}(q)$ is transitive on the singular points of $Q(8,q)$ ($p$ odd). By Theorem \ref{thm_classS_Vstruc} there is an irreducible $k\tilde{S}$-module $M$ such that \eqref{eqn_VkMTens} holds and $d=(\dim_kM)^s$, where $s:=e/f$ is an integer. We have $\dim_kM\ge d'$ by \cite[Theorem~1.11.5]{bray2013maximal} and the assumption $d'\ge 7$. Also, we have $|\Aut(S)|=sfq^{s\ell^2}\prod_{i=1}^{\ell}(q^{2si}-1)$. The module $V$ is self dual by \ref{eqn_dualMod}, so we expect $\cP$ to be orthogonal. For $s\ge 2$, we have $d\ge 49$. For $s=1$, we have $d\ge 15$ unless $V$ is weakly equivalent to the spin module of dimension $2^{3}$ with $d'=7$ by \cite[Proposition~5.4.11]{kleidman1990subgroup}. In the latter case, $\cP=Q^+(7,q)$ and the rank $r=4$ by \cite[Proposition~5.4.9]{kleidman1990subgroup}. We deduce that $p^{fr}-1$ has a primitive prime divisor $p'$ in all cases by \cite{zsigmondy1892theorie}. We have $\gcd(p',2fs)=1$ as usual and deduce from $p'$ divides $|\Aut(S)|$ that $r\le (d'-1)s$. It follows that $d'^s\le d\le2r+2\le 2(d'-1)s+2$, which holds only if $s=1$. Then $d\le 2d'$, and we deduce that $V$ is weakly equivalent to the spin module with $d'\in\{7,9\}$ by \cite[Proposition~5.4.11]{kleidman1990subgroup}. By Section \ref{subsec_spin}, $\POmeg_7(q)$ is transitive on $\cP=Q^+(7,q)$ and $\POmeg_9(q)$ has two orbits on $\cP=Q^+(15,q)$ which are $(q^3+1)$-  and $(q^7-q^3)$-tight sets respectively. This completes the proof.
\end{proof}

\begin{proposition}\label{prop_classS_OmegP}
We can not have $S=\textup{P}\Omega^+_{d'}(p^e)$ with $d'$ even and $d'\ge 8$.
\end{proposition}
\begin{proof}
Suppose to the contrary that $S=\textup{P}\Omega^+_{d'}(p^e)$ with $d'$ even and $d'\ge 8$. Write $\ell=\frac{1}{2}d'$, and let $W$ be the natural $\Omega^+_{d'}(p^e)$-module over $\F_{p^e}$. We have $\tilde{S}=\textup{Spin}_{d'}^+(p^e)$ and $G=\textup{Spin}(k)$. If $q=p^e$ then $V$ is not weakly equivalent to $W$ since $\Omega^+_{d'}(q)$ is transitive on $Q^+(d'-1,q)$. By Theorem \ref{thm_classS_Vstruc} there is an irreducible $k\tilde{S}$-module $M$ such that \eqref{eqn_VkMTens} holds and $d=(\dim_kM)^s$, where $s:=e/f$ is an integer. We have $\dim_kM\ge d'$ by \cite[Theorem~1.11.5]{bray2013maximal}. Also, we have  $|\Aut(S)|=2f\epsilon q^{s\ell(\ell-1)}(q^{s\ell}-1)\prod_{i=1}^{\ell-1}(q^{2si}-1)$, where $\epsilon=3$ or $1$ according as $d'=8$ or not. For $s\ge 2$, we have $d\ge 64$. For $s=1$, by \cite[Proposition~5.4.11]{kleidman1990subgroup} we have $d\ge 15$ unless $\ell=4$ and $V=W$ up to weak equivalence which was excluded. We thus always have $r\ge 7$.
It follows that $p^{fr}-1$ has a primitive prime divisor $p'$ in all cases by \cite{zsigmondy1892theorie}. It is routine to show that $\gcd(p',6fs)=1$ as usual and we deduce from $p'$ divides $|\Aut(S)|$ that $r\le (d'-2)s$.  We thus have $d'^s\le d\le (2d'-4)s+2$, which holds only if $s=1$. Since $d\le 2d'-2$, we deduce that $d'\in\{8,10\}$ and $V$ is the half spin module up to weak equivalence by \cite[Proposition~5.4.11]{kleidman1990subgroup}. If $d'=10$, then $V$ has no form by \cite[Table~4.1]{maxsub_1617}. If $d'=8$, then $V$ is weakly equivalent to $W$ which was excluded. This completes the proof.
\end{proof}

\begin{lemma}
We can not have $S=\PSU_{d'}(p^e)$ with $d'\ge 3$.
\end{lemma}

\begin{proof}
Suppose to the contrary that $S=\PSU_{d'}(p^e)$ with $d'\ge 3$. We have $\tilde{S}=\SU_{d'}(p^e)$, $G=\SL_{d'}(k)$, and $|\Aut(S)|= 2ep^{ed'(d'-1)}\prod_{i=2}^{d'}(p^{ei}-(-1)^i)$. Let $W$ be the natural $\SU_{d'}(p^e)$-module over $\F_{p^{2e}}$, and set $W_0=W\otimes k$. We take an orthonormal basis $e_1,\ldots,e_{d'}$ of $W$ with respect to its unitary form. We can not have $V\otimes k=W_0$, since otherwise $\F_q=\F_{p^{2e}}$, $V=W$ and $S$ is transitive on the singular points of $H(d-1,q)$. If $d'=6$ and $V\otimes k$ is the restriction of the $kG$-module $L(p^i\omega_3)$ for some integer $i$, then we have $q=p^e$ by \cite[Theorem~5.1.13]{bray2013maximal} and $V=\wedge^3(W)$ up to weak equivalence. Here, the restriction of ${}^{\phi^e}L(p^i\omega_3)$ as a $k\tilde{S}$-module is the same as the restriction of $L(p^i\tau(\omega_3))=L(p^i\omega_3)$. The space $V$ has a symplectic form if $p$ is odd by \cite[A.~3]{LubeckSmalldegree} and has a quadratic form if $p=2$, and the latter has been excluded in Section \ref{subsec_wedge3}.

We claim that $V\otimes k$ is not $\wedge^2(W_0)$ up to weak equivalence. Suppose to the contrary that $V\otimes k=\wedge^2(W_0)$, so that $V\otimes k$ is the restriction of the $kG$-module $L(\omega_{d'-1})$. We have $d'\ge 4$ by the fact $d=\binom{d'}{2}\ge 5$. We use \cite[Theorem~5.1.13]{bray2013maximal} to determine $\F_q$, the smallest field over which $\wedge^2(W_0)$ can be realized. For $d'=4$ we have $\F_q=\F_{p^e}$, $V$ is the natural module of $\Omega_6^-(p^e)$ and $S$ is transitive on the singular points of $Q^-(5,q)$. For $d'\ge 5$ we have $\F_q=\F_{p^{2e}}$ and so $V=\wedge^2(W)$. The space $V$ has a nondegenerate unitary form $\kappa$ by \cite[Proposition~5.2.4]{bray2013maximal}. We have $d=\binom{d'}{2}$, and  $r=\lfloor\frac{d}{2}\rfloor\ge 5$. Since $f=2e$ and $r\ge5$, $p^{fr}-1$ has a primitive prime divisor $p'$. We have $p'\ge 1+fr>2e$, so $\gcd(p',2e)=1$. Since $p'$ divides $|\Aut(S)|$, we deduce that $fr\le 2ed'$, i.e., $r\le d'$. It follows from $d\le 2r+2$ that $\binom{d'}{2}\le 2d'+2$ which holds only if $d'=5$. For $d'=5$, we have $d=10$ and $r=5$. The condition that $\frac{q^5-1}{q-1}$ divides $|\Aut(S)|$ simplifies to $(p^{4e}+\cdots+p^e+1)\mid 2e(4p^{3e}+p^{2e}+p^e+4)$ which never holds: a contradiction. This establishes the claim.

We claim that $V\otimes k$ is not $S^2(W_0)$ up to weak equivalence for odd $p$. Suppose to the contrary that $V\otimes k=S^2(W_0)$, so that it is the restriction of the $kG$-module $L(2\omega_{d'})$. Similar to the $\wedge^2(W_0)$ case, we deduce that $f=2e$ and so $V=S^2(W)$. The space $V$ has a nondegenerate unitary form $\kappa$ by \cite[Proposition~5.2.4]{bray2013maximal}. We have $d=\binom{d'+1}{2}$, and  $r=\lfloor\frac{d}{2}\rfloor\ge 3$. For $d'=3$, we have $(d,r)=(6,3)$, and the condition that $\frac{q^3-1}{q-1}$ divides $|\Aut(S)|$ simplifies to $(p^{2e}+p^e+1)\mid 2e(2p^e+1)$ which never holds. For $d'=4$ or $5$, we have $(d,r)=(10,5)$ or $(15,7)$ respectively and similarly $\frac{q^r-1}{q-1}$ does not divide $|\Aut(S)|$. Therefore, we have $d'\ge 6$. It follows that $d\ge 21$ and so $r\ge 10$. We deduce from \cite{zsigmondy1892theorie} that $p^{fr}-1$ has a primitive prime divisor $p'$, and we have $p'\ge 1+fr>2e$. Since $p'$ divides $|\Aut(S)|$, we deduce that $r\le d'-[\![2\mid d']\!]$. It follows from $d\le 2r+2$ that $\binom{d'+1}{2}\le 2d'+2-2\cdot[\![2\mid d']\!]$. It does not hold for $d'\ge 6$: a contradiction. This establishes the claim.

By Theorem \ref{thm_classS_Vstruc} there is an integer $s$ and an irreducible $k\tilde{S}$-module $M$ such that $d=(\dim_kM)^s$. Moreover, $e\in\{fs,\frac{1}{2}fs\}$, and if $e=\frac{1}{2}fs$ then $s$ is odd. We have $\dim_kM\ge d'$ by \cite[Theorem~1.11.5]{bray2013maximal}, and $d\ge \binom{d'+1}{2}+1$ by \cite[Proposition~5.4.11]{kleidman1990subgroup} and the preceding arguments. In particular, we have  $d\ge 7$ and so $r\ge 3$. We claim that $p^{fr}-1$ has a primitive prime divisor. It suffices to show that $p^{fr}\ne 2^6$ by \cite{zsigmondy1892theorie}. For $s\ge2$, we have either $d=9$ or $d\ge 16$. It follows that $r=4$ or $r\ge 7$, and so the claim holds for $s\ge 2$. Suppose that $s=1$ and $p^{fr}=2^6$. Since $r\ge 3$, we have  $(p^f,r)=(2^2,3)$ or $(2,6)$. For $(p^f,r)=(2^2,3)$, we deduce from
$\binom{d'+1}{2}+1\le d\le 2r+2$ that $d'=3$ in which case $7\le d\le 8$. We deduce that $e=2$, since $e\in\{1,2\}$ and $\PSU_3(2)$ is soluble. Then by \cite[A.~6]{LubeckSmalldegree} $V$ is the adjoint module of $\SU_3(4)$ of dimension $8$ which was excluded in Section \ref{subsec_adj}. For $(p^f,r)=(2,6)$ we have $e=1$. The condition $\binom{d'+1}{2}+1\le d\le 2r+2$ implies $d'=4$, and correspondingly $11\le d\le 14$. We deduce that  $V$ is the adjoint module of dimension $14$ from \cite[A.~7]{LubeckSmalldegree} which was excluded in Section \ref{subsec_adj}. Therefore, $p^{fr}\ne 2^6$  and the claim also holds for $s=1$.

Let $p'$ be a primitive prime divisor of $p^{fr}-1$.  We have $\gcd(p',2fs)=1$ as in the proof of Proposition \ref{prop_classS_SL}, and we deduce from $p'$ divides $|\Aut(S)|$ that $fr\le 2e(d'-[\![2\mid d']\!])$. Since $d\le 2r+2$, we have
\begin{equation}\label{eqn_tt11}
\max\left\{d'^s,\binom{d'+1}{2}+1\right\}\le d\le 2+\frac{4e}{f}\cdot(d'-[\![2\mid d']\!]).
\end{equation}
If $e=\frac{1}{2}fs$ with $s$ odd, then \eqref{eqn_tt11} holds only for $(d',s)=(3,1)$ and correspondingly $d\in\{7,8\}$. We deduce from \cite[A.~6]{LubeckSmalldegree} that $V\otimes k$ is the restriction of $L(p^i(\omega_1+\omega_2))$ for some $i$, and it can be realized over $\F_{p^e}$ by \cite[Theorem~5.1.13]{bray2013maximal}: a contradiction to $e=\frac{1}{2}f$. If $e=fs$, then \eqref{eqn_tt11} holds only if $(d',s)=(3,3)$, or $s\le 2$ and $d'\le 7$. By Theorem \ref{thm_classS_Vstruc} (ii), we have $M={}^{\tau_o}M$ and so  $\dim_kM\ge d'^2-1-[\![ p\mid d']\!]$  for $3\le d'\le 7$ by \cite{LubeckSmalldegree}. Therefore, we improve \eqref{eqn_tt11} to $(d'^2-2)^s\le d\le 4(d'-[\![ 2\mid d']\!])s+2$ which holds only if $(d',s)=(3,1)$ or $(4,1)$. In both cases, we have $d\le 14$. We deduce from \cite{LubeckSmalldegree} that up to weak equivalence $V$ is the adjoint module, which was excluded in Section \ref{subsec_adj}. This completes the proof.
\end{proof}

\begin{lemma}
If $S=\textup{P}\Omega_{d'}^-(p^e)$ with $d'\ge 8$, then $(d',p^e)=(8,q^{1/2})$, $\cP=Q^+(7,q)$ and the two $S$-orbits are respectively $(q^{3/2}+1)$- and $(q^3-q^{3/2})$-tight sets.
\end{lemma}
\begin{proof}
Write $\ell=d'/2\ge 4$, and let $W$ be the natural module of $\Omega_{d'}(p^e)$ over $\F_{p^e}$. We have $\tilde{S}=\textup{Spin}^-_{d'}(p^e)$, $G=\textup{Spin}_{d'}(k)$, and $|\Aut(S)|=2ep^{e\ell(\ell-1)}(p^{e\ell}+1)\prod_{i=1}^{\ell-1}(p^{2ei}-1)$. If $V=W$, then $q=p^e$ and $S$ is transitive on the singular points of $Q^-(7,q)$.  If $V$ is the spin module of dimension $2^{\ell-1}$ for $\tilde{S}$, then we have $q=p^{2e}$ and $r=2^{\ell-2}$ by \cite[Proposition~5.4.9 (iii)]{kleidman1990subgroup}. Since $r$ is a power of $2$ and $r\ge 4$, $p^{fr}-1$ has a primitive prime divisor $p'$. We have $\gcd(p',2e)=1$ as usual and deduce from $p'$ divides $|\Aut(S)|$ that $fr\le 2e\ell$. It follows that $r=2^{\ell-2}\le \ell$ which holds only if $\ell=4$. The latter case was handled in Section \ref{subsec_spin}: there are two $S$-orbits which are respectively $(q^{3/2}+1)$- and $(q^3-q^{3/2})$-tight sets of $\cP=Q^+(7,q)$.

We assume that $V$ is neither the natural module for  $\Omega_{d'}(p^e)$ nor the spin module of dimension $2^{\ell-1}$ for $\tilde{S}$ in the sequel. We deduce that $d\ge \binom{d'}{2}-1$ by \cite[Proposition~5.4.11]{kleidman1990subgroup}. In particular, $d\ge 27$ and so the rank $r\ge 13$. Let $p'$ be a primitive prime divisor of $p^{fr}-1$, cf. \cite{zsigmondy1892theorie}. By Theorem \ref{thm_classS_Vstruc} there is an integer $s$ and an irreducible $k\tilde{S}$-module $M$ such that $d=(\dim_kM)^s$. Moreover, $e\in\{fs,\frac{1}{2}fs\}$, and if $e=\frac{1}{2}fs$ then $s$ is odd. We have $d=(\dim_kM)^s\ge d'^s$ by \cite[Theorem~1.11.5]{bray2013maximal}. We have $\gcd(p',2fs)=1$ as in the proof of Proposition \ref{prop_classS_SL} and deduce from $p'$ divides $|\Aut(S)|$ that $fr\le 2e\ell$. Since $e\le fs$, it follows that
$\max\left\{d'^s,\frac{1}{2}d'(d'-1)-1\right\}\le d\le 2+2d's$. It holds for no $(d',s)$ pairs with $d'$ even and $d'\ge 8$: a contradiction. This completes the proof.
\end{proof}

\begin{lemma}\label{lem_S_def_2B2}
The group $S$ is not one of ${}^2B_2(p^{2a+1})$ with $2a+1>1$, ${}^2G_2(p^{2a+1})$ with $2a+1>1$ or ${}^2F_4(p^{2a+1})'$, where $p=2,3,2$ respectively.
\end{lemma}

\begin{proof}
Suppose to the contrary that $S$ is one of the three groups. By Theorem \ref{thm_classS_Vstruc} there is an irreducible $k\tilde{S}$-module $M$ such that $d=(\dim_kM)^s$, where $s:=\frac{2a+1}{f}$ is an odd integer. For the type ${}^2B_2$, we have $\dim_kM=4^a$ for an integer $a$ by \cite[Theorem 12.2]{stein63} and an irreducible $4$-dimensional module has a nondegenerate alternating form.
It follows that $d\ge 16$ and so $r\ge 7$ for ${}^2B_2$.  For the type ${}^2G_2$, there are nonnegative integers $b,c$ such that $\dim_kM=7^b27^c$ respectively. For the type ${}^2F_4$, we have $\dim_kM\ge 26$. We thus have $d\ge7$, $26$ and correspondingly $r\ge 3$, $12$ for ${}^2G_2$ and ${}^2F_4$ respectively. We deduce that $p^{fr}-1$ has a primitive prime divisor $p'$ in all three cases by \cite{zsigmondy1892theorie}. We have $\gcd(p',fs)=1$ as usual and we deduce from $p'$ divides $|\Aut(S)|$ that $fr$ is upper bounded by $4fs$, $6fs$, $6fs$ for  the types ${}^2B_2$, ${}^2G_2$ and ${}^2F_4$ respectively. For $S={}^2B_2(p^{2a+1})$, we
deduce from $4^{as}\le d\le 2+2r\le 2+8s$ that $a=s=1$ which contradicts $d\ge 16$. For $S={}^2G_2(p^{2a+1})$, we similarly deduce that $\dim_kM=7$ and $s=1$, i.e., $f=2a+1$. It follows that $d=7$ and so $r=3$. The condition that $\frac{q^3-1}{q-1}$ divides $|\Aut(S)|$ simplifies to $(3^{2f}+3^f+1)\mid f(3^f-1)$ which does not hold: a contradiction. For $S={}^2F_4(p^{2a+1})'$, we have $26^s\le d\le 2+12s$ which never holds: a contradiction. This completes the proof.
\end{proof}

\begin{lemma}
The group $S$ is not one of $E_6(p^e)$, $E_7(p^e)$, $E_8(p^e)$, ${}^2E_6(p^e)$ or ${}^3D_4(p^e)$.
\end{lemma}
\begin{proof}
First suppose that $S$ is one of $E_6(p^e)$, $E_7(p^e)$ and $E_8(p^e)$. By Theorem \ref{thm_classS_Vstruc} there is an irreducible $k\tilde{S}$-module $M$ such that $d=(\dim_kM)^s$, where $s:=\frac{e}{f}$ is an integer. We have $\dim_kM\ge R_p(S)$ by \cite{LubeckSmalldegree}, where $R_p(S)=27$, $56$, $248$ for $E_6$, $E_7$ and $E_8$ respectively. It follows that $d\ge 27$ and so $r\ge 13$ in each case. Then $p^{fr}-1$ has a primitive prime divisor $p'$ by \cite{zsigmondy1892theorie}. We have $\gcd(p',2fs)=1$ as in the proof of Proposition \ref{prop_classS_SL}, and we deduce from $p'$ divides $|\Aut(S)|$ that  $r\le 12s$, $18s$, $30s$ for $E_6$, $E_7$ and $E_8$ respectively. For $E_6$ we deduce from $d\le 2r+2$ that $R_p(S)^s\le 24s+2$ which never holds. We derive similar contradictions for $E_7$ and $E_8$.

Next suppose that $S={}^tX_\ell(p^e)$ with ${}^tX_\ell={}^2E_6$ or ${}^3D_4$. By Theorem \ref{thm_classS_Vstruc}, there is an integer $s$ and an irreducible $k\tilde{S}$-module $M$ such that $d=(\dim_kM)^s$. Moreover, $e\in\{fs,\frac{1}{t}fs\}$, and if $e=\frac{1}{t}fs$ then the prime $t$ divides $f$ but not $s$. For ${}^2E_6$ we have $\dim_kM\ge 27$ and so $r\ge 13$. For ${}^3D_4$ we have either $\dim_kM=8$ or $\dim_kM\ge 26$ by \cite[A.~41]{LubeckSmalldegree}. If $S={}^3D_4(p^e)$ and $d=8$, then $q=p^{3e}$ and $\cP=Q^+(7,p^{3e})$ by \cite[Theorem~5.6.1]{bray2013maximal}. We thus have either $r=4$ or $r\ge \frac{d}{2}-1\ge 12$ for ${}^3D_4$. Therefore, $p^{fr}-1$ has a primitive prime divisor $p'$ in all cases by \cite{zsigmondy1892theorie}. We have $\gcd(p',tfs)=1$ as usual and deduce from $p'$ divides $|\Aut(S)|$ that: $fr$ divides one of $\{8e,10e,12e,18e\}$ if $S={}^2E_6(p^e)$, and $fr$ divides $12e$ if $S={}^3D_4(p^e)$. For $S={}^2E_6(p^e)$, we have $27^s\le d\le 2+2r\le 2+\frac{36e}{f}$. Since $e\in\{fs,\frac{1}{2}fs\}$, we deduce that $e=f$ and $s=1$. It follows that $27\le d\le 38$, and we have $d=27$ by \cite[A.~51]{LubeckSmalldegree}. Then $r=13$ which divides neither of $\{8,10,12,18\}$: a contradiction. For $S={}^3D_4(p^e)$, we have $8^s\le d\le 2+\frac{24e}{f}$. Since $e\in\{fs,\frac{1}{3}fs\}$, we deduce that $s=1$. If $e=\frac{1}{3}f$, then $r\mid 4$ and $8\le d\le 10$. It follows that $r=4$. The condition that $\frac{q^r-1}{q-1}$ divides $|\Aut(S)|$ simplifies to $(p^{2e}+1)\mid 6e(p^e+1)$ which never holds. If $e=f$, then $r\mid 12$ and $8\le d\le 26$. We have $M={}^{\tau_o}M$ by Theorem \ref{thm_classS_Vstruc} and so $\dim_kM\ge 26$ by \cite[A.~41]{LubeckSmalldegree}. It follows that $d=26$ and $r=12$. The condition that $\frac{q^r-1}{q-1}$ divides $|\Aut(S)|$ simplifies to $(p^{2f}+1)(p^f+1)\mid 6f$ which never holds. This completes the proof.
\end{proof}

\begin{lemma}
If $S=F_4(p^e)$, then $q=3^e$, $\cP=Q(24,q)$ and the two $S$-orbits are respectively $(q^4+1)$- and $(q^{12}-q^4)$-tight sets.
\end{lemma}
\begin{proof}
We refer the reader to \cite{Cohen1988} for the case $p=3$ and \cite[Remark~3.1]{TS_Albert} for the case $p\ne3$.
\end{proof}

\begin{lemma}\label{lem_S_def_G2}
The group $S$ can not be $G_2(p^e)'$.
\end{lemma}
\begin{proof}
By Theorem \ref{thm_classS_Vstruc} there is an irreducible $k\tilde{S}$-module $M$ such that $d=(\dim_kM)^s$, where $s:=\frac{e}{f}$ is an integer. We have $\dim_kM\ge 7-[\![ p=2]\!]$ by \cite[A.~49]{LubeckSmalldegree}. We first show that $p^{fr}-1$ has a primitive prime divisor. For $s\ge 2$ we have $d\ge 6^2$ and so $r\ge 17$, so the claim holds in this case. Suppose that $s=1$.  If $d=7-[\![ p=2]\!]$, then there is exactly one such candidate module up to weak equivalence. We have $q=p^e$, $\cP=W(6,q)$ for $p=2$ and $\cP=Q(6,q)$ for odd $p$ by \cite[Propositions~5.7.1,~5.7.2]{bray2013maximal}. The group $S$ is transitive on $\cP=Q(6,q)$ if $p$ is odd by \cite[p.~125]{wilson2009finite}. Therefore, we have $d>7-[\![ p=2]\!]$. It follows that $d=\dim_kM\ge 14$ for $p\ne 3$ and $d=\dim_kM\ge 27$ for $p=3$ by \cite[A.~49]{LubeckSmalldegree}. We have $r\ge 6$ in both cases.  If $p^{fr}=2^6$, then we deduce that $(p^f,r)=(2,6)$. It follows that $(p^e,d)=(2,14)$, and we use the Atlas data \cite{Atlas} to check that there are more than two $K$-orbits on the points of $\cP=Q^-(13,2)$. Therefore, $p^{fr}-1$ also has a primitive prime divisor for $s=1$ by \cite{zsigmondy1892theorie}. This establishes the claim.

Let $p'$ be a primitive prime divisor of $p^{fr}-1$. We have $|\Aut(S)|=\epsilon eq^{6}(q^{2s}-1)(q^{6s}-1)$, where $\epsilon=2$ or $1$ according as $p=3$ or not.
As usual, we have $\gcd(p',2fs)=1$ and deduce from the condition $p'$ divides $|\Aut(S)|$ that $fr\mid 6fs$, i.e., $r\mid 6s$.  Since $d\le 2r+2$, we have $\max\{6^s,14\}\le d\le 2+12s$. It follows that $(s,d)=(1,14)$, and so $p\ne 3$. By \cite[Proposition~9.3.7]{maxsub_1314}, we have $\cP=Q^{\pm}(13,q)$ where the sign $\pm$ depends on $q$. Since $r\mid 6$, we deduce that $r=6$ and so we expect $\cP=Q^-(13,q)$. We claim that there is no integer $i$ such that $\textup{lcm}(i,q^7+1-i)\cdot\frac{q^6-1}{q-1}$ divides $|\Aut(S)|$, i.e.,  $\textup{lcm}(i,q^7+1-i)$ divides $fq^{6}(q^{2}-1)(q-1)$. Suppose to the contrary that there is such an integer $i$. It is clear that $i\ne 1,q^7$. If neither of $i,q^7+1-i$ is a multiple of $p$, then $\textup{lcm}(i,q^7+1-i)$ divides $f(q^{2}-1)(q-1)$. We deduce that $f(q^{2}-1)(q-1)\ge\frac{1}{2}(q^7+1)$ which never holds. We assume without loss of generality that $i=up^t$ with $p\nmid u$ and $t>0$. It follows that $\lcm(u,q^7+1-up^t)$ divides $f(q^{2}-1)(q-1)$. We have  $p^t\le q^7-up^t\le f(q^{2}-1)(q-1)-1<fq^3$, i.e., $1+p^t\le fq^3$. Since both $u$ and $q^7+1-up^t$ are at most $f(q^{2}-1)(q-1)$, the sum of $up^t$ and $q^7+1-up^t$, namely $q^7+1$, is upper bounded by $(1+p^t)f(q^{2}-1)(q-1)$. Since $1+p^t\le fq^3$, we deduce that $q^7+1\le f^2q^3(q^{2}-1)(q-1)$ which never holds. This completes the proof.
\end{proof}

To summarize, we have handled the simple groups of Lie type in the defining characteristic in this subsection. We obtain the five infinite families of tight sets in Table \ref{tab_TS_cSLiedef}.

\subsection{The groups of Lie type in cross characteristic and sporadic groups}\label{susbec_cross}

We first suppose that $S$ is a simple group of Lie type ${}^tX_\ell(\F_{q_1})$, where $q_1=p_1^e$ with $p_1$ a prime distinct from $p$. We assume that $S$ is not one of $\PSL_2(4)=\PSL_2(5)$, $\PSL_2(9)$, $\PSL_4(2)$, $\PSp_4(2)'$ (these groups are isomorphic to alternating groups, which were handled previously). If $S=\PSL_2(7)=\PSL_3(2)$, then we assume that $p_1\ne 2,7$; if $S=\PSp_4(3)=\PSU_4(2)$, then we assume that $p_1\ne 2,3$; if $S=\PSU_3(3)=G_2(2)'$, then we assume that $p_1\ne 2,3$. We follow the strategy as outlined in Remark \ref{rem_classS_Smethod}. By Lemma \ref{lem_ParaCond}, the dimension $d$ of $V$ is upper bounded by
\begin{equation}\label{eqn_dmax}
    d_{\max}:=2+\left\lfloor\log_{p_2}(2\cdot|\Aut(S)|)\right\rfloor,
\end{equation}
where $p_2$ is the smallest prime distinct from $p_1$. The value $d_{\max}$ is roughly polynomial in $\ell$ and linear in $\ln(q_1)$. There is a lower bound $e(S)$ on $d$ by \cite{La-Se} which is roughly polynomial in $q_1^\ell$, see also \cite[Theorem~5.3.9]{kleidman1990subgroup}. There has been extensive work that improves the bound in the literature, cf. \cite{cross_bound} and the references therein, but the bound $e(S)$ in \cite{La-Se} suffices for our purpose. It is straightforward to check that the condition $d_{\max}\ge e(S)$ holds only in the following cases: (A) $S=\PSL_2(q_1)$, $q_1\in\{7,11,13\}$, (B) $S=\PSL_n(q_1)$, $(n,q_1)\in\{(3,3),(3,4)\}$, (C) $S=\PSp_{2m}(q_1)$, $(m,q_1)\in\{(2,3),(3,2),(3,3)\}$, (D) $S=\PSU_n(q_1)$, $(n,q_1)\in\{(3,3),(3,4),(3,5),(4,3), (5,2),(6,2)\}$, (E) $S=\POmeg_{8}^+(2)$ or $\POmeg_{7}(3)$, (F) $S=G_2(q_1)'$, $q_1\in\{3,4\}$. In particular, $S$ is not one of $\POmeg^-_{2m}(q_1)$ ($m\ge 4$), $E_6(q_1)$, $E_7(q_1)$, $E_8(q_1)$, $F_4(q_1)$, $^2E_6(q_1)$, ${}^3D_4(q_1)$, ${}^2F_4(q_1)$, ${}^2B_2(q_1)$ or ${}^2G_2(q_1)$. For the surviving cases, there are finitely many prime power $q$'s such that $d_{\max}'\ge e(S)$, where
\begin{equation}\label{eqn_dmaxprime}
   d_{\max}':=2+\left\lfloor \log_q(2\cdot|\Aut(S)|)\right\rfloor.
\end{equation}
Recall that $L=H^\infty$ and $K$ is its normalizer in $\Gamma(V,\kappa)$.
\begin{proposition}
If $S$ is a simple group of Lie type in characteristic distinct from $p$ which is not isomorphic to an alternating group, then we have one of the following:
\begin{enumerate}
\item[(1)] $L=\PSL_2(11)$, $\cP=H(4,4)$, and the two $S$-orbits are $11$- and  $22$-tight sets respectively,
\item[(2)] $L=\PSp_4(3)$, $\cP=Q^-(5,5)$, and the two $S$-orbits are $36$- and  $90$-tight sets respectively,
\item[(3)] $L=3.\PSU_4(3)$, $\cP=H(5,4)$, and the two $S$-orbits are $6$- and  $27$-tight sets respectively,
\item[(4)] $L=G_2(3)$, $\cP=Q^-(13,2)$, and the two $S$-orbits are $12$- and  $117$-tight sets respectively.
\end{enumerate}
\end{proposition}
\begin{proof}
We continue with the arguments preceding the statement of this proposition. We call a $(q,r)$ pair feasible for $S$ if it satisfies that $\max(\{2,\lceil \frac{1}{2}e(S)\rceil-1\})\le r\le \lfloor\frac{1}{2}d_{\max}'\rfloor$ and $\frac{q^r-1}{q-1}$ divides $|\Aut(S)|$. For (D) and (E), we exclude $\PSU_3(4)$, $\PSU_3(5)$, $\PSU_6(2)$, $\POmeg_7(3)$ by the fact that there is no feasible pair. The following three cases are excluded due to the fact that there is no candidate representation with Frobenius-Schur indicator $+$ or $\circ$ whose dimension is in the interval $[2r,2r+2]$: $S=\PSU_5(2)$, with $(3,4)$ as the only feasible $(q,r)$ pair; $S=\PSL_3(3)$, with $(5,4)$ as the only feasible $(q,r)$ pair; $S=G_2(4)$, with $(3,6)$ as the only feasible $(q,r)$ pair. For $S=\PSU_3(3)=G_2(2)'$, we have $e(S)=6$ and an irreducible representation of dimension $6$ has Frobenius-Schur indicator $-$ by \cite{HissMalle}. Hence we have $d\ge 7$ and so $r\ge 3$. There is no feasible $(q,r)$ pair with $r\ge 3$, $p\ne 2,3$ for such a $S$.

Suppose that $S=\PSL_2(7)=\PSL_3(2)$. For a feasible $(q,r)$ pair we have $r=2$, and $q$ is one of  $3$, $5$, $11$, $13$, $23$, $27$, $41$, $47$, $83$ and $167$. It follows that $d\le 6$.  Since each $q$ is not a square, we expect $\cP$ to be orthogonal. By \cite{HissMalle}, there is exactly one candidate representation of dimension $d=6$ with Frobenius-Schur indicator $+$. The condition $d\le d_{\max}'$ holds only for $q=3,5$, where $d_{\max}'$ is as in \eqref{eqn_dmaxprime}. The condition $|\Aut(S)|\ge\frac{1}{2}|\cP|$ does not hold for $q=5$, and the normalizer $K$ has three orbits on $\cP=Q^-(5,3)$ for $q=3$.

Suppose that $S=\PSL_2(11)$. For a feasible $(q,r)$ pair, either it is one of $\{(2,4),(3,4)\}$, or $r=2$ and $q$ is one of $18$ integers; moreover, if $q$ is a square, then $(q,r)\in\{(4,2),(9,2)\}$. It follows that $d\le 2r+2\le 10$, and if $\cP$ is unitary then $q\in\{4,9\}$ and $d\in\{4,5\}$. By \cite{HissMalle}, there are two candidate representations up to weak equivalence: (i) $d=5$, with Frobenius-Schur indicator $\circ$ and quadratic irrationality $\sqrt{-11}$, (ii) $d=10$,  with Frobenius-Schur indicator $+$.  For (i), $X^2+X+3=0$ has roots in $\F_3$ and $V$ does not have a unitary form for $q=3^2$ by \cite[Corollary~4.4.2]{bray2013maximal}. If $q=4$, then $S$ has two orbits of sizes $55$, $110$ respectively on $\cP=H(4,4)$ by a computer check.  For (ii), we examine the list of feasible pairs to see that $r=4$ and $q\in\{2,3\}$. Since $d=2r+2$, we expect $\cP=Q^-(9,q)$. The condition $|\Aut(S)|\ge \frac{1}{2}|\cP|$ does not hold for $q=3$, and there is no integer $i$ such that $\lcm(i,2^4+1-i)\cdot (2^4-1)$ divides $|\Aut(S)|$ for $q=2$.

Suppose that $S=\PSL_2(13)$. For a feasible $(q,r)$ pair, either it is one of $(2,3)$, $(3,3)$, $(4,3)$, $(9,3)$, $(16,3)$, $(5,4)$ and $(3,6)$, or $r=2$ and $q$ is one of $19$ integers. The pairs $(9,3)$, $(16,3)$, $(5,4)$ and $(3,6)$ are excluded by the condition $d_{\max}'\ge d\ge 2r$, where $d_{\max}'$ is as in \eqref{eqn_dmaxprime}.  It follows that $r\le 3$, and so $d\le 8$. By \cite{HissMalle} up to weak equivalence there is exactly one candidate representation of dimension $d=7$ whose Frobenius-Schur indicator is $+$ for odd $p$, and $\F_q$ is the splitting field of $X^2-13$ over $\F_p$. We then have $r=3$, so $q=3$ for $(q,r)$ to be a feasible pair. It turns out that the normalizer $K$ has three orbits on $\cP=Q(6,3)$.

Suppose that $S=\PSL_3(4)$. For a feasible $(q,r)$ pair, either it is  $(3,4)$, or $r=2$ and $q$ is one of $57$ integers; moreover, if $q$ is a square then $(q,r)=(9,2)$. We have $r\le 3$, and so $d\le 8$. If $\cP$ is unitary, then $(q,r)=(9,2)$ and $d\in\{4,5\}$. There is exactly one such candidate representation of dimension $4$ in which case $L=4.L_3(4)$ up to weak equivalence, and $S$ is transitive on $\cP=H(3,9)$ by direct check. If $\cP$ is orthogonal, then there is exactly one such candidate representation whose dimension is in the range $[4,8]$ and Frobenius-Schur indicator is $+$ up to weak equivalence by \cite{HissMalle}. For such a representation we have $L=2.L_3(4)$, and $S$ is transitive on $\cP=Q^-(5,3)$ by direct check.

Suppose that $S=\PSp_4(3)=\PSU_4(2)$. For a feasible $(q,r)$ pair with $p\ne 2,3$, we have $r=2$ and $q=p$ is one of $33$ primes. It follows that $\cP$ is orthogonal and $d\le 6$. There is exactly one candidate representation up to weak equivalence whose dimension is in the range $[4,6]$ and Frobenius-Schur indicator is $+$ by \cite{HissMalle}, and such a representation has dimension $d$. Since $r=2$, we expect $\cP$ to be $Q^-(5,p)$. We deduce from $|\Aut(S)|\ge\frac{1}{2}|\cP|$ that $p\le 17$. There is an integer $i$ such that $(p+1)\cdot\textup{lcm}(i,p^3+1-i)$ divides $|\Aut(S)|$ only if  $q=5$ or $7$. For $q=5$, we have $\cP=Q^-(5,5)$ and there are exactly two $S$-orbits of sizes $540,216$. Their sizes are not divisible by the ovoid number $126$, so the two orbits are both tight sets. For $q=7$, we have $\cP=Q^+(5,7)$ and so $r=3$: a contradiction.

Suppose that $S=\PSp_6(2)$. Then $(e(S),d_{\max})=(7,24)$, so $3\le r\le 12$. The only feasible $(q,r)$ pair is $(3,4)$. It follows that $8\le d\le 10$, and $\cP$ is orthogonal since $q$ is not a square. There is exactly one candidate representation which has dimension $8$ up to weak equivalence, for which we have $L=2.\PSp_6(2)$. The group $S$ is transitive on $\cP=Q^+(7,3)$ by direct check.

Suppose that $S=\PSp_6(3)$. Then $e(S)=13$, and so $r\ge 6$. We deduce from $d_{\max}'\ge d\ge e(S)$ that $q\le 8$, where $d_{\max}'$ is as in \eqref{eqn_dmaxprime}. A feasible $(q,r)$ pair with $q\le 8$ is one of $(2,6)$, $(4,6)$ and $(2,12)$. If $\cP$ is unitary, then $(q,r)=(4,6)$ and $d\in\{12,13\}$. There is exactly one such candidate representation up to weak equivalence which has dimension $13$ by \cite{HissMalle}. It turns out that $\cP=H(12,4)$ and there are four $K$-orbits of different sizes on the points of  $\cP$. There is no candidate representation whose dimension is in the ranges $[13,14]$ and $[24,26]$ with a Frobenius-Schur indicator $+$ by \cite{HissMalle}.

Suppose that $S=\PSU_4(3)$. Then $e(S)=6$ and a candidate representation of dimension $6$ has Frobenius-Schur indicator $\circ$ by \cite{HissMalle}. We deduce that $r\ge 3$ in all cases. A feasible $(q,r)$ pair with $r\ge 3$ is one of $(2,3)$, $(4,3)$, $(2,4)$ and $(2,6)$. It follows that $p=2$ and $6\le d\le 12$. The only candidate representation has dimension $d=6$ and Frobenius-Schur indicator $\circ$ by \cite{HissMalle}. We thus expect $\cP=H(5,q)$ with $q$ square, so that  $q=4$. There is exactly one conjugacy class of $\PSU_4(3)$ in $\PGU_6(2)$, and it  has two orbits of sizes $126$, $567$ on $\cP=H(5,4)$ which are $6$- and $27$-tight sets respectively.

Suppose that $S=\textup{P}\Omega^+_{8}(2)$. We have $e(S)=8$, so the rank $r\ge 3$. A feasible $(q,r)$ pair with $r\ge 3$ is either $(3,4)$ or $(7,4)$. It follows that $8\le d\le 10$. There is exactly one candidate representation of dimension $8$  whose Frobenius-Schur indicator is $+$ up to weak equivalence by \cite{HissMalle}, and the corresponding covering group is $L=2.\POmeg^+_8(2)$. We first realize $2.\POmeg^+_8(2)$ as the Weyl group of the $E_8$-root lattice in $\mathbb{R}^8$ generated by the $120$ reflections in the $240$ root vectors and then take tensor with $\F_p$ to get the representation in Magma, cf. \cite{Atlas}. The group $S$ is transitive on $\cP=Q^+(7,3)$ for $q=3$, and there are three $K$-orbits on $\cP=Q^+(7,7)$ for $p=7$.

Suppose that $S=G_2(3)$. Then $e(S)=14$, and so $r\ge 6$. The only feasible $(q,r)$ pair is $(2,6)$.  Since $e(S)\le d\le 2r+2$, we deduce that $d=14$.  There is exactly one such candidate representation whose Frobenius-Schur indicator is $+$ up to weak equivalence by \cite{HissMalle}. We check by Magma \cite{Magma} that $\cP=Q^-(13,2)$ and $S$ has two orbits on the singular points which are $12$- and $117$-tight sets respectively. This completes the proof.
\end{proof}

We next consider the case where $S$ is a sporadic group.  We also take the same approach as outlined in Remark \ref{rem_classS_Smethod}. By \cite[Theorem~2.3.2]{LPS3} there is a lower bound $R(S)$ on the dimension $d$, see also \cite[p.~187]{kleidman1990subgroup}. We define $d_{\max}'$ as in \eqref{eqn_dmaxprime} which is also an upper bound on $d$ by Lemma \ref{lem_ParaCond}.
\begin{proposition}
If $S$ is a sporadic simple group, then we have one of the following:
\begin{enumerate}
\item[(1)]$L=M_{11}$, $\cP=Q^-(9,2)$ and $S$ has two orbits on the points of $\cP$ which are $11$- and $22$-tight sets;
\item[(2)]$L=3.M_{22}$, $\cP=H(5,4)$ and $S$ has two orbits on the points of $\cP$ which are $11$- and $22$-tight sets;
\item[(3)]$L=Co_1$, $\cP=Q^+(23,2)$ and $S$ has two orbits on the points of $\cP$ which are $24$- and $2025$-tight sets.
\end{enumerate}
\end{proposition}
\begin{proof}
We call a $(q,r)$ pair feasible for $S$ if $\max(\{2,\lceil \frac{1}{2}R(S)\rceil\}-1)\le r\le \lfloor\frac{1}{2}d_{\max}'\rfloor$ and $\frac{q^r-1}{q-1}$ divides $|\Aut(S)|$. For $S=Fi_{23}$, $Fi_{24}'$, $HN$, $Ly$, $BM$, $F_1$ or $He$, the condition $d_{\max}'\ge R(S)$ does not hold for any $q$. For $S=M_{23}$ or $M_{24}$, $R(S)=11$ and the only feasible $(q,r)$ pair is $(2,6)$.  For $S=J_2$,  $R(S)=6$ and we have $2\le r\le 6$ for all feasible $(q,r)$ pairs. There is no candidate representation with Frobenius-Schur indicator $+$ or $\circ$ whose dimension is the interval $[2r,2r+2]$ in each case by \cite{HissMalle}.

For $S=M_{11}$, $R(S)=5$ and a feasible $(q,r)$ pair either has $r=2$ or is one of $(2,4)$ and $(3,4)$. It follows that $5\le d\le 10$.  There are two candidate representations which have dimension $10$ and Frobenius-Schur indicator $+$ up to weak equivalence by \cite{HissMalle}, one defined over $\F_2$ and the other over $\F_3$. For $q=2$, we have  $\cP=Q^-(9,2)$ and $S$ has two orbits on the points of $\cP$ which are $11$- and $22$-tight sets respectively. For $q=3$, it turns out that $\cP=Q^+(9,3)$ and so $r=5$: a contradiction to $r=4$.

For $S=M_{12}$, $R(S)=6$ and a feasible $(q,r)$ pair either has $r=2$ or is one of $(2,4)$, $(3,4)$. It follows that $6\le d\le 10$. There are two candidate representations which have dimension $10$ and Frobenius-Schur indicator $+$ up to weak equivalence by \cite{HissMalle}, one defined over $\F_2$ and the other over $\F_3$. For $q=2$, we have $\cP=Q^-(9,2)$ and $S$ is transitive on the points of $\cP$. For $q=3$, we have $\cP=Q^+(9,3)$ and so $r=5$: a contradiction to $r=4$.

For $S=M_{22}$, $R(S)=6$ and a feasible $(q,r)$ pair either has $r=2$ or is one of  $(2,3)$, $(2,4)$, $(2,6)$, $(4,3)$ and $(3,4)$. It follows that $6\le d\le 12$. There is  exactly one candidate representation of dimension $6$ defined over $\F_{4}$ with Frobenius-Schur indicator $\circ$ up to weak equivalence  by \cite{HissMalle}. It turns out that $\cP=H(5,4)$, $L=3.M_{22}$ and $S$ has two orbits on the points of $\cP$ which are $11$- and $22$-tight sets respectively.

For $S=J_1$, $R(S)=7$ and a feasible $(q,r)$ pair is one of $(2,3)$, $(4,3)$, $(7,3)$, $(11,3)$, $(2,4)$ and $(3,4)$. It follows that $7\le d\le 10$. There is  exactly one candidate representation of dimension $7$ defined over $\F_{11}$ with Frobenius-Schur indicator $+$ up to weak equivalence  by \cite{HissMalle}. It turns out that $\cP=Q(6,11)$ and the normalizer $K$ has more that two orbits on the singular points.

For $S=J_3$, $R(S)=9$ and a feasible $(q,r)$ pair is one of $(2,4)$, $(2,8)$, $(3,4)$ and $(4,4)$. It follows that $9\le d\le 10$.  There is  exactly one candidate representation of dimension $9$ defined over $\F_4$ with Frobenius-Schur indicator $\circ$ up to weak equivalence  by \cite{HissMalle}. It turns out that $\cP=H(8,4)$ and $S$ is transitive on the points of $\cP$.

For $S=Suz$, $R(S)=12$ and a feasible $(q,r)$ pair is one of $(2,6)$, $(2,12)$, $(3,6)$ and  $(4,6)$. It follows that $12\le d\le 26$.  There is  exactly one candidate representation of dimension $12$ defined over $\F_4$ with Frobenius-Schur indicator $\circ$ up to weak equivalence  by \cite{HissMalle}. It turns out that $\cP=H(11,4)$ and the normalizer $K$ has three orbits on the singular points.

For $S=Co_1$, $R(S)=24$ and the only feasible pair is $(q,r)=(2,12)$. It follows that $24\le d\le 26$. There is exactly one candidate representation of dimension $24$ whose Frobenius-Schur indicator is $+$ up to weak equivalence by \cite{HissMalle}. It turns out that $\cP=Q^+(23,2)$, $L=Co_1$ and $S$ has exactly two orbits on the points of $\cP$ which are $24$- and $2025$-tight sets respectively. This completes the proof.
\end{proof}

To summarize, we have established Theorem \ref{thm_TS} for the case where $H$ is a subgroup of Aschbacher class $\cS$ in this section. The resulting examples are collected in Tables \ref{tab_TS_cSLiedef} and \ref{tab_TS_cSother}. In particular, the case where $\kappa$ is symplectic is handled in Proposition \ref{prop_classS_caseS} based on the main theorem in \cite{LiebeckRank3}.

\section{Proof of Theorem \ref{thm_TS}: the geometric subgroups}

In this section, we present the proof of Theorem \ref{thm_TS} for the case of geometric subgroups. Suppose that $V$ is a $d$-dimensional vector space over $\F_q$ that is equipped with a nondegenerate  quadratic, unitary or symplectic form $\kappa$. Let $\cP$ be the associated polar space, and suppose that it has rank $r\ge 2$ and is not $Q^+(3,q)$. Let $H$ be an irreducible group of semisimilarities that has two orbits $O_1$, $O_2$ on the points of $\cP$ which are respectively $i_1$- and $i_2$-tight sets, and assume that $H$ is of Aschbacher class $\cC_2,\ldots,\cC_7$ or $\cC_8$. Here, we say that $H$ is of Aschbacher class $\cC_i$ if it is contained in a maximal subgroup of Aschbacher class $\cC_i$ in $\Gamma(V,\kappa)$, cf. \cite{asch_max}. Let $\F_{q^b}$ be the largest extension field of $\F_q$ such that $H$ is a subgroup of $\GamL_{d/b}(q^b)$. Let $E$ be the corresponding copy of $\F_{q^b}$ in $\End(V)$, and regard $V$ as a vector space $V'$ of dimension $d/b$ over $E$. There is a nondegenerate reflexive sesquilinear form or quadratic form $\kappa'$ on $V'$ such that $(V',\kappa')$, $(V,\kappa)$ are associated as in Table \ref{tab_extfieldP'gt0}. Let $\cP'$ be the polar space associated with $(V',\kappa')$. The group $H$ is a subgroup of $\Gamma^\#(V',\kappa')$, where $\Gamma^\#(V',\kappa')$ is defined in \eqref{eqn_GamJ}. We write $\Gamma^\#$ for $\Gamma^\#(V',\kappa')$ and $\Delta^\#$ for its intersection with $\Delta(V,\kappa)$ for short.
\begin{lemma}\label{lem_C3_P'ne0}
If $b>1$, then $V'$ has a nonzero singular or isotropic vector.
\end{lemma}
\begin{proof}

Suppose to the contrary that $b>1$ and $V'$ has no nonzero singular or isotropic vector. This happens exactly when $d=b$ for the rows 4, 5, 7, 9, 12 and when $d=2b$ for the row 3 by Table \ref{tab_extfieldP'gt0}. We observe that the case $d=2b$ for the row 3 reduces to the case $d=b$ for the row 9, so we only need to consider the first five cases with $d=b$. It is routine to check that the conditions $\frac{q^r-1}{q-1}$ divides $\textup{P}\Gamma^\#(V',\kappa')$  and $\textup{P}\Gamma^\#(V',\kappa')\ge \frac{1}{2}|\cP|$ exclude rows 4, 5, 7, 12. For row 9, we have $d=2r+2$, $|\textup{P}\Gamma^\#(V',\kappa')|=df(q^{r+1}+1)$,  and $|\cP|=(q^{r+1}+1)\frac{q^{r}-1}{q-1}$. The two conditions simplify to $(q^{r}-1)\mid f(q^2-1)$, $2f(2r+2)(q-1)\ge q^{r}-1$, which hold only for a finite list of $(q,r)$ pairs. For such a $(q,r)$ pair, we examine whether there is an integer $i$ such that $\lcm(i,q^{r+1}+1-i)\frac{q^{r}-1}{q-1}$ divides $f(2r+2)(q^{r+1}+1)$. A $(q,d)$ pair that satisfies all the conditions either is $(2,10)$, or has $d=6$ and $q\in  \{7, 3^2, 11, 2^4, 3^3, 2^5, 2^6\}$. We use Magma to check that $\Gamma^\#(V',\kappa')$ has at least three orbits on the points of $\cP$ in each case. This completes the proof.
\end{proof}

We first consider the case where $b>1$. Let $\cM_1$ be as defined in \eqref{eqn_M1def}, i.e., $\cM_1=\{\la \eta v\ra_{\F_q}\mid \eta\in\F_{q^b}^*,\,\la v\ra_{\F_{q^b}}\in \cP'\}$. Let $\cM_2$ be its complement in $\cP$. The set $\cM_1$ is nonempty by Lemma \ref{lem_C3_P'ne0}, and the set $\cM_2$ is empty only for row 1, rows 9 (with $b=2$) and row 10 (with $b=2$) upon inspection of Table \ref{tab_extfieldP'gt0}. We define
\begin{equation}\label{eqn_C3_Oip}
    O_i'=\{\la v\ra_{\F_{q^b}}\mid \la v\ra_{\F_q}\in O_i\} \quad\textup{ for }i=1,2.
\end{equation}

\begin{proposition}\label{prop_C3_M1prop}
If $\cM_1,\cM_2$ are proper subsets of $\cP$, then we have one of the cases in Table \ref{tab_TS_C3} and $\cM_1$, $\cM_2$ are the two $H$-orbits.
\end{proposition}

\begin{proof}
Write $q=p^f$ with $p$ prime, and set $b'=b$  or $b'=b/2$ according as $\kappa'$ is orthogonal or unitary. Let $\alpha$ be as in Table \ref{tab_extfieldP'gt0} for rows 8, 12, and set $\alpha=1$ for the remaining rows. Since $\cM_2$ is nonempty, we can not have row 1, row 9 (with $b=2$) or row 10 (with $b=2$) of Table \ref{tab_extfieldP'gt0}. Since the set of singular points of $\cP'$ is a trivial intriguing set, the set $\cM_1$ is a tight set for rows 2, 6, 8, 10 of Table \ref{tab_extfieldP'gt0} and is an $m$-ovoid for rows 3, 5, 7, 9 by \cite{KellyCons}. It is routine to check that the condition $\frac{q^r-1}{q-1}$ divides $|\cM_1|$ does not hold for rows 4, 11, and it holds for row 12 only if $b=2$. Therefore, we only need to consider rows 2, 6, 8, 10 (with $b>2$), 12 (with $b=2$) of Table \ref{tab_extfieldP'gt0}. In all those cases $\cP'$ is orthogonal or unitary.  Since $\cP'$ is nonempty by Lemma \ref{lem_C3_P'ne0}, we deduce that $\{\kappa'(x)\textup{ or }\kappa'(x,x)\mid x\in V'\setminus\{0\}\}=\F_{q^{b'}}$. We set $X:=\{\kappa'(x)\textup{ or }\kappa'(x,x)\mid \la x\ra\in\cM_2\}$, and write $\lambda(\Delta^\#):=\{\lambda(g)\mid g\in\Delta^\#\}$, where $\lambda(g)$ is as in \eqref{eqn_Gam}. We have $\kappa=\tr_{\F_{q_0^b}/\F_q}\circ \alpha\kappa'$ with $q_0\in\{q,q^{1/2}\}$ by Table \ref{tab_extfieldP'gt0}, so $X=\{a\in\F_{q^{b'}}^*\mid\tr_{\F_{q_0^b}/\F_q}(\alpha a)=0\}$ correspondingly.

Since $\cM_1$, $\cM_2$ are $\Gamma^\#$-invariant, they are the two $H$-orbits by our assumption. It follows that $\Gamma^\#$ should be transitive on both $\cM_1$ and $\cM_2$. The group $\Gamma^\#$ is transitive on $\cM_1$ by Witt's Lemma.  We now determine when $\Gamma^\#$ is transitive on $\cM_2$. Take a point $\la v\ra$ in $\cM_2$, and write $\eta:=\kappa'(v)$ or $\kappa'(v,v)$. Since $\cM_2$ is the $\Gamma^\#$-orbit of $\la v\ra$, we deduce that $X=X'$ with $X'=\{\lambda(g)\eta^{\sigma(g)}\mid g\in\Gamma^\#\}$.  Conversely, if $X'=X$ then we deduce from Witt's Lemma that $\Gamma^\#$ is transitive on $\cM_2$.
\begin{enumerate}
\item[(1)] For row 2 of Table \ref{tab_extfieldP'gt0} we have $\lambda(\Delta^\#)=\F_q^*$ and $\alpha=1$,  so $X'=\{a\eta^{p^i}\mid i\ge 0, a\in\F_{q}^*\}$. We also have $X=\{a\in\F_{q^{b}}^*\mid\tr_{\F_{q^{b}}/\F_q}(a)=0\}$. It follows from $X=X'$ that $bf(q-1)$ is divisible by $q^{b-1}-1$. This holds only if $b=2$ or $(p^f,b)\in\{(2,3),(8,3)\}$, and we have $X=X'$ in those cases.
\item[(2)] For row 6 of Table \ref{tab_extfieldP'gt0} we have $\lambda(\Delta^\#)=\F_{q^{1/2}}^*$  and $\alpha=1$, so $X'=\{a\eta^{p^i}\mid i\ge 0, a\in\F_{q^{1/2}}^*\}$. We also have $X=\{a\in\F_{q^{b/2}}^*\mid \tr_{\F_{q^{b/2}}/\F_{q^{1/2}}}(a)=0\}$. It follows from $X=X'$ that $q^{(b-1)/2}-1$ divides $(q^{1/2}-1)bf/2$. This holds only if $(p^f,b)\in\{(2^2,3),(2^6,3)\}$, in which cases we do have $X=X'$.
\item[(3)] For row 8 of Table \ref{tab_extfieldP'gt0} we have $X=\F_{q^{b/2}}^*$. We choose $v$ properly so that $\eta=1\in X$. In this case $\lambda(\Delta^\#)=\F_{q^{b/2}}^*\cap\F_q^*=\F_q^*$, since $b$ is even. It follows that $\F_q^*\subseteq X'$. We have $\kappa'(vg,vg)=\lambda(g)\in\F_q^*\cdot\alpha^{\sigma(g)-1}$ for $g\in\Gamma^\#$ by \eqref{eqn_GamJ}. It follows that $X$ is a subset of $X''=\{\alpha^{p^i-1}a\mid a\in\F_q^*,i\ge 0\}$. Since $\alpha^{q^{b/2}}+\alpha=0$, the latter set has size at most $bf(q-1)/2$. We deduce from $X\subseteq X''$ that $q^{b/2}-1\le bf(q-1)/2$, which holds only if $b=2$. When $b=2$, we have $X=X'=\F_q^*$ by comparing sizes.
\item[(4)] For row 10 of Table \ref{tab_extfieldP'gt0} with $b>2$, we have $\lambda(\Delta^\#)=\F_{q}^*$ and $\alpha=1$,  so $X'=\{a\eta^{p^i}\mid i\ge 0, a\in\F_{q}^*\}$. We also have $X=\{a\in\F_{q^{b/2}}^*\mid \tr_{\F_{q^{b/2}}/\F_q}(a)=0\}$. Since $\alpha=1$, we deduce from $X=X'$ that $q^{b/2-1}-1$ divides $(q-1)bf/2$. It holds only if $b=4$ or $(p^f,b)\in\{(2,6),(2^3,6)\}$, and $X=X'$ holds in each case.
\item[(5)] For row 12 of Table \ref{tab_extfieldP'gt0} with $b=2$, we have $X=\{u\in\F_{q^{2}}^*\mid \tr_{\F_{q^{2}}/\F_q}(\alpha u)=0\}$. Since $\eta$ is in $X$, we have $X=\F_q^*\cdot\eta$. In this case $\lambda(\Delta^\#)$ is the intersection of $\F_q^*$ with the set of nonzero squares of $\F_{q^2}^*$, i.e., $\lambda(\Gamma^\#)=\F_q^*$. It follows that $X'$ contains $\F_q^*\cdot\eta$. We deduce that $X=X'$ by comparing sizes.
\end{enumerate}

We list the above cases where $\Gamma^\#$ is transitive on both $\cM_1$ and $\cM_2$ in Table \ref{tab_TS_C3}.  It remains to the determine the subgroups of $\Gamma^\#$ that are transitive on both $\cM_1$ and $\cM_2$.
\begin{enumerate}
\item[(i)] For row 2 of Table \ref{tab_extfieldP'gt0}, we have $d\ge 6$ since $\kappa$ has plus sign and we do not consider the case $\cP=Q^+(3,q)$. Since $H$ is transitive on both $\cM_1$ and $\cM_2$, we deduce that it is transitive on both the set of singular points and the isometry class of nonsingular points of $(V',\kappa')$ that contains $\la v\ra_{\F_{q^b}}$. We consider only the case $d/b\ge 6$. By Lemmas \ref{lem_Q+(2m,q)sub} and the fact $q^b$ is a square or cube, we deduce that either $\Omega_{d/2}^+(q^b)\unlhd H$, or $\SU_{d/(2b)}(q^b)\unlhd H$ with $d/(2b)$ even, or $\textup{Spin}_7(q^b)\unlhd H$ with $d/b=8$.
\item[(ii)] For row 6 of Table \ref{tab_extfieldP'gt0}, we have $b=3$ and $q\in\{2^2,2^6\}$. Since $H$ is transitive on the singular points of $\cP'$, we deduce that either $\SU_{d/3}(q^{3/2})\unlhd H$ or $(d/b,q)=(3,4)$ by \cite[Theorem~4.1]{giudici2020subgroups}.
\item[(iii)] For row 8 of Table \ref{tab_extfieldP'gt0}, we have $b=2$ and $d/2$ is even by Table \ref{tab_extfieldP'gt0}. Since $H$ is transitive on both $\cM_1$ and $\cM_2$, we deduce that it is transitive on both the singular points and the nonsingular points of $(V',\kappa')$. By Lemma \ref{lem_H_TrSinNonsin}, we see that either $\Omega(V',\kappa')\unlhd H$ or $(q,d)=(2,4)$.
\item[(iv)]For row 10 of Table \ref{tab_extfieldP'gt0}, $d/b$ is even. Since $H$ is transitive on both $\cM_1$ and $\cM_2$, we deduce that it is transitive on both the set of singular points and the set of nonsingular points of $(V',\kappa')$. Since $q^b$ is a $b$-th power with $b\in\{4,6\}$, we deduce that $\SU_{d/b}(q^{b/2})\unlhd H$ by Lemma \ref{lem_H_TrSinNonsin}.
\item[(v)] For row 12 of Table \ref{tab_extfieldP'gt0}, we have  $b=2$. Since $H$ is transitive on both $\cM_1$ and $\cM_2$, we deduce that it is transitive on both the set of singular points and one isometry class of nonsingular points of $(V',\kappa')$. By Lemma \ref{lem_Q(dodd,q)sub}, we have either $\Omega_{d/2}(q^2)\unlhd H$, or $G_2(q^2)\unlhd H$ with $d=14$.
\end{enumerate}
This completes the proof.
\end{proof}

By Lemma \ref{lem_C3_P'ne0} and Proposition \ref{prop_C3_M1prop}, it remains to consider the case where $\cM_1=\cP$ for $b>1$. In such a case, $(V,\kappa)$ and $(V',\kappa')$ have the same set of nonzero isotropic or singular vectors.  Since $\cM_1=\cP$, we are in  rows 1, 9 (with $b=2$) and 10 (with $b=2$) of Table \ref{tab_extfieldP'gt0}.

\begin{proposition}\label{prop_C3_O1peqPp}
Suppose that $\cM_1=\cP$ and $b>1$, and let $O_1'$, $O_2'$ be as in \eqref{eqn_C3_Oip}.  If one of $O_1'$, $O_2'$ consists of all points of $\cP'$, then we have the two rows in the first block of Table \ref{tab_C3Remain}.
\end{proposition}
\begin{proof}
We assume without loss of generality that $O_1'=\cP'$, and let $Y$ be the set of nonzero isotropic or singular vectors of $(V',\kappa')$. Let $Z,~\tilde{Z}$ be the center of $\GL(V)$ and $\GL(V')$ respectively. Since $H$ is transitive on $O_1$, it is transitive on the points of $\cP'$. Therefore, $HZ$ is not transitive on $Y$ while $H\tilde{Z}$ does. The $HZ$-orbits on $Y$ are in bijection with the $H$-orbits on the points of $\cP$, so $HZ$ has two orbits $Y_1$, $Y_2$ on $Y$. Since $HZ$ is normal in $H\tilde{Z}$, there is an element $\eta\in \tilde{Z}$ such that $Y_1\eta=Y_2$. It follows that $[H\tilde{Z}:\,HZ]$ is even. In particular, $[\tilde{Z}:\,Z]=\frac{q^b-1}{q-1}$ is even which forces $q$ to be odd and $b$ to be even. Since $\Omega(V',\kappa')$ is transitive on $Y$ by Witt's Lemma, $H$ does not contain $\Omega(V',\kappa')$. By the choice of $b$, $H$ is not a subgroup of $\Gamma(V',\kappa')$ of class $\cC_3$.

Suppose that we are in the situation of row 1 of Table \ref{tab_extfieldP'gt0}, so that $\cP=W(d-1,q)$ and $\cP'=W(d/b-1,q^b)$ with $d/b$ even. Since $H$ is transitive on the points of $\cP'$ and $q^b$ is an odd square, we deduce that $(q,d/2,H^\infty)=(3,2,\SL_2(5))$ by \cite[Theorems~3.1,~5.2]{giudici2020subgroups}. We use the Atlas data \cite{Atlas} to verify that $\SL_2(5)$ has exactly two orbits of sizes $40$ on the nonzero vectors of $\kappa'$.  Both orbits are $\F_3^*$-invariant and they yield $5$-tight sets of $\cP=W(3,3)$ respectively.

Suppose that we are in the situation of row 9 or 10 of Table \ref{tab_extfieldP'gt0} with $b=2$, so that $\cP=Q^{\mp}(d-1,q)$ and $\cP'=H(d/2-1,q^2)$. Since $\cP$ has rank $r\ge 2$ and we do not consider $Q^+(3,q)$, we have $d\ge 6$. Since $H$ is transitive on the points of $\cP'$ and $q$ is odd, we deduce that $(d/2,q,H^\infty)$ is one of $(3,3,\PSL_2(7))$, $(3,5,3.A_7)$, $(4,3,4.\PSL_3(4))$ by \cite[Theorem~4.1]{giudici2020subgroups}. For the triple $(3,5,3.A_7)$, the two $3.A_7$-orbits are both $3$-ovoids of $Q^-(5,5)$. For the triple $(3,3,\PSL_2(7))$, $\PSL_2(7)$ has four orbits of sizes $56$ which are $\F_3^*$-invariant and its normalizer in $\Delta^\#(V',\kappa')$ has two orbits which are both $14$-tight sets of $Q^-(5,3)$. For the triple $(4,3,4.\PSL_3(4))$, $4.\PSL_3(4)$ is transitive on the set $Y$ of nonzero singular vectors of $\cP=Q^+(7,3)$. This completes the proof.
\end{proof}

\begin{proposition}\label{prop_C3_OipProper_r'eq1}
Suppose that $\cM_1=\cP$ and $b>1$, and write $r'$ for the rank of $\cP'$. Let $O_1'$, $O_2'$ be as in \eqref{eqn_C3_Oip}, and assume that they are both proper subsets of $\cP'$. If $r'=1$, then either $(\cP,\cP')=(W(3,q),W(1,q^2))$ and $H$ occurs in (T1-a) of Theorem \ref{thm_TS}, or it occurs in the second block of Table \ref{tab_C3Remain}.
\end{proposition}
\begin{proof}
The set $\cP'$ is nonempty by Lemma \ref{lem_C3_P'ne0}. There are two such cases that satisfy the hypothesis in the proposition: (1) $\cP'=W(1,q^2)$ and $\cP=W(3,q)$, (2) $\cP'=H(2,q^2)$ and $\cP=Q^-(5,q)$. Let $\tilde{Z}$ be the center of $\GL(V')$. Let $s_1$, $s_2$ be the sizes of the two $H$-orbits on $\cP'$. We have $s_1+s_2=|\cP'|$, $\lcm(s_1,s_2)$ divides $[H\tilde{Z}:\,\tilde{Z}]$, and so $[H\tilde{Z}:\,\tilde{Z}]\ge\frac{1}{2}|\cP'|$. The group $H$ is an irreducible subgroup of $\Gamma(V',\kappa')$ which is not of class $\cC_3$ by the choice of $b$, and it is not of $\cC_4$, $\cC_7$ by comparing dimensions. We consider the two cases separately.

We first consider the case where $\cP'=W(1,q^2)$ and $\cP=W(3,q)$. Take a basis $e_1,e_2$ of $V'$ over $\F_{q^2}$, and let $\phi$ be the semilinear mapping of $V'$ such that $x_1e_1+x_2e_2\mapsto x_1^pe_1+x_2^pe_2$. If $H$ is of class $\cS$, then we have $H^\infty=2.A_5$ and $q=p\equiv 3\pmod{10}$ by \cite[Table~8.2]{bray2013maximal}. We deduce from $[H\tilde{Z}:\,\tilde{Z}]\ge\frac{1}{2}|\cP'|$ that $q\in\{3,7\}$. The group $2.A_5$ is transitive on $\cP'$ for $q=3$ and has two orbits on $\cP'$ of sizes $20,~30$ respectively for $q=7$. The latter yields $40$, $60$-tight sets of $\cP=W(3,7)$. If $H$ is of class $\cC_2$ as a subgroup of $\Gamma(V',\kappa')$, then it is of the same class as a subgroup of $\Gamma(V,\kappa)$ which is covered by (T1-a) of Theorem \ref{thm_TS}. It is not of class $\cC_6$ as a subgroup of $\Gamma(V',\kappa')$ by \cite[Table~4.6.B]{kleidman1990subgroup} and the fact that $q^2$ is not a prime. If it is of class $\cC_8$, then the condition $\textup{P}\Gamma^\#|\ge\frac{1}{2}|\cP|$ is not satisfied.
If $H$ is of class $\cC_5$, then it stabilizes a subline of size $q_0+1$, where $q^2=q_0^s$ with $s$ prime. By \cite[Case~(IIc)]{LiebeckRank3} and the fact $\dim(V')=2$, we have $s\in\{2,3\}$ and $H^\infty=\SL_2(q^{2/s})$. If $s=2$, then the normalizer of $\SL_2(q)$ in $\Delta^\#(V',\kappa')$ is $\GL_2(q)\F_{q^2}^*$ which has two orbits on the nonzero vectors of $V'$. If $s=3$, then $(q,q_0)=(q_1^3,q_1^2)$ for some integer $q_1$, and $\F_q\cap\F_{q_0}=\F_{q_1}$. We write $K_1$, $K$ for the normalizers of $\SL_2(q_1^2)$ in $\Delta^\#(V')$ and $\Gamma^\#(V')$ respectively, so that $K=K_1:\la\phi\ra$. We have $K_1=\GL_2(q_1^2)\F_{q^2}^*$ or $K_1=\left(\frac{1}{2}\GL_2(q_1^2)\right)\F_{q^2}^*$ according as $p$ is even or odd. Here, $\frac{1}{2}\GL_2(q_1^2)$ is the unique subgroup of index $2$ in $\GL_2(q_1^2)$ for odd $p$. For $q$ even, $K_1$ has two orbits on nonzero vectors by the arguments in \cite[Case~(IIc)]{LiebeckRank3}. For $q$ odd, the group $K_1$ has one orbit $\cN_1$ of size $q_1^2+1$ and two orbits $\cN_2$, $\cN_3$ of sizes $\frac{1}{2}(q_1^6-q_1^2)$ on the points of $\PG(V')$ by the same arguments as in \cite[Case~(IIc)]{LiebeckRank3}. Write $q_1=p^{2^ak}$ with $k$ odd, and take an element $\eta\in\F_{p^{3k}}\setminus\F_{q_1^2}$. Then $\la e_1+\eta e_2\ra_{\F_{q^2}}$ is stabilized by the Sylow $2$-subgroup of $\la\phi\ra$. We deduce that the $K$-orbit of $\la e_1+\eta e_2\ra_{\F_{q^2}}$ equals one of $\cN_2$, $\cN_3$ but not their union. We conclude that $H$ has two orbits on $\cP'$ only if $q$ is odd for $s=3$.

We suppose that $\cP'=H(2,q^2)$ and $\cP=Q^-(5,q)$ in the sequel. If $H$ is of class $\cC_2$, then it preserves a decomposition $\cD':\,V'=\la e_1\ra_{\F_{q^2}}\oplus\la e_2\ra_{\F_{q^2}}\oplus\la e_3\ra_{\F_{q^2}}$ with each $e_i$ nonsingular. Let $N_{\Gamma'}(\cD')$ be the normalizer of $\cD'$ in $\Gamma(V',\kappa')$. It has no subgroup that has two orbits on the nonsingular points of $\cP'$ for $q=2$, so assume that $q>2$. Let $\cN_i'$ be the set of singular points $x_1e_1+x_2e_2+x_3e_3$ with $i$ nonzero coordinates for $i=2$, $3$, which are $N_{\Gamma'}(\cD')$-invariant. Then $|\cN_2'|=3(q+1)$,  $|\cN_3'|=(q+1)^2(q-2)$, and they are the two $H$-orbits and thus also the two $N_{\Gamma'}(\cD')$-orbits. Since $|\cN_2|$ and $|\cN_3|$ both divide $|N_{\Gamma'}(\cD')|=6\cdot 2f\cdot(q-1)(q+1)^3$, we deduce that $\lcm(3,~(q+1)(q-2))$ divides $12f(q+1)$. This holds only for $p^f\in\{3,5,2^2,2^3,2^5\}$, and the group $N_{\Gamma'}(\cD')$ is transitive on both $\cN_2'$ and $\cN_3'$ upon direct check. The sets $\cN_2'$, $\cN_3'$ yield tight sets for each such $q$.

If $H$ is of class $\cC_5$, there are two cases: (i) $q=q_0^2$ and $H$ stabilizes a subplane $\pi'$ of order $q_0$ that intersects $\cP'$ in a conic $\cC'$, (ii) $q=q_0^r$ for some odd prime $r$ and $H$ stabilizes a subplane $\pi'$ of order $q_0^2$ that intersects $\cP'$ in a unital $\cU'=H(2,q_0^2)$. For (i), the set of singular points outside $\cC'$ that lie on a secant line to $\cC'$ is $H$-invariant and has size $\binom{q+1}{2}(q-1)<|\cP'|-|\cC'|$, so there are at least two orbits outside $\cC'$. For (ii) there are $n_1=q_0^4+q_0^2+1-(q_0^3+1)=q_0^2(q_0^2-q_0+1)$ secant lines to $\cU'$ in $\pi'$, and the set of singular points outside $\cU'$ that those lines cover is $H$-invariant and has size $n_1(q-q_0)<|\cP'|-|\cU'|$. Therefore, $H$ is not of class $\cC_5$.

If $H$ is of class $\cC_6$, then we deduce from $d=3$ that it normalizes an irreducible subgroup $R=3_+^{1+2}$. By \cite[Table~4.6B]{bray2013maximal} we have $q=p$ and $p\equiv 2\pmod{3}$. The order of $H\tilde{Z}$ divides  $3^2\cdot|\Sp_2(3)|\cdot(q^2-1)$, so $216\ge\frac{1}{2}|\cP'|=\frac{1}{2}(p^3+1)$. It holds only for $p\in\{2,5\}$. For $p=2$ the group $R$ is transitive on the nonzero singular vectors. For $p=5$, there is a unique conjugacy class of $R$ in $\Gamma=\Gamma(V,\kappa)$ and its normalizer $N_\Gamma(R)$ has two orbits on $\cP=Q^-(5,5)$ that are respectively $54,~72$-tight sets.

If $H$ is of class $\cS$, then by \cite[Table~8.6]{bray2013maximal} we have one of the following cases: (i) $H^\infty\le\PSL_2(7)$, $q=p\equiv 3,5,6\pmod{7}$, (ii) $H^\infty\le 3.A_6$, $q=p\equiv 11,14\pmod{15}$ or $q=5$, (iii) $H^\infty\le 3.A_7$, $q=5$. For (i) we have $H^\infty=\PSL_2(7)$, and we deduce from $[H\tilde{Z}:\,\tilde{Z}]\ge\frac{1}{2}|\cP'|$ that $q\in\{3,5\}$. The group $\PSL_2(7)$ is transitive on $\cP'$ for $q=3$ and has two orbits of sizes $42,~84$ on $\cP'$ for $q=5$. In the latter case we obtain $42,~84$-tight sets of $Q^-(5,5)$. For (ii) we deduce from $[H\tilde{Z}:\,\tilde{Z}]\ge\frac{1}{2}|\cP'|$ that $q\in\{5,11\}$. There is no integer $i$ such that $\lcm(i,q^3+1-i)$ divides $|\Aut(A_6)|$ for $q=11$.  The group $3.A_6$ has two orbits of sizes $36,~90$ on $\cP'$ for $q=5$, and they yield $36,~90$-tight sets of $Q^-(5,5)$ respectively. The only quasisimple subgroup of $3.A_6$ whose normalizer is transitive on both orbits is $3.A_6$, so $H^\infty=3.A_6$. For (iii), $3.A_7$ has two orbits on nonzero singular points that yield $3$-ovoids of $Q^-(5,5)$ respectively. This completes the proof.
\end{proof}

\begin{proposition}\label{prop_C3_OipProper}
Suppose that $\cM_1=\cP$ and $b>1$, and write $r'$ for the rank of $\cP'$. Let $O_1'$, $O_2'$ be as in \eqref{eqn_C3_Oip}, and assume that they are both proper subsets of $\cP'$. If $r'\ge 2$, then $O_1'$, $O_2'$ form the two $H$-orbits on $\cP'$ and are both tight sets with the same parameters as $O_1$, $O_2$.
\end{proposition}
\begin{proof}
Since $H$ is transitive on  both $O_1$ and  $O_2$, we deduce that $O_1'$, $O_2'$ are the two $H$-orbits on $\cP'$. Moreover, we have $\{v\in V\setminus\{0\}\mid \la v\ra_{\F_q}\in O_i\}=\{v\in V'\setminus\{0\}\mid\la v\ra_{\F_{q^b}}\in O_i'\}$ for $i=1,2$. In particular, we have $O_i=\{\la \lambda v\ra_{\F_q}\mid\lambda\in\F_{q^b}^*,\la v\ra_{\F_{q^b}}\in O_i'\}$ for $i=1,~2$. For row 1 of Table \ref{tab_extfieldP'gt0} where both $\kappa$ and $\kappa'$ are symplectic, $O_i$ and $O_i'$ are intriguing sets of the same type with the same parameters for $i=1,~2$ by \cite[Table~1]{KellyCons}. For row 9 or 10 of Table \ref{tab_extfieldP'gt0} with $b=2$, $\kappa$ is orthogonal and $\kappa'$ is unitary. We have $\kappa(v)=\kappa'(v,v)$, and the associated bilinear form of $\kappa$ is $\tr_{\F_{q^2}/\F_q}\circ\kappa'$ up to a scalar. It is routine to deduce that
\[
 \#\{\la\lambda w\ra_{\F_q}\mid \lambda\in\F_{q^2}^*,\la w\ra_{\F_{q^2}}\in O_j', \tr_{\F_{q^2}/\F_q}(\kappa'(\lambda w,v))=0\}=
 q\cdot \#\{\la w\ra_{\F_{q^2}}\in O_j'\mid\kappa'(w,v)=0\}+|O_j'|
\]
for $i,~j\in\{1,2\}$ and $\la v\ra_{\F_q}\in O_i$. It then readily follows that $O_1',~O_2'$ are intriguing sets of the same type and have the same parameters as $O_1,~O_2$. This completes the proof.
\end{proof}

\begin{remark}\label{rem_C3}
Suppose that $b>1$, and define $O_1',~O_2'$ as in \eqref{eqn_C3_Oip}. Let $r,~r'$ be the rank of $\cP$ and $\cP'$ respectively. By Lemma \ref{lem_C3_P'ne0}, Propositions \ref{prop_C3_M1prop}-\ref{prop_C3_OipProper_r'eq1}, it remains to consider the cases where $\cM_1=\cP$, $r'\ge 2$ and $O_1',~O_2'$ are proper subsets of $\cP'$. Since $\cM_1=\cP$, we are in  rows 1, 9 (with $b=2$) and 10 (with $b=2$) of Table \ref{tab_extfieldP'gt0}.
By Proposition \ref{prop_C3_OipProper}, each $O_i'$ is a tight set of $\cP'$ with the same parameter as $O_i$ for $i=1,~2$.
\end{remark}

We assume that $b=1$ for now, and will return to the three remaining $\cC_3$ cases mentioned in Remark \ref{rem_C3} after we handle the Aschbacher classes $\cC_2$, $\cC_4,\ldots,\cC_8$. If $H$ is of class $\cC_8$, then $\cP=W(d-1,q)$ with $q$ even and $H$ stabilizes a nondegenerate quadric $\cQ$. The quadric $\cQ$ is a tight set or an $m$-ovoid of $W(d-1,q)$ according as it is hyperbolic or elliptic. This yields (T4) of Theorem \ref{thm_TS}, and we refer the reader to Section \ref{subsec_grp_lems} for the groups of semisimilarities of $Q^+(d-1,q)$ that are transitive on both $\cQ$ and its complement.

\begin{lemma}\label{lem_C3_allCorNon}
Suppose that  $q$ is a prime power.
\begin{enumerate}
\item[(1)] If $q\ge7$ and $3\le s\le 5$, then there is a vector $u\in\F_q^s$ with all coordinates nonzero such that $\sum_{i=1}^su_i^2=0$.
\item[(2)] If either $q\ge 3$ and $s\in\{2,3,4\}$, or $q=2$ and $s\in\{2,4,6\}$, then there is a vector $u\in\F_{q^2}^s$ with all coordinates nonzero such that $\sum_{i=1}^su_i^{q+1}=0$.
\end{enumerate}
\end{lemma}
\begin{proof}
The proofs of (1) and (2) are similar, so we only give details for (1) here. The quadratic form $Q(x)=x_1^2+x_2^2+x_3^2$ on $\F_q^3$ is nondegenerate, and the corresponding quadric has $q+1$ singular points. Its intersection with each of the three planes defined by $x_i=0$, $1\le i\le 3$, has at most two points respectively. We deduce that there are at least $q+1-6>0$ singular points with all three coordinates nonzero. By a similar argument, we deduce that there are at least $q^2+1-4(q+1)>0$ and $(q^2+1)(q+1)-5(q+1)^2>0$ singular points with all coordinates nonzero for $s=4,~5$ respectively. The claim then follows.
\end{proof}
\begin{proposition}\label{prop_C2}
Suppose that $b=1$ and $H$ stabilizes a subspace decomposition $\cD:\,V=V_1\oplus \cdots\oplus V_t$, where $t\ge 2$ and $\dim(V_i)=m$ for each $i$. Then we have  one of the cases in Table \ref{tab_TS_C2} and (T1) of Theorem \ref{thm_TS}.
\end{proposition}
\begin{proof}
By \cite[Table~4.2.A]{kleidman1990subgroup} we have one of the following cases: 
\begin{enumerate}
\item[(2A)] $V_i$'s are isometric, nondegenerate and pairwise perpendicular, 
\item[(2B)] $t=2$, $qm$ is odd, $\kappa$ is orthogonal, and $V_1,V_2$ are nondegenerate and non-isometric, 
\item[(2C)] $t=2$, and $V_1,~V_2$ are totally singular or isotropic.
\end{enumerate}

We write $N_\Gamma(\cD)$ for the stabilizer of the decomposition $\cD$ in $\Gamma(V)$.  For a vector $x=(x_1,\ldots,x_t)$ with $x_i\in V_i$, we define its $\cD$-length as $\omega_\cD(x):=\#\{1\le i\le t\mid x_i\ne 0\}$. The $\cD$-length is preserved by the action of $N_\Gamma(\cD)$. We define $\Lambda:=\{\omega_\cD(x)\mid \la x\ra\in \cP\}$ and $\cN_i=\{\la x\ra\in\cP\mid \omega_{\cD}(x)=i\}$ for $i\in\Lambda$. For (2B) we have $\{\kappa(x)\mid x\in V_i\setminus\{0\}\}=\F_q$ for $i=1,~2$ by the fact $m\ge 3$, and so the sets $\cN_1$, $\cN_2$ are nonempty. The set $\{\la x_1+x_2\ra\mid x_i\in V_i\setminus\{0\},\kappa(x_1)=\kappa(x_2)=0\}$ is a proper subset of $\cN_2$, so there are at least three $N_\Gamma(\cD)$-orbits. We conclude that (2B) does not occur. For (2C) $\cN_1$, $\cN_2$ are also nonempty. If $\kappa$ is symplectic or unitary, then the set $\{\la x_1+x_2\ra\mid x_i\in V_i\setminus\{0\},\kappa(x_1,x_2)=0\}$ is a proper $N_\Gamma(\cD)$-invariant subset of $\cN_2$. Therefore, $\kappa$ is orthogonal for (2C) and we obtain (T1-b) of Theorem \ref{thm_TS}.

It remains to consider (2A). If $\kappa$ is symplectic, then $\Lambda=\{1,\ldots, t\}$ and we deduce that $t=2$. This leads to (T1-a) of Theorem \ref{thm_TS}. Suppose that $\kappa$ is unitary or orthogonal. We use the same arguments for excluding (2B) to deduce that each $V_i$ has no nonzero singular vector. If $\kappa$ is unitary, then $m=1$ and so $t=d\ge 4$. By Lemma \ref{lem_C3_allCorNon} we have $\{2,3,4\}\subseteq\Lambda$ for $q\ge3^2$ and $\{2,4,6\}\subseteq\Lambda$ for $q=2^2$ with $t\ge 6$. Therefore, we have $q=2^2$ and $t\in\{4,~5\}$ if $\kappa$ is unitary. The two sets $\cN_2,~\cN_4$ are  respectively  $2,~3$-ovoids of $H(3,4)$ for $t=4$ and $6,~27$-tight sets of $H(4,4)$  for $t=5$, and $N_\Gamma(\cD)$ is transitive on them in both cases. If $\kappa$ is orthogonal, then either $m=1$ or each $V_i$ is elliptic of dimension $2$. For the latter case $H$ lies in $\GamL_{t}(q^2)$ which contradicts the assumption $b=1$, so we have $m=1$. We then have $d=t\ge 5$, and so $|\Lambda|\ge 3$ for $q\ge 7$ by Lemma \ref{lem_C3_allCorNon}. There are singular vectors of $\cD$-lengths $2$, $4$, $5$ respectively for $q=5$. For $q=3$ the set $\Lambda$ consists of multiples of $3$ and it has size at most $2$ if and only if $5\le t\le 8$. The group $N_{\Gamma}(\cD)$ is transitive on the singular vectors of each given $\cD$-length in those cases. For $t=5$, the nonzero singular vectors all have $\cD$-length $3$, and we exhaust the subgroups of $\GO_1(3)\wr S_5$ whose orders are multiples of $4$ to obtain $4,~6$-tight sets of $Q(4,3)$ when there are exactly two orbits. For $t=6$, $\cN_3$ and $\cN_6$ are $8,~20$-tight sets of $Q^-(5,3)$ respectively. For $t\in\{7,8\}$, $\cN_3$ and $\cN_6$ are not tight sets.  We summarize those examples in Table \ref{tab_TS_C2}. This conclude the analysis of (2A). This completes the proof.
\end{proof}

\begin{proposition}\label{prop_C2_H}
Take notation as in Proposition \ref{prop_C2}, and assume that $t=2$ and $m\ge 3$. Let $H_{V_1}^{V_1}$ be the restriction of the stabilizer $H_{V_1}$ of $V_1$ in $H$ to $V_1$.   If  $\cP=Q^+(2m-1,q)$ with $V_1,~V_2$ totally singular, then either $\SL_m(q)\unlhd H_{V_1}^{V_1}$ , or $(q,m)=(2,4)$ and $A_7\unlhd H_{V_1}^{V_1}$, or $(q,m)\in\{(2,3),(8,3)\}$ and $H_{V_1}^{V_1}\unlhd\GamL_1(q^3)$.
\end{proposition}
\begin{proof}
Let $B$ be the associated bilinear form of $\kappa$ and $\cN_1$ be the set of singular points in $V_1$ and $V_2$. Let $\cN_2$ be its complement in $\cP$, so that
$\cN_2=\{\la x_1+x_2\ra\mid x_i\in V_i\setminus\{0\},B(x_1,x_2)=0\}$. Let $N_\Gamma(V_1,V_2)$ be the stabilizer of both $V_1$ and $V_2$ in $\Gamma(V)$, which is a subgroup of $N_\Gamma(\cD)$. We define
\begin{equation}
  \iota:\,N_\Gamma(V_1,V_2)\rightarrow \GamL(V_1), \quad g\mapsto g|_{V_1},
\end{equation}
which is the restriction to $V_1$. We have $\ker(\iota)\cong\F_q^*$, and its elements act as identity on $V_1$ and as scalars on $V_2$. Let $H_1$ be the stabilizer of $V_1$ in $H$. Let $Z_1$ be the center of $\GL(V_1)$, and define $G_1=\iota(H_1)Z_1$. The group $G_1$ is transitive on nonzero vectors of $V_1$, and such subgroups have been determined in \cite{TransLinHering}. We follow the statement in \cite[Theorem~2.18]{pls_linear}. Since $m\ge 3$, we have one of the following cases: (H1) $G_1\unlhd\GamL_1(q^m)$, (H2) $\SL_{n}(q_1)\unlhd G_1$ with $q_1^n=q^m$ and $n\ge 2$, (H3) $G_2(q_1)'\unlhd G_1$ with $q_1^6=q^m$ and $q$ even, (H4) $p=3$ and $\SL_2(5)\unlhd G_1\unlhd\GamL_2(9)$, (H5) $G_1=A_6$ or $A_7$ with $q^m=2^4$, (H6) $G_1=\SL_2(13)$ and $q^m=3^6$, (H8) $D_8\circ Q_8\unlhd G_1$ with $q^m=3^4$.

We have $|\cN_2|=(q^{m-1}-1)\frac{q^m-1}{q-1}$. Since $[H:\,H_1]=2$, $H_1$ either is transitive on $\cN_2$ or has two orbits of the same size on $\cN_2$. The latter case occurs only for odd $q$. In particular, $|\cN_2|$ divides $\gcd(2,q-1)\cdot|H_1|$. The group $H_1$ lies in the preimage of $G_1$ in $N_\Gamma(V_1,V_2)$ under $\iota$, so $|H_1|$ divides $(q-1)\cdot |G_1|$. We deduce that $|\cN_2|$ divides $\gcd(2,q-1)\cdot|G_1|$. This excludes (H4), (H6) and (H8) upon direct check. For (H2), $G_1^\infty=\SL_n(q_1)$ lies in $\GL(V_1)^\infty=\SL_m(q)$, so is $\F_q$-linear. Together with the fact $b=1$, we deduce that $q_1=q$, i.e., $G_1^\infty=\SL_m(q)$.  For (H3), we have $q=q_1$ as in the case (H2). We thus have $m=6$, and so $p^{5f}-1$ has a primitive prime divisor $p'$. We similarly have $p'\ge 1+5f$ and $\gcd(p',6f)=1$. The group $H_1$ has order dividing $f\cdot|G_2(q)|=fq^6(q^2-1)(q^6-1)$, while the latter is not divisible by $p'$ since $5f$ does not divide $2f$ or $6f$. This contradiction excludes (H3). For (H5), $(q,m)=(2,4)$, and $H$ has order dividing $n:=2\cdot (q-1)^2\cdot |\Aut(A_s)|$ with $s\in\{6,7\}$. We have $|\cN_2|=3\cdot 5\cdot 7$, so $s=7$ by the fact that $|\cN_2|$ divides $n$. The group $\iota^{-1}(A_7)$ is transitive on $\cN_2$ by direct check. For (H1), $|\cN_2|$ divides $2\cdot |\GamL_1(q^m)|$ only if $m=3$ and $q\in\{2,3,5,8,11\}$. In this case, we take the model $V=\F_{q^3}\oplus\F_{q^3}$, $\kappa((x_1,x_2))=\tr_{\F_{q^3}/\F_q}(x_1x_2)$, and define $\tau:\,(x_1,x_2)\mapsto (x_2,x_1)$. Then $H_1$ is a subgroup of $K_1:=\iota^{-1}\left(\GamL_1(q^3)\right)$ which consists of $(x_1,x_2)\mapsto(ax_1^{p^i},a^{-1}bx_2^{p^i})$, where $a\in\F_{q^3}^*$, $b\in\F_q^*$ and $0\le i\le 3f-1$. The group $H$ is a subgroup of $K:=K_1\rtimes\la\tau\ra$. Take an element $\alpha\in\F_{q^3}^*$ such that $\tr_{\F_{q^3}/\F_q}(\alpha)=0$. The group $K$ is transitive on $\cN_2$ only if $q\in\{2,8\}$ by direct check. This completes the proof.
\end{proof}

\begin{proposition}\label{prop_C4C7}
If $H$ preserves a tensor decomposition of $V$, then we have (T2) of Theorem \ref{thm_TS}.
\end{proposition}

\begin{proof}
We first handle the case where $H$ is of class $\cC_7$. Then $H$ preserve a decomposition $\cD:\,(V,\kappa)=(V_1,f_1)\otimes\cdots\otimes(V_t,f_t)$,
where $V=V_1\otimes\cdots\otimes V_t$ with $t\ge 2$, each $(V_i,f_i)$ is a classical geometry similar to $(V_1,f_1)$ and $\Omega(V_1,f_1)'$ is quasisimple. Let $N_\Gamma(D)$ be the stabilizer of $\cD$ in $\Gamma(V,\kappa)$, and let $\cN_1$ be the set of points $\la v_1\otimes\cdots\otimes v_t\ra$ such that $v_i\in V_i$ for each $i$ and $f_j(v_j,v_j)=0$ for at least one $j$. Then $\cN_1$ is a proper  $N_\Gamma(\cD)$-invariant subset of $\cP$. For an nonzero vector $x=v_1\otimes\cdots\otimes v_t$, we define $s(x):=\#\{1\le j\le t\mid f_j(v_j,v_j)=0\}$. Then the function $s(x)$ is preserved by $N_\Gamma(\cD)$. If $q$ is odd and both $f_1$ and $\kappa$ are symmetric, then $(V_1,f_1)$ has a nonzero singular vector since $\Omega(V_1,f_1)'$ is quasisimple. Then $\{s(x)\mid \la x\ra\in\cN_1\}$ has size at least $2$, so $\cN_1$ comprises of at least two $N_\Gamma(\cD)$-orbits already: a contradiction. The same arguments exclude the case where both $f_1$ and $\kappa$ are unitary. It remains to consider the cases where $f_1$ is symplectic. We have $|\cN_1|=\left(\frac{q^m-1}{q-1}\right)^t$, and $m$ is even. The form $\kappa$ is orthogonal with plus sign if $qt$ is even, and it is symplectic otherwise. In either case, the rank of $\cP$ is $r=\frac{1}{2}m^t$. If $qt$ is odd, then $t\ge 3$ and so $r\ge 4$. If $qt$ is even, then we have $d=m^t\ge 8$ and so $r\ge 4$ by the fact $t\ge 2$ and our hypothesis that $\cP\ne Q^+(3,q)$. If $q^r=2^6$, then we deduce that $(q,r)=(2,6)$ which leads to the contradiction $2r=12=m^t$. By \cite{zsigmondy1892theorie}, $q^r-1$ has a primitive prime divisor $p'$. We have $\textup{ord}_{p'}(p)=fr$, where $q=p^f$ with $p$ prime. Since $p'$ divides $|\cN_1|$, we deduce that $r\le m$, i.e., $\frac{1}{2}m^t\le m$. It follows that $m=t=2$, which contradicts the bound $d=m^t\ge 8$ for $qt$ even. We conclude that $H$ is not of class $\cC_7$.

We then suppose that $H$ is of class $\cC_4$. It preserves a decomposition   $\cD:\,(V,\kappa)=(V_1,f_1)\otimes (V_2,f_2)$, where $V=V_1\otimes V_2$, and $(V_1,f_1)$ is not similar to $(V_2,f_2)$. We write $m_i=\dim(V_i)$ for $i=1,2$. By \cite[Table~4.4.A]{kleidman1990subgroup}, we assume that $m_i\ge 3$ if $f_i$ is symmetric. If both $f_1,f_2$ are symplectic or unitary, then we assume without loss of generality that $m_1<m_2$.  Let $\cN_1$ be the set of points $\la v_1\otimes v_2\ra$ such that $v_i\in V_i$ for $i=1,2$ and $f_j(v_j,v_j)=0$ for at least one $j$, and let $\cN_2$ be its complement in $\cP$. Then $\cN_1$ is a proper $N_\Gamma(\cD)$-invariant subset of $\cP$ in each case. If both of $(V_1,f_1)$, $(V_2,f_2)$ contain singular points and at least one contains a nonsingular point, then $\cN_1$ comprises of at least two $N_\Gamma(\cD)$-orbits as in the proof of the $\cC_7$ case: a contradiction. This excludes all but the case where $f_1$, $f_2$ are both symplectic. In this case, $\kappa$ is orthogonal of plus sign. Choose a basis $\{v_1,\ldots,v_{m_1}\}$ of $V_1$ and a basis $\{e_1,\ldots,e_{m_2/2}\}$ of a maximal totally isotropic subspace of $V_2$. Then $\sum_{i=1}^{k}v_i\otimes e_i$, $1\le k\le\min\{m_1,m_2/2\}$, are all singular vectors. They have different ranks, and so are in distinct $N_\Gamma(\cD)$-orbits. It follows that $\min\{m_1,m_2/2\}\le 2$. Since $m_1,m_2$ are even and $m_1<m_2$, we deduce that $m_1=2$ as desired. Take a basis $v_1$, $v_2$ of $V_1$ such that $f_1(v_1,v_2)=1$. If $\la x\ra\in\cN_2$, then $x=v_1\otimes u_1+v_2\otimes u_2$ for linearly independent vectors $u_1$, $u_2$ in $V_2$. Since $\kappa$ vanishes on $v_1\otimes u_1$ and $v_2\otimes u_2$, we deduce that $\kappa(x)=0$ if and only if $f_2(u_1,u_2)=0$, i.e., $\la u_1,u_2\ra$ is a totally isotropic line of $V_2$. Therefore, $\cN_2$ consists of the points $\la v_1\otimes u_1+v_2\otimes u_2\ra$ with $\la u_1,u_2\ra$ a totally isotropic line of $V_2$. The group $\Delta(V_1,f_1)\circ\Delta(V_2,f_2)$ is transitive on both $\cN_1$ and $\cN_2$. This completes the proof.
\end{proof}

\begin{remark}\label{rem_C4C7}
Suppose that $\kappa$ is orthogonal and  $H$ preserves a tensor decomposition $(V,\kappa)=(V_1,f_1)\otimes (V_2,f_2)$, where $\dim(V_1)=2$, $\dim(V_2)=2m\ge 4$, and $f_1$, $f_2$ are both symplectic. We write $N_\Gamma(\cD)$, $N_\Delta(\cD)$ for the stabilizers of $\cD$ in $\Gamma(V)$, $\Delta(V)$ respectively. We have $N_\Delta(\cD)=\Delta(V_1,f_1)\circ\Delta(V_2,f_2)$. For $i=1,2$, we define
$\pi_i:\,N_\Delta(\cD)\rightarrow\textup{P}\Delta(V_i,f_i)$ such that $\pi_i(g_1\circ g_2)\mapsto\overline{g_i}$, and extend it to a homomorphism $\pi_i:\,N_\Gamma(\cD)\rightarrow\textup{P}\Gamma(V_i,f_i)$ in the natural way. Let $\cN_1$ be the set of points $\la v_1\otimes v_2\ra$ such that $v_i\in V_i\setminus\{0\}$ for $i=1,2$, and let $\cN_2$ be its complement in $\cP$. Then $\cN_1$ and $\cN_2$ are the two $H$-orbits by the proof of Proposition \ref{prop_C4C7}. Since $H$ is transitive on $\cN_1$, we deduce that $\pi_i(H)$ is transitive on the points of $\PG(V_i)$ for $i=1,2$. Since $H$ is transitive on $\cN_2$, we deduce that $\pi_2(H)$ is transitive on the totally isotropic lines of $(V_2,f_2)$. For the subgroups of $\PGamL_2(q)$ that is transitive on the projective lines $\PG(1,q)$, we refer the reader to \cite[Theorem~3.1]{giudici2020subgroups}. Since $\pi_2(H)$ is transitive on both the set of isotropic points and the set of totally isotropic lines of $(V_2,f_2)$, we have either $\Sp_{2m}(q)'\unlhd \pi_2(H)$, or $(m,q)=(2,3)$ by \cite[Theorem~5.2]{giudici2020subgroups}. For the case $(m,q)=(2,3)$, we check by Magma \cite{Magma} that either  $\Sp_4(3)\unlhd \pi_2(H)$ or $2_-^{1+4}.A_5\unlhd \pi_2(H)$.
\end{remark}

\begin{proposition}\label{prop_C5}
If $H$ is of class $\cC_5$, then we have (T3) of Theorem \ref{thm_TS}.
\end{proposition}

\begin{proof}
There is a subfield $\F_{q_0}$ of index $s>1$ in $\F_q$ and an $\F_{q_0}$-span $V_{\#}$ of an $\F_q$-basis $\{e_1,\ldots,e_n\}$ of $V$ such that $H\le N_\Gamma(V_\#)\F_q^*$, where $N_\Gamma(V_\#)$ is the stabilizer of $V_{\#}$ in $\Gamma(V,\kappa)$. We have $q=q_0^s$. Let $\cP_{\#}$ be the associated polar space of $(V_\#,\kappa_\#)$, where $\kappa_\#$ is a proper scalar multiple of the restriction of $\kappa$ to $V_\#$. Let $\cN_1$ be the set of singular or isotropic points of $\cP_\#$, and let $\cN_2$ be its complement in $\cP$. Both $\cN_1$ and $\cN_2$ are proper $N_\Gamma(V_\#)$-invariant subsets of $\cP$. Let $\theta$ be the semilinear transformation of $V$ that maps $\sum_{i}x_ie_i$ to $\sum_{i}x_i^{q_0}e_i$ for $x_i\in \F_q$. For a singular point $P=\la u\ra$ not in $\cP_\#$, write $P^\theta=\la u\theta\ra$, and let $PP^\theta$ be the line that passes through $P$ and $P^\theta$ in $\PG(V)$. We observe that $\la ug\theta\ra=\la u\theta g\ra$ for $g\in N_\Gamma(V_\#)$, so if $Pg=P'$ then $(Pg)^\theta=P'^\theta$.

First consider the case where both $\kappa$ and $\kappa_\#$ are symplectic. Take nonzero vectors $u$, $v$ and $w$ in $V_{\#}$ such that $\kappa(u,w)=0$, $\kappa(v,w)\ne 0$. We define $P_1=\la u+\eta w\ra$, $P_2=\la v+\eta w\ra$ for an element $\eta\in\F_q\setminus\F_\#$. Then $P_i\ne P_i^\theta$ for $i=1,2$, and they are not in $\cN_1$. The line $\la u ,w\ra_{\F_q}$ intersects $V_\#$ in a totally singular subline $\la u, w\ra_{\F_\#}$, and it contains $P_1$, $P_1^\theta$. The line $\la v ,w\ra_{\F_q}$ intersects $V_\#$ in a subline $\la v, w\ra_{\F_\#}$ that is not totally singular, and it contains $P_2$, $P_2^\theta$. The points $P_1,P_2$ are not in the same $N_\Gamma(V_\#)$-orbit, since otherwise $\la u,w\ra_{\F_\#}$ and $\la v,w\ra_{\F_\#}$ would be in the same $N_\Gamma(V_\#)$-orbit: a contradiction. Therefore, this case is impossible.

We next consider the case where both $\kappa$ and $\kappa_\#$ are both orthogonal. Let $\ell$ be a totally singular line of $V$ that intersects $V_\#$ in a subline, and take a point $P_1$ outside $\cN_1$ on $\ell$. Then $P_1\ne P_1^\theta$, $\ell$ is stabilized by $\theta$ and $\ell=P_1P_1^\theta$. Take a subline $\la u,v\ra_{\F_\#}$ of $V_\#$ that contains no singular point, and set $\ell'=\la u,v\ra_{\F_q}$. The quadratic polynomial $\kappa(u+xv)=0$ in $x$ has no solution in $\F_\#$, so is irreducible over $\F_{q_0}$. If it has a solution $x_0$ in $\F_q$, then $P_2=\la u+x_0v\ra$ is a singular point such that $P_2P_2^\theta=\ell'$. We then deduce that $P_1,P_2$ are not in the same $N_\Gamma(V_\#)$-orbits as in the symplectic case. Therefore, $\kappa(u+xv)=0$ has no solution in $\F_q$, and so $s$ is odd. It follows that $\cP$ and $\cP_\#$ has the same rank $r$ and $sr\ge 6$. Then $\frac{q^r-1}{q-1}$ divides $|\cN_1|=(q_0^{d-1-r}+1)\frac{q_0^r-1}{q_0-1}$. If $q^r=q_0^{sr}=2^6$, then $(q_0,s,r)=(2,3,2)$ and $\cP$ is one of $Q(4,2)$, $Q^-(5,2)$.  If $\cP_{\#}=Q(4,2)$, then $|\cN_1|=15$ is not divisible by $\frac{q^r-1}{q-1}=9$; if $\cP_{\#}=Q^-(5,2)$, then $|\cN_2|=2\cdot 3^4\cdot 7$ does not divide $|N_\Gamma(V_\#)|=2^7\cdot 3^5\cdot 5$. Therefore, we have $q^r\ne 2^6$.
Let $p'$ be a primitive prime divisor of $q^r-1$. Since $\textup{ord}_{p'}(q_0)=sr$, we deduce from $p'$ divides $|\cN_1|$ that $sr$ divides $2(d-1-r)$. Since $d\le 2r+2$, we have $2(d-1-r)\le 2r+2\le 3r$. Since $s$ is odd and $sr\le 3r$, we deduce that $s=3$, $r=2$ and $d=2r+2=6$. In this case, $\cP=Q^-(5,q_0^3)$, $|\cN_2|=q_0(q_0^3+1)(q_0^8-1)$, and a primitive prime divisor of $q_0^8-1$ does not divide the order of $N_\Gamma(V_\#)$ upon direct check: a contradiction.  To sum up, this case is not possible.

We then consider the cases where $\kappa$ is unitary and $\kappa_\#$ is orthogonal or unitary. We observe that each line of $\PG(V)$ contains either $1$, $q^{1/2}+1$ or $1+q$ singular points of $\kappa$, and $q^{1/2}+1>\max\{2,1+q_0^{1/2}\}$. Take a totally singular lines $\ell$ that intersects $V_\#$ in a subline, and take a point $P_1$ on it outside $\cP_\#$. Take a line $\ell'$ that intersects $V_\#$ in $2$ or $q_0^{1/2}+1$ singular points respectively according as $\kappa_\#$ is orthogonal or unitary, and take a singular point $P_2$ on it outside $\cP_\#$. Then $P_i\ne P_i^\theta$ for $i=1,2$, and the lines $\ell$, $\ell'$ are stabilized by $\theta$. It follows that $\ell=P_1P_1^\theta$, $\ell'=P_2P_2^\theta$. We deduce that $P_1,P_2$ are in distinct $N_\Gamma(\cD)$-orbits as in the previous cases. We conclude that those two cases are not possible.

Finally, we consider the case where $\kappa$ is unitary and $\kappa_\#$ is symplectic. In this case, we have $q=q_0^2$. By \cite[p.~991]{C5USp_CK}, $\cN_1$ and $\cN_2$ are the two $\Sp_{d}(q_0)$-orbits.  Assume that $H$ does not contain $\Sp_{d}(q_0)'$ from now on. For each $P\in\cN_2$ the line $PP^\theta$ intersects $V_\#$ in a totally singular subline.  We deduce that $H$ is transitive on both the set of points and the set of totally isotropic lines of $\cP_\#$.  By \cite[Theorem~5.2]{giudici2020subgroups} we deduce that the totally isotropic lines are maximal, so $d=4$. Moreover, either $q$ is even and $\Omega_{3}(q_0)\unlhd H$, or $q_0=3$ by the same theorem. In the former case $H$ is not transitive on $\cN_1$: a contradiction. For $q_0=3$, we  examine all the subgroups of $N_\Gamma(V_\#)$ whose orders are multiples of $\lcm(|\cN_1|,|\cN_2|)=240$ and there is no such subgroup that is transitive on both $\cN_1$ and $\cN_2$ besides those that contains $\Sp_4(3)$.  This completes the proof.
\end{proof}

To consider the subgroups of class $\cC_6$, we need the following technical lemma.
\begin{lemma}\label{lem_C6_ppd}
Suppose that $n>0$ and $p$ is an odd prime. If $p'$ is an odd prime divisor of $p^{2^n}+1$, then $p'$ is a primitive prime divisor of $p^{2^{n+1}}-1$.
\end{lemma}
\begin{proof}
Suppose to the contrary that $p'$ is an odd prime divisor of $p^{2^n}+1$ that is not a primitive prime divisor of $p^{2^{n+1}}-1$. Then $m:=\textup{ord}_{p'}(p)$ is a proper divisor of $2^{n+1}$, i.e., $m=2^s$ with $s\le n$. We have $0\equiv p^{2^n}+1\equiv 1^{2^{n-s}}+1\equiv 2\pmod{p'}$ which is a contradiction. This completes the proof.
\end{proof}

\begin{proposition}\label{prop_C6}
Let $r_0$ be a prime distinct from $p$. If $H$ normalizes an absolutely irreducible $r_0$-group $R$ of symplectic type and $H$ is not contained in a subgroup of class $\cC_5$, then we have one of the cases in Table \ref{tab_TS_C6}.
\end{proposition}
\begin{proof}
We write $N_\Gamma(R)$ and $N_\Delta(R)$ for the normalizer of $R$ in $\Gamma(V)$ and $\Delta(V)$ respectively.  Let $e$ be the smallest integer for which $p^e\equiv 1\pmod{|Z(R)|}$. Then $q=p^e$ and $e$ satisfies the restrictions in \cite[Table~4.6.B]{kleidman1990subgroup}, since $N_\Gamma(R)$ does not lie in a subgroup of class $\cC_5$. The case where $R\cong 2_-^{1+2m}$ and $\kappa$ is symplectic has been handled in \cite[Case~(IIe)]{LiebeckRank3}, and we have one of the symplectic cases in Table \ref{tab_TS_C6} by (B) of the main theorem therein.

Suppose that $R\cong r_0^{1+2m}$ ($r_0$ odd) and $\kappa$ is unitary. In this case, $d=r_0^m$, $e$ is even, $|Z(R)|=r_0$, $N_\Delta(R)=N_{\textup{GL}(V)}(R)=\F_q^*\circ (r_0^{1+2m}.\textup{Sp}_{2m}(r_0))$. The rank of $\cP=H(d-1,q)$ is $r=\frac{r_0^m-1}{2}$. If $p^{er}=2^6$, then $p=e=2$, $r=3$. It follows that $r_0^m=7$, i.e., $r_0=7$ and $m=1$. However, $p^e\equiv 1\pmod{|Z(R)|}$ does not hold: a contradiction.  Let $p'$ be a primitive prime divisor $p'$ of $p^{er}-1$, so that $\textup{ord}_{p'}(p)=er$. We have $\gcd(p',er)=1$, and $p'\ge 1+\frac{e}{2}(r_0^m-1)\ge r_0^m$. Since $\gcd(p',r)=1$ and $r_0$ is odd, we deduce that $p'\ge r_0^m+2$. Since $\frac{q^r-1}{q-1}$ divides $|H|$, we deduce that $p'$ divides $|N_\Gamma(R)|$. It follows that $p'$ divides $\prod_{i=1}^{m}(r_0^{2i}-1)$. The prime numbers of the latter number do not exceed $\frac{1}{2}(r_0^m+1)$: a contradiction.

Suppose that $R\cong 4\circ 2^{1+2m}$ and $\kappa$ is unitary. In this case, $d=2^m$, $r_0=e=2$, $|Z(R)|=4$, $N_\Delta(R)=N_{\textup{GL}(V)}(R)=\F_q^*\circ (2^{1+2m}.\textup{Sp}_{2m}(2))$. Since $e=2$, we deduce that $p\equiv 3\pmod{4}$.  The rank of $\cP$ is $r=2^{m-1}\ge 2$, so $m\ge 2$. Let $p'$ be a primitive prime divisor of $p^{er}-1$, so that $\textup{ord}_{p'}(p)=2^m$. We have  $\gcd(p',er)=1$ and $p'\ge 1+2^m$.  Let $s$ be the largest integer such that $p'^s$ divides $p^{er}-1$. We argue as in the previous case to deduce that $p'^s$ divides $x:=\prod_{i=1}^{m}(2^{2i}-1)$.  By examining the prime divisors of $x$ and using the fact $p'\ge 1+2^m$, we deduce that $p'=1+2^m$ and $s=1$. Since $er=2^{m}$ and $p^{2^{m-1}}+1\equiv 2\pmod{4}$, we deduce from Lemma \ref{lem_C6_ppd} that $p^{2^{m-1}}+1=2(1+2^m)$. It holds only if $m=2$, $p=3$.  We check by Magma \cite{Magma} that there are two $N_\Gamma(R)$-orbits of sizes $120,~160$ respectively on $\cP=H(3,9)$. Their sizes are not divisible by the ovoid number $28$, so the two orbits are both tight sets.

Finally, suppose that $R\cong 2_+^{1+2m}$ and $\kappa$ is hyperbolic of plus sign. In this case, $d=2^m$, $r_0=|Z(R)|=2$, $e=1$, and $N_\Delta(R)=N_{\textup{GL}(V)}(R)=\F_q^*\circ (2_+^{1+2m}.\GO^+_{2m}(2))$.
Since we do not consider $Q^+(3,q)$, we have $m\ge 3$. The rank of $\cP$ is $r=2^{m-1}\ge 4$. Let $p'$ be an odd prime divisor of $p^{r/2}+1$, and let $s$ be the largest integer such that $p'^s$ divides $p^{r/2}+1$. Then $p'$ is a primitive prime divisor of $p^r-1$ by Lemma \ref{lem_C6_ppd}, and so $p'\ge 1+2^{m-1}$. We deduce from $\frac{p^r-1}{p-1}$ divides $N_\Gamma(R)$ that $p'^s$ divides
$y=(2^m-1)\prod_{i=1}^{m-1}(2^{2i}-1)$.  By examining the prime divisors of $y$ and using the fact $p'\ge 1+2^{m-1}$ that $p'=1+2^{m-1}$, $s=1$. It holds for each odd prime divisor $p'$ of $p^{r/2}+1$ and $2$ strictly divides $p^{r/2}+1$, so $p^{2^{m-2}}+1=2(1+2^{m-1})$.  It holds if and only if $(p,m)=(3,3)$. We go over the subgroups of $N_\Gamma(R)$ whose orders are multiple of $\frac{q^r-1}{q-1}=40$ by Magma \cite{Magma} that the possible sizes for $O_1,~O_2$ are $\{280,840\}$, $\{480,640\}$ and $\{160,960\}$. They all yield tight sets upon direct check. This completes the proof.
\end{proof}

We have completed the analysis of the geometric Aschbacher class $\cC_2$ and classes $\cC_4$ to $\cC_8$, and we are now ready to handle the three remaining cases in class $\cC_3$, cf. Remark \ref{rem_C3}. Recall that $H$ has two orbits $O_1,~O_2$ on the points of $\cP$ which are $i_1$- and $i_2$-tight sets respectively, and the corresponding subsets $O_1',~O_2'$ as in \eqref{eqn_C3_Oip} are nonempty and form the two $H$-orbits on $\cP'$, and $\cP'$ has rank at least $2$. The parameters of $O_1$, $O_2$ are the same as $O_1'$, $O_2'$ by Proposition \ref{prop_C3_OipProper}. Since $H$ is irreducible on $V$, it is irreducible on $V'$.  It is not of class $\cC_3$ as a subgroup of $\Gamma(V',\kappa')$ by the choice of $b$.

\begin{proposition}\label{prop_C3_remain}
Suppose that $\cM_1=\cP$ and $b>1$, and write $r'$ for the rank of $\cP'$. Let $O_1'$, $O_2'$ be as in \eqref{eqn_C3_Oip}, and assume that they are both proper subsets of $\cP'$. If $r'\ge 2$, then we have (T1-a) of Theorem \ref{thm_TS} or one of the cases in the third block of Table \ref{tab_C3Remain}.
\end{proposition}
\begin{proof}
Suppose that we are in row 1 of Table \ref{tab_extfieldP'gt0}. In this case, both $\kappa$ and $\kappa'$ are symplectic. First consider the case where $H$ is a geometric subgroup of $\Gamma(V',\kappa')$.  We go over the arguments for this section to see that: there are no feasible subgroups of classes $\cC_4-\cC_7$  with $q^b$ not a prime, and we obtain (T1-a) of Theorem \ref{thm_TS} for class $\cC_2$. If $H$ is of class $\cC_8$, then $q$ is even and there is a quadratic form $Q'$ of plus sign on $V'$ that polarizes to $\kappa'$. The two $H$-orbits are thus $O_1=\{\la v\ra_{\F_q}\mid Q'(v)=0,\,v\ne 0\}$, $O_2=\{\la v\ra_{\F_q}\mid Q'(v)\ne0\}$. Take a vector $v$ such that $Q'(v)=1$. Let $K$ be the intersection of $\Gamma^\#$ and  the group of semisimilarities of $(V',\,Q')$. Since $H$ lies in $K$ and $K$ stabilizes both $O_1$ and $O_2$, we deduce that $O_2$ is the $K$-orbit of $\la v\ra_{\F_q}$. The latter is equivalent to $\F_q^*=\F_{q^b}^*$ by considering the $Q'$-values for both sets. It follows that $b=1$: a contradiction. Next consider the case where $H$ is a non-geometric subgroup of $\Gamma(V',\kappa')$. By Proposition \ref{prop_classS_caseS} and the fact $q^b$ is not a prime, we see that $W(d/b-1,q^b)=W(5,4)$, $H=J_2$, and the two orbits are $25,~40$-tight sets respectively. We check that the group $J_2$ has two orbits on the nonzero vectors of $V'$ by using the Atlas data \cite{Atlas}. Since $4=2^2$, we deduce that $b=2$, $\cP=W(11,2)$. The two $J_2$-orbits $O_1,~O_2$ on the points of $\cP$ are $25,~40$-tight sets respectively.

Suppose that we are in rows 9 or 10 of Table \ref{tab_extfieldP'gt0} with $b=2$. We have $\cP'=H(d/2-1,q^2)$, $\cP=Q^{\pm}(d-1,q^2)$, and $\kappa(x)=\kappa'(x,x)$. First consider the case where $H$ is a geometric subgroup of $\Gamma(V',\kappa')$. We go over the arguments in this section to obtain $6,~27$-tight sets of $Q^-(9,2)$ from those of $H(4,4)$ for class $\cC_2$, $\Sp_{2m}(q)'$-invariant $(q+1)$ and $(q^{2m-1}+1)$-tight sets of $Q^+(4m-1,q)$ with $d=4m$ for class $\cC_5$, $N_\Gamma(R)$-invariant $12,~16$-tight sets of $Q^+(7,3)$ from those of $H(3,9)$ with $R=4\circ 2^{1+4}$ for class $\cC_6$, and there are no further instances. For the $\cC_5$ case, we remark that the group $\Sp_{d/2}(q)$ has exactly two orbits on the nonzero singular vectors of the unitary form $\kappa'$ by \cite[p.~991]{C5USp_CK}. We tabulate those examples in the third block of  Table \ref{tab_C3Remain} and indicate their Aschbacher classes as a subgroup of $\Gamma(V',\kappa')$ in the Remark column. Next consider the case where $H$ is a non-geometric subgroup of $\Gamma(V',\kappa')$. We go over Section \ref{sec_classS} to obtain the rows with an $\cS$ entry in the Remark column of Table \ref{tab_C3Remain}. This completes the proof.
\end{proof}

Let us summarize what we have done so far. Let $\cP$ be a classical polar space of rank $r\ge 2$ over $\F_q$ such that $\cP\ne Q^+(3,q)$, and write $(V,\kappa)$ for the ambient space and form. Suppose that $H$ is an irreducible group of semisimilarities that has two orbits on the singular points that are tight sets of $\cP$. The group $\Omega(V,\kappa)'$ is transitive on the singular points, so $H$ does not contain $\Omega(V,\kappa)'$. By \cite{asch_max}, $H$ is either of Aschbacher class $\cS$ or is contained in a maximal subgroup of $\Gamma(V,\kappa)$ of classes $\cC_2,\ldots,\cC_8$. In Section \ref{sec_classS} we have classified such subgroups of Aschbacher class $\cS$. In Propositions \ref{prop_C3_M1prop}, \ref{prop_C3_O1peqPp}, \ref{prop_C3_OipProper_r'eq1} and \ref{prop_C3_remain}, we have classified such subgroups of Aschbacher class $\cC_3$. If $H$ is of class $\cC_8$, then we have (T4) of Theorem \ref{thm_TS}. In Propositions \ref{prop_C2}, 5.9, 5.10, 5.12, and \ref{prop_C6} we have classified such subgroups of the remaining geometric Aschbacher classes. This concludes the proof of Theorem \ref{thm_TS}.

\section{Proof of Theorem \ref{thm_mOvoid}}

We have classified the $m$-ovoids of finite classical polar spaces that admit a transitive automorphism group of semisimilarities acting irreducibly on the ambient space in \cite{FengLiTao}. In this section, we prove Theorem  \ref{thm_mOvoid} by making use of such classification results. Suppose that $H$ is an irreducible subgroup of $\Gamma(V,\kappa)$ that has exactly two orbits on the singular points of $\cP$ which are respectively $m_1,~m_2$-ovoids. We check by using Magma \cite{Magma} that: for $\cP=Q^-(5,2)$ there is no such irreducible subgroup $H$; for $\cP=H(3,4)$ the subgroup $H$ is of class $\cC_2$ and appears in Table \ref{tab_TS_C2}; for $\cP=W(5,2)$ either $H$ stabilizes an elliptic quadric $Q^-(5,2)$ and covered by Theorem \ref{thm_mOvoid} (M2), or it appears in row 5 of Table \ref{tab_TrOvoid_C3}. Similarly, for $\cP=Q^+(7,2)$ either $H^\infty\in\{A_9,\PSL_2(8)\}$ and one orbit is an ovoid, or $H\le\GO_2^-(2)\wr S_4$ and $\{m_1,m_2\}=\{6,9\}$. The two cases appear in Table \ref{tab_TrOvoid_cS} and Table \ref{tab_TrOvoid_C3Rem} respectively. We suppose that $\cP$ is none of those four polar spaces from now on.

Let $b$ be the largest integer such that $H$ lies in $\GamL_{d/b}(q^b)$. We first consider the case where $b>1$. There is a vector space $V'$ of dimension $m:=d/b$ over $\F_{q^b}$ equipped with a nondegenerate form $\kappa'$ such that $(V',\kappa')$ is associated with $(V,\kappa)$ as in Table \ref{tab_extfieldP'gt0}. Then $H$ is a subgroup of $\Gamma^\#$, cf. \eqref{eqn_GamJ}. We define $\cM_1$ as in \eqref{eqn_M1def} which corresponds to the nonzero singular points of $\cP'$, and let $\cM_2$ be its complement in $\cP$. We list the cases where both  $\cM_1,~\cM_2$ are nonempty transitive $m$-ovoids in Table \ref{tab_TrOvoid_C3} by \cite[Examples~3.3,~3.14]{FengLiTao}. There are no further such cases by the main theorems therein. We now derive the last column on $H$ of Table \ref{tab_TrOvoid_C3}.
\begin{enumerate}
\item[(1)] If $\cP=H(3m-1,q)$ and $\cP'=H(m-1,q^3)$, then $b=3$, $m\ge 3$ is odd and $q\in\{2^2,2^6\}$. The two $H$-orbits are $\cM_1$, $\cM_2$, so we deduce that $H$ is a subgroup of $\Gamma(V',\kappa')$ that is transitive on both the set of singular points and the set of nonsingular points of $V'$. Since $q^3$ is a cube and $q$ is even, we deduce that $\SU_{m}(q^{3/2})\unlhd H$ by Lemma \ref{lem_H_TrSinNonsin}.
\item[(2)] If $\cP=Q^-(2m-1,q)$ and $\cP'=Q^-(m-1,q^2)$, then $b=2$ and $m\ge 4$ is even. The two $H$-orbits are $\cM_1$, $\cM_2$, so we deduce that $H$ is a subgroup of $\Gamma(V',\kappa')$ that is transitive on both the set of singular points and one class of nonsingular points of $V'$. Since $q^2$ is a square, by Lemma \ref{lem_Q-(2m,q)sub} we deduce that either $\Omega_{m}^-(q^2)\unlhd H$, or $m/2$ is odd and $\SU_{m/2}(q^4)\unlhd H$, or $m=q=4$ and $\GO_2^-(4^2).4\unlhd H$. The last case does not occur, since otherwise $b=4$ by the definition of $b$: a contradiction to $b=2$.
\item[(3)] If $\cP=Q^-(3m-1,q)$ and $\cP'=Q^-(m-1,q^3)$, then $b=3$, $q\in\{2,2^3\}$ and $m\ge 4$ is even. The two $H$-orbits are $\cM_1$, $\cM_2$, so we deduce that $H$ is a subgroup of $\Gamma(V',\kappa')$ that is transitive on both the set of singular points and the set of nonsingular points of $V'$. Since $q^3$ is a cube, by Lemma  \ref{lem_Q-(2m,q)sub} we deduce that either $\Omega_{m}^-(q^3)\unlhd H$, or $m/2$ is odd and $\SU_{m/2}(q^3)\unlhd H$.
\item[(4)] If $\cP=Q^+(2m-1,q)$  and $\cP'=Q(m-1,q^2)$, then $b=2$, $qm$ is odd and $m\ge 3$. The two $H$-orbits are $\cM_1$, $\cM_2$, so we deduce that $H$ is a subgroup of $\Gamma(V',\kappa')$ that is transitive on both the set of singular points and the set of nonsingular points of $V'$. Moreover, the latter set has size $\frac{1}{2}|\cM_2|=\frac{1}{2}q^{m-1}(q^{m-1}+1)$. By Lemma \ref{lem_Q(dodd,q)sub} we deduce that either $\Omega_{m}(q^2)\unlhd H$, or $m=7$ and $G_2(q^2)\unlhd H$.
\item[(5)] If $\cP=W(2m-1,q)$ and $\cP'=H(m-1,q^2)$, then $b=2$ and $m\ge 3$ is odd. The two $H$-orbits are $\cM_1$ and $\cM_2$, so we deduce that $H$ is transitive on both singular and nonsingular points of $\cP'$. By Lemma \ref{lem_H_TrSinNonsin}, we have either $\SU_{m}(q)\unlhd H$, or $(q,m)=(2,3)$ and $H=3_+^{1+2}\rtimes C_8$.
\end{enumerate}
It remains to consider the cases where one of $\cM_1$, $\cM_2$ is empty. There are no such instances for $H(d-1,q)$ or $Q(d-1,q)$ by \cite[Theorems~16,~17]{FengLiTao}. For $\cP=Q^-(d-1,q)$, there is a transitive $3$-ovid of $Q^-(5,5)$ with $H^\infty=3.A_7$ and $\cP'=H(2,5^2)$ by \cite[Theorem~3.18]{FengLiTao}. Since $3+3=\frac{5^r-1}{5-1}$ with $r=2$, it is a candidate case. There are exactly two $3.A_7$-orbits of the same sizes on the nonsingular vectors of $H(2,5^2)$, so we have row 1 of Table \ref{tab_TrOvoid_C3Rem}. For $\cP=W(d-1,q)$, we have row 3 of Table \ref{tab_TrOvoid_C3Rem} from \cite[Example~3.13]{FengLiTao}. We consider case (b) in  \cite[Theorem~3.20]{FengLiTao}, where $\cP=W(3,q)$ and $\cP'=H(0,q^4)$. We take the following model for $(V,\kappa)$: $V=\F_{q^4}$, $\kappa(x,y)=\tr_{\F_{q^4}/\F_q}(\delta xy)$ with $\delta^q+\delta=0$. For $a\in\F_{q^4}^*$ and $0\le i\le 4f-1$, define $\psi_{a,i}\mid x\mapsto ax^{p^i}$. Then $H_0$ is a subgroup of $M=\{\psi_{a,i}\mid a^{q^2+1}\in\F_q^*,\,0\le i\le 4f-1\}$, where $|M|=4f(q^2+1)(q-1)$. An $M$-orbit is $O_1=\{\la x\ra\mid x^{q^2+1}\in\F_q^*\}$, so the $H$-orbits should be $O_1$ and its complement $O_2$. We have $|O_1|=q^2+1$, $|O_2|=q(q^2+1)$, and $\lcm(|O_1|,|O_2|)$ divides $4f(q^2+1)$ if and only if $q$ is one of $2,4,16$. It is straightforward to check that $M$ is transitive on $O_2$ if and only if $q=2,4$. This yields row 2 of Table \ref{tab_TrOvoid_C3Rem}. By \cite[Theorem~3.20]{FengLiTao} there are no further instances where $\cM_1$ or $\cM_2$ is empty and there are transitive $m_1,~m_2$-ovoids with $m_1+m_2=\frac{q^r-1}{q-1}$. This completes the analysis of the case $b>1$.

We next consider the case where $b=1$. Then $H$ is not of class $\cC_3$ as a subgroup of $\Gamma(V,\kappa)$ by the choice of $b$. We now go over \cite[Theorems~16-20]{FengLiTao} for the subgroups of each Aschbacher class in $\Gamma(V,\kappa)$ except for $\cC_3$. If $H$ is of class $\cS$, then we have one of the cases in Table \ref{tab_TrOvoid_cS}. If $H$ is of class $\cC_2$, i.e., $H$ preserves a decomposition $V=V_1\oplus\ldots\oplus V_t$ with $t\ge 2$ and $m=\dim(V_i)$ for each $i$, then $m=1$ and $\cP$ is one of  $Q(4,3)$, $Q(6,3)$ and $Q^+(7,3)$. This leads to the examples in Table \ref{tab_TS_C2} by examining all the irreducible subgroups whose orders are divisible by the respective ovoid numbers using Magma. We remark that the $Q^+(7,2)$ row of Table \ref{tab_TrOvoid_C3Rem} also falls in this class with each $V_i$ elliptic of dimension $2$. There are no candidates for $H$ in the classes $\cC_4$, $\cC_5$ and $\cC_7$ by the main theorems in \cite{FengLiTao}. If $H$ is of class $\cC_6$, then we have the example in Theorem \ref{thm_mOvoid} (M1) from \cite[Example~3.4]{FengLiTao} and there are no further cases.  If $H$ is of class $\cC_8$, then we have the example in Theorem \ref{thm_mOvoid} (M2) from \cite[Example~3.1]{FengLiTao}. This completes the proof of Theorem \ref{thm_mOvoid}.

\vspace{0.1in} 

\noindent{\bf Acknowledgements.} The research of Tao Feng is partially supported by National Natural Science Foundation of China Grant No. 12225110, 12171428, and the research work of Qing Xiang is partially supported by the National Natural Science Foundation of China Grant No. 12071206, 12131011, 12150710510, and the Sino-German Mobility Programme M-0157.


\end{document}